\pdfoutput=1
\documentclass[a4paper,reqno]{amsart}
\usepackage[foot]{amsaddr}
\usepackage[margin=3cm]{geometry}

\newcommand{\ifarticle}[2]{
    \csname@ifclassloaded\endcsname{beamer}{#2}{#1}
}

\newcommand{\ifbook}[2]{
    \csname@ifclassloaded\endcsname{amsbook}{#1}{#2}
}

    \usepackage[T1]{fontenc} 
    \usepackage{xparse} 
    \usepackage{xspace} 
    \usepackage{etoolbox} 
    \usepackage{xpatch} 
    \usepackage{pgffor} 

    \usepackage{amsmath,amssymb} 
    \usepackage{keytheorems} 
    \allowdisplaybreaks[1]

    \usepackage{mathtools} 
    \usepackage{mathrsfs} 
    \usepackage{stmaryrd} 
    \usepackage{bbm} 
    \usepackage{quiver} 
    \usepackage{xcolor} 
    \usepackage{scalerel} 
    \usepackage{booktabs} 
    \usepackage{nameref} 
    \usepackage[british]{babel} 
    \usepackage{csquotes} 
    \usepackage[UKenglish]{isodate} 
    \usepackage{microtype} 
    \usepackage[splitrule]{footmisc} 

    \ifarticle{
        \PassOptionsToPackage{hyphens}{url}
        \usepackage[pdfusetitle,linktoc=all,colorlinks,citecolor=cyan,linkcolor=purple,urlcolor=blue]{hyperref} 
        \urlstyle{rm}

        \usepackage{zref-clever} 
        \zcsetup{nameinlink,noabbrev,cap=true}
        \newcommand{\cref}{\zcref}
        \newcommand{\Cref}{\zcref[cap=true]}
        \usepackage{footnotebackref} 

        \usepackage[shortlabels]{enumitem} 
        \setlist{topsep=2pt,itemsep=2pt,partopsep=2pt,parsep=2pt} 

        \setcounter{tocdepth}{1}
    }{}
    \usepackage[style=alphabetic,abbreviate=false,backref=true,backrefstyle=three,maxnames=9,minalphanames=3,maxalphanames=4]{biblatex} 
    \DeclareUnicodeCharacter{0301}{\TODO[Invalid symbol.]} 

    \DeclareFieldFormat*{citetitle}{\emph{#1}} 
    \DeclareCiteCommand{\citetitle}
        {\boolfalse{citetracker}%
        \boolfalse{pagetracker}%
        \usebibmacro{prenote}}
        {\ifciteindex
            {\indexfield{indextitle}}
            {}%
        \printtext[bibhyperref]{\printfield[citetitle]{labeltitle}}}
        {\multicitedelim}
        {\usebibmacro{postnote}}
    \DeclareCiteCommand*{\citeauthor}
        {\defcounter{maxnames}{99}%
        \defcounter{minnames}{99}%
        \defcounter{uniquename}{2}%
        \boolfalse{citetracker}%
        \boolfalse{pagetracker}%
        \usebibmacro{prenote}}
        {\ifciteindex{\indexnames{labelname}}{}%
        \printnames{labelname}}
        {\multicitedelim}
        {\usebibmacro{postnote}}

    \usepackage{calc} 
    \usepackage{tabularray} 
    \usepackage{ebproof} 
    \usepackage{forest} 
    \usepackage{adjustbox} 
    \usepackage{float} 
    \usepackage{stfloats}

    \makeatletter
    \ifarticle{
        \makeatletter\def\@makefntext{\indent\@makefnmark}\makeatletter

        \xpretocmd{\@adminfootnotes}{\let\@makefntext\BHFN@OldMakefntext}{}{}
        \xpatchcmd{\@maketitle}{\let\@makefnmark\relax}{\let\@makefnmark\no@makefnmark}{}{}
        \def\no@makefnmark{}
        \renewcommand\@makefntext[1]{%
          \ifx\@makefnmark\no@makefnmark
            \BHFN@OldMakefntext{#1}%
          \else
            \renewcommand\@makefnmark{%
            \mbox{%
                \textsuperscript{%
                \normalfont
                \hyperref[\BackrefFootnoteTag]{\@thefnmark}%
                }%
            }\,%
            }%
            \BHFN@OldMakefntext{#1}%
          \fi
        }

        \LetLtxMacro{\BHFN@Old@footnotemark}{\@footnotemark}
        \renewcommand*{\@footnotemark}{%
            \refstepcounter{BackrefHyperFootnoteCounter}%
            \xdef\BackrefFootnoteTag{bhfn:\theBackrefHyperFootnoteCounter}%
            \label{\BackrefFootnoteTag}%
            \BHFN@Old@footnotemark
        }

        \makeatletter
        \def\paragraph{\@startsection{paragraph}{4}%
          \z@\z@{-\fontdimen2\font}%
          {\normalfont\bfseries}}
        \makeatother
    }{}
    \makeatother

    \ifarticle{
        \theoremstyle{plain}
        \ifbook{
            \newkeytheorem{theorem}[name=Theorem, parent=chapter, refname={theorem, theorems}, Refname={Theorem, Theorems}]
        }{
            \newkeytheorem{theorem}[name=Theorem, parent=section, refname={theorem, theorems}, Refname={Theorem, Theorems}]
        }

        \newkeytheorem{theorem*}[name=Theorem, numbered=false]

        \newkeytheorem{corollary}[name=Corollary, sibling=theorem, refname={corollary, corollaries}, Refname={Corollary, Corollaries}]
        \newkeytheorem{lemma}[name=Lemma, sibling=theorem, refname={lemma, lemmas}, Refname={Lemma, Lemmas}]
        \newkeytheorem{proposition}[name=Proposition, sibling=theorem, refname={proposition, propositions}, Refname={Proposition, Propositions}]

        \theoremstyle{definition}

        \newkeytheorem{axiom}[name=Axiom, sibling=theorem, qed=$\lrcorner$, refname={axiom, axioms}, Refname={Axiom, Axioms}]
        \newkeytheorem{conjecture}[name=Conjecture, sibling=theorem, qed=$\lrcorner$, refname={conjecture, conjectures}, Refname={Conjecture, Conjectures}]
        \newkeytheorem{construction}[name=Construction, sibling=theorem, qed=$\lrcorner$, refname={construction, constructions}, Refname={Construction, Constructions}]
        \newkeytheorem{definition}[name=Definition, sibling=theorem, qed=$\lrcorner$, refname={definition, definitions}, Refname={Definition, Definitions}]
        \newkeytheorem{example}[name=Example, sibling=theorem, qed=$\lrcorner$, refname={example, examples}, Refname={Example, Examples}]
        \newkeytheorem{notation}[name=Notation, sibling=theorem, qed=$\lrcorner$]
        \newkeytheorem{note}[name=Note, sibling=theorem, qed=$\lrcorner$, refname={note, notes}, Refname={Note, Notes}]
        \newkeytheorem{remark}[name=Remark, sibling=theorem, qed=$\lrcorner$, refname={remark, remarks}, Refname={Remark, Remarks}]

        \newenvironment{sketch}{\proof}{\endproof}

        \AtBeginEnvironment{axiom}{%
            \zcsetup{countertype={enumi=axiom}}
            \setlist[enumerate,1]{ref={\csname theaxiom\endcsname.(\arabic*)}}
            \setlist[enumerate,2]{ref={\theaxiom.(\arabic*).(\alph*)}}
        }

        \AtBeginEnvironment{conjecture}{%
            \zcsetup{countertype={enumi=conjecture}}
            \setlist[enumerate,1]{ref={\csname theconjecture\endcsname.(\arabic*)}}
            \setlist[enumerate,2]{ref={\theconjecture.(\arabic*).(\alph*)}}
        }

        \AtBeginEnvironment{construction}{%
            \zcsetup{countertype={enumi=construction}}
            \setlist[enumerate,1]{ref={\csname theconstruction\endcsname.(\arabic*)}}
            \setlist[enumerate,2]{ref={\theconstruction.(\arabic*).(\alph*)}}
        }

        \AtBeginEnvironment{corollary}{%
            \zcsetup{countertype={enumi=corollary}}
            \setlist[enumerate,1]{ref={\csname thecorollary\endcsname.(\arabic*)}}
            \setlist[enumerate,2]{ref={\thecorollary.(\arabic*).(\alph*)}}
        }

        \AtBeginEnvironment{definition}{%
            \zcsetup{countertype={enumi=definition}}
            \setlist[enumerate,1]{ref={\csname thedefinition\endcsname.(\arabic*)}}
            \setlist[enumerate,2]{ref={\thedefinition.(\arabic*).(\alph*)}}
        }

        \AtBeginEnvironment{example}{%
            \zcsetup{countertype={enumi=example}}
            \setlist[enumerate,1]{ref={\csname theexample\endcsname.(\arabic*)}}
            \setlist[enumerate,2]{ref={\theexample.(\arabic*).(\alph*)}}
        }

        \AtBeginEnvironment{lemma}{%
            \zcsetup{countertype={enumi=lemma}}
            \setlist[enumerate,1]{ref={\csname thelemma\endcsname.(\arabic*)}}
            \setlist[enumerate,2]{ref={\thelemma.(\arabic*).(\alph*)}}
        }

        \AtBeginEnvironment{notation}{%
            \zcsetup{countertype={enumi=notation}}
            \setlist[enumerate,1]{ref={\csname thenotation\endcsname.(\arabic*)}}
            \setlist[enumerate,2]{ref={\thenotation.(\arabic*).(\alph*)}}
        }

        \AtBeginEnvironment{note}{%
            \zcsetup{countertype={enumi=note}}
            \setlist[enumerate,1]{ref={\csname thenote\endcsname.(\arabic*)}}
            \setlist[enumerate,2]{ref={\thenote.(\arabic*).(\alph*)}}
        }

        \AtBeginEnvironment{proposition}{%
            \zcsetup{countertype={enumi=proposition}}
            \setlist[enumerate,1]{ref={\csname theproposition\endcsname.(\arabic*)}}
            \setlist[enumerate,2]{ref={\theproposition.(\arabic*).(\alph*)}}
        }

        \AtBeginEnvironment{remark}{%
            \zcsetup{countertype={enumi=remark}}
            \setlist[enumerate,1]{ref={\csname theremark\endcsname.(\arabic*)}}
            \setlist[enumerate,2]{ref={\theremark.(\arabic*).(\alph*)}}
        }

        \AtBeginEnvironment{theorem}{%
            \zcsetup{countertype={enumi=theorem}}
            \setlist[enumerate,1]{ref={\csname thetheorem\endcsname.(\arabic*)}}
            \setlist[enumerate,2]{ref={\thetheorem.(\arabic*).(\alph*)}}
        }

        \newcommand{\qedshift}{\vspace*{-\baselineskip}}
    }{}

    \ExplSyntaxOn
    \NewDocumentCommand{\mathcommand}{mO{0}m}
     {
      \exp_args:Nc \NewCommandCopy {khue_kept_\cs_to_str:N #1} { #1 }
      \exp_args:Nc \newcommand {khue_new_\cs_to_str:N #1}[#2]{#3}
      \DeclareDocumentCommand {#1} {}
       {
        \mode_if_math:TF
         {
          \use:c {khue_new_\cs_to_str:N #1}
         }
         {
          \use:c {khue_kept_\cs_to_str:N #1}
         }
       }
     }
    \ExplSyntaxOff

    \newenvironment{iffseq}{%
        \global\let\externaldblbackslash\\
        \[\begin{array}{cl}
        \ifundef{\internaldblbackslash}{%
            \global\let\internaldblbackslash\\%
            \gdef\\{\internaldblbackslash\cmidrule{1-1}\morecmidrules\cmidrule{1-1}}%
        }{}
    }{%
        \end{array}\]
        \global\undef\internaldblbackslash
        \global\let\\\externaldblbackslash
    }

    \newsavebox\tikzcdbox
    \newenvironment{tikzcdscale}{%
      \begin{lrbox}{\tikzcdbox}%
      \begin{tikzcd}%
    }{%
      \end{tikzcd}%
      \end{lrbox}%
      \resizebox{\textwidth}{!}{\usebox\tikzcdbox}%
    }

    \mathcommand{\h}{\textup{-}}

    \newcommand{\tx}{\mathrm}
    \mathcommand{\b}{\mathbf}
    
    \newcommand{\cl}{\mathcal}
    \mathcommand{\bb}{\mathbb}
    \DeclareMathAlphabet{\bbn}{U}{bbold}{m}{n}
    \newcommand{\dc}[1]{\TextOrMath{double category\xspace#1}{\b{\bb#1}}}
    
    \mathcommand{\sf}{\mathsf}
    \mathcommand{\u}{\underline}
    \mathcommand{\o}{\overline}

    \newcommand{\TODO}[1][TODO]{\textcolor{orange}{\textup{#1}}\xspace}

    \newcommand{\flip}[1]{\text{\rotatebox[origin=c]{-180}{$#1$}}}
    
    \newcommand{\datetoday}{\date{\cleanlookdateon\today}}

    \newcommand{\defeq}{\mathrel{:=}}
    \mathcommand{\d}{\mathbin{;}}
    \mathcommand{\c}{\circ}
    \newcommand{\ph}[1][]{{({-}_{#1})}}
    
    \newcommand{\iso}{\cong}
    \renewcommand{\equiv}{\simeq}

    \newcommand{\from}{\leftarrow}
    \newcommand{\xto}{\xrightarrow}
    \newcommand{\xfrom}{\xleftarrow}
    \newcommand{\tto}{\Rightarrow}
    \newcommand{\tfrom}{\Leftarrow}

    \newcommand{\ffto}{\hookrightarrow}
    \newcommand{\xffto}{\xhookrightarrow}

    \newcommand{\monoto}{\rightarrowtail}
    
    \newcommand{\squigto}{\rightsquigarrow}

    \makeatletter
    \def\slashedarrowfill@#1#2#3#4#5{%
    $\m@th\thickmuskip0mu\medmuskip\thickmuskip\thinmuskip\thickmuskip
    \relax#5#1\mkern-7mu%
    \cleaders\hbox{$#5\mkern-2mu#2\mkern-2mu$}\hfill
    \mathclap{#3}\mathclap{#2}%
    \cleaders\hbox{$#5\mkern-2mu#2\mkern-2mu$}\hfill
    \mkern-7mu#4$%
    }
    \def\rightslashedarrowfill@{%
    \slashedarrowfill@\relbar\relbar\mapstochar\rightarrow}
    \newcommand\xslashedrightarrow[2][]{%
    \ext@arrow 0055{\rightslashedarrowfill@}{#1}{#2}}
    \def\leftslashedarrowfill@{%
    \slashedarrowfill@\leftarrow\relbar\mapsfromchar\relbar}
    \newcommand\xslashedleftarrow[2][]{%
    \ext@arrow 0055{\leftslashedarrowfill@}{#1}{#2}}
    \def\rightdoubleslashedarrowfill@{%
    \slashedarrowfill@\relbar\relbar{\mapstochar\mkern-2mu\mapstochar}\rightarrow}
    \newcommand\xdoubleslashedrightarrow[2][]{%
    \ext@arrow 0055{\rightdoubleslashedarrowfill@}{#1}{#2}}
    \makeatother

    \newcommand{\xlto}{\xslashedrightarrow}
    \newcommand{\lto}{\xlto{}}
    \newcommand{\xlfrom}{\xslashedleftarrow}

    \newcommand{\inv}{^{-1}}
    \newcommand{\op}{{}^\tx{op}}
    \newcommand{\co}{{}^\tx{co}}
    \newcommand{\coop}{{}^\tx{co\,op}}

    \newcommand{\tp}[1]{\langle#1\rangle}
    \newcommand{\unit}{{\tp{}}}
    
    \newcommand{\adj}{\dashv}
    
    \newcommand{\ob}[1]{|#1|}

    \DeclareFontFamily{U}{min}{}
    \DeclareFontShape{U}{min}{m}{n}{<-> udmj30}{}

    \mathcommand{\comma}{\downarrow}

    \newcommand{\copi}{\flip\pi}
    
    \newsavebox{\whitecircstar}\sbox{\whitecircstar}{\kern.075em\tikz{\node[draw, circle,line width=.36pt, inner sep=0]{$*$};}\kern.075em}
    
    \newsavebox{\blackcircstar}\sbox{\blackcircstar}{\kern.075em\tikz{\node[fill, circle, line width=.36pt, inner sep=0, text=white]{$*$};}\kern.075em}

    \newcommand{\skt}{\olessthan}
    \newcommand{\pow}{\pitchfork}
    \newcommand{\copow}{\cdot}

    \makeatletter
    \def\widebreve{\mathpalette\wide@breve}
    \def\wide@breve#1#2{\sbox\z@{$#1#2$}%
         \mathop{\vbox{\m@th\ialign{##\crcr
    \kern0.08em\brevefill#1{0.8\wd\z@}\crcr\noalign{\nointerlineskip}%
                        $\hss#1#2\hss$\crcr}}}\limits}
    \def\brevefill#1#2{$\m@th\sbox\tw@{$#1($}%
      \hss\resizebox{#2}{\wd\tw@}{\rotatebox[origin=c]{90}{\upshape(}}\hss$}
    \makeatother

    \NewDocumentCommand{\jrule}{om}{%
        \IfNoValueTF{#1}
            {\textsc{#2}}
            {$#1$-\textsc{#2}}%
    }

    \newcommand{\N}{{\bb N}}
    
    \newcommand{\Set}{{\b{Set}}}
    
    \newcommand{\Cat}{\b{Cat}}

    \newcommand{\ff}{fully faithful}
    
    \newcommand{\ffness}{full faithfulness}

    \newcommand{\xth}{\textsuperscript{th}}

    \newcommand{\eg}{e.g.\@\xspace}
    \newcommand{\ie}{i.e.\@\xspace}
    \newcommand{\viz}{viz.\@\xspace}
    \newcommand{\cf}{cf.\@\xspace}
    
    \newcommand{\aka}{a.k.a.\@\xspace}
    \NewDocumentCommand{\etc}{t.}{etc.\@\xspace}
    \NewDocumentCommand{\ibid}{t.}{ibid.\@\xspace}
    \NewDocumentCommand{\loccit}{t.}{loc.\ cit.\@\xspace}

    \newcommand{\sdc}{strict double category}

    \newcommand{\pdc}{pseudo double category}
    \newcommand{\pdcs}{pseudo double categories}

    \newcommand{\dcs}{double categories\xspace}
    \newcommand{\vd}{virtual double}
    \newcommand{\vdc}{\vd{} category}
    \newcommand{\vdcs}{\vd{} categories}
    
    \newcommand{\ve}{virtual equipment}

    \newcommand{\fct}{formal category theory}

\makeatletter

\define@key{beamerframe}{c}[true]{
    \beamer@frametopskip=0pt plus 1fill\relax%
    \beamer@framebottomskip=0pt plus 1fill\relax%
}

\patchcmd{\beamer@sectionintoc}{\vfill}{\vskip\itemsep}{}{}

\makeatother

\ifarticle{}{

  \colorlet{colour-bg}{black!85} 
  \definecolor{colour-primary}{HTML}{cc80ff} 
  \colorlet{colour-text}{black!10} 
  \colorlet{colour-subtle}{black!40} 
  \colorlet{colour-block-bg}{black!80} 
  \definecolor{colour-warning-bg}{HTML}{ffea80} 
  \definecolor{colour-warning-primary}{HTML}{e08152} 

  \hypersetup{linktoc=all,colorlinks, citecolor={colour-primary}, linkcolor={colour-primary}, urlcolor={colour-primary}} 
  \urlstyle{sf}

  \usepackage{changepage} 

  \usepackage{ragged2e} 

  \setbeamertemplate{section in toc}[sections numbered]

  \apptocmd{\frame}{}{\justifying}{}

  \usefonttheme[onlymath]{serif}

  \beamertemplatenavigationsymbolsempty
  \setbeamertemplate{footline}{
    \hfill%
    \usebeamercolor[fg]{page number in head/foot}%
    \usebeamerfont{page number in head/foot}%
    \setbeamertemplate{page number in head/foot}[pagenumber]%
    \usebeamertemplate*{page number in head/foot}\kern.2cm\vskip.2cm%
  }

  \setbeamertemplate{frametitle}[default][center]
  \addtobeamertemplate{frametitle}{\vspace*{1ex}}{\vspace*{1ex}}

  \setbeamertemplate{itemize item}{$\bullet$}

  \setbeamertemplate{theorems}[numbered]
  \newtheorem{proposition}[theorem]{\translate{Proposition}}

  \addtobeamertemplate{block begin}
  {}
  {\vspace{0ex} 
  \begin{adjustwidth}{1em}{1em} 
  \tikzcdset{background color=colour-block-bg}
  }
  \addtobeamertemplate{block end}{\end{adjustwidth}%
  \vspace{1ex}} 
  {}
  \renewenvironment<>{block}[1]{%
      \begin{actionenv}#2%
        \def\insertblocktitle{\begin{adjustwidth}{.8em}{.8em}#1\end{adjustwidth}}%
        \par%
        \usebeamertemplate{block begin}}
      {\par%
        \usebeamertemplate{block end}%
      \end{actionenv}}

  \addtobeamertemplate{block example begin}
  {}
  {\vspace{0ex} 
  \begin{adjustwidth}{1em}{1em} 
  \tikzcdset{background color=colour-block-bg}
  }
  \addtobeamertemplate{block example end}{\end{adjustwidth}%
  \vspace{1ex}} 
  {}
  \renewenvironment<>{exampleblock}[1]{%
      \begin{actionenv}#2%
          \def\insertblocktitle{\begin{adjustwidth}{.8em}{.8em}#1\end{adjustwidth}}%
          \par%
          \only<presentation>{
            \setbeamercolor{local structure}{parent=example text}}%
          \usebeamertemplate{block example begin}}
        {\par%
          \usebeamertemplate{block example end}%
        \end{actionenv}}

    \setbeamercolor{background canvas}{bg=colour-bg}
    \tikzcdset{background color=colour-bg}

    \setbeamercolor{block title}{fg=colour-primary,bg=colour-block-bg}
    \setbeamercolor{block title example}{fg=colour-primary,bg=colour-block-bg}
    \setbeamercolor{block body}{bg=colour-block-bg}
    \setbeamercolor{block body example}{bg=colour-block-bg}

    \setbeamercolor{title}{fg=colour-primary}
    \setbeamercolor{frametitle}{fg=colour-primary}
    \setbeamercolor{alerted text}{fg=colour-primary}

    \setbeamercolor{normal text}{fg=colour-text}
    \setbeamercolor{item}{fg=colour-text}
    \setbeamercolor{section in toc}{fg=colour-text}

    \setbeamercolor{page number in head/foot}{fg=colour-subtle}

    \setbeamercolor*{bibliography entry author}{fg=colour-text}
    \setbeamercolor*{bibliography entry note}{fg=colour-text}

}

\bibliography{references}

\newcommand{\C}{{\cl C}}
\newcommand{\D}{{\b D}}
\newcommand{\E}{{\b E}}
\newcommand{\G}{{\b G}}
\newcommand{\M}{{\b M}}
\newcommand{\Grph}{{\b{Grph}}}
\newcommand{\VDbl}{{\b{VDblCat}}}
\newcommand{\VDbln}{\b{NVDblCat}}
\newcommand{\VEquip}{\b{VEquip}}
\newcommand{\Dbl}{{\b{DblCat}}}
\newcommand{\Multicat}{{\b{Multicat}}}
\newcommand{\DFib}{{\b{DFib}}}
\newcommand{\LoCo}{{\b{LoCo}}}

\newcommand{\A}{{\dc A}}
\newcommand{\B}{{\dc B}}
\newcommand{\F}{{\dc F}}
\newcommand{\W}{{\dc W}}
\newcommand{\X}{{\dc X}}
\newcommand{\Y}{{\dc Y}}

\newcommand{\Span}{{\dc{Span}}}
\newcommand{\Dist}{{\dc{Dist}}}
\newcommand{\Mod}{{\dc{Mod}}}
\newcommand{\Mat}{{\dc{Mat}}}
\newcommand{\U}[1]{#1'}
\newcommand{\Lax}{{\dc{Lax}}}
\newcommand{\LaxN}{{\dc{Lax_N}}}
\newcommand{\LvAmnIso}{{\dc{LvwAmIsof}}}
\newcommand{\LvDFib}{{\dc{LvwDFib}}}
\newcommand{\DFibDbl}{{\dc{DFib}}}
\newcommand{\El}{{\dc{El}}}
\newcommand{\Ch}{{\dc{Ch}}}
\newcommand{\Dis}[1]{{\b 0 \rightrightarrows #1}}
\newcommand{\dfibsl}{/_{\tx{dfib}\,}}
\mathcommand{\P}{{\dc P}}

\newcommand{\up}{^{{-}}}
\newcommand{\down}{_{{-}}}
\newcommand{\crr}{\ob}
\newcommand{\I}{{\sf I}}
\newcommand{\cod}{{\tx{cod}}}
\newcommand{\dom}{{\tx{dom}}}
\renewcommand{\exp}{_{\tx{exp}}}
\newcommand{\pb}{{\it{pb}}}
\newcommand{\cart}{{\tx{cart}}}
\newcommand{\opcart}{{\tx{opcart}}}
\newcommand{\opt}{{}^\tx{op\,t}}
\newcommand{\ex}[1]{\epsilon_{#1}}

\DeclareFontFamily{U}{mathx}{}
\DeclareFontShape{U}{mathx}{m}{n}{<-> mathx10}{}
\DeclareSymbolFont{mathx}{U}{mathx}{m}{n}
\DeclareMathAccent{\widecheck}{0}{mathx}{"71}

\title[Exponentiable virtual double categories]{Exponentiable virtual double categories \\ and presheaves for double categories}
\author{Nathanael Arkor}
\address{Department of Software Science, Tallinn University of Technology, Estonia}
\thanks{During the writing of this paper, the author was supported by a departmental postdoctoral grant from the Department of Software Science at Tallinn University of Technology, and by the Estonian Research Council grant PSG1242.}
\subjclass{18D15,18D20,18D60,18M65,18N10}
\datetoday

\begin{document}

\begin{abstract}
    Given a pair of \pdcs{} $\A$ and $\B$, the lax functors from $\A$ to $\B$, along with their transformations, modules, and multimodulations, assemble into a \vdc{} $\Lax(\A, \B)$. We exhibit a universal property of this construction by observing that it arises naturally from the consideration of exponentiability for \vdcs. In particular, we show that every \pdc{} is exponentiable as a \vdc{}, whereby the \vdc{} $\Lax(\A, \B)$ of lax functors arises as the \vdc{} $\Mod(\B^\A)$ of monads and modules in the exponential $\B^\A$. We explore some consequences of this characterisation, demonstrating that it leads to simple proofs of statements that heretofore required unwieldy computations. For instance, we deduce that the 2-category of \pdcs{} and lax functors is enriched in the 2-category of normal \vdcs, and demonstrate that several aspects of the Yoneda theory of pseudo double categories -- such as the correspondence between presheaves and discrete fibrations -- are substantially simplified by this perspective.
\end{abstract}

\dedicatory{Dedicated to Bob Par\'e, a profound influence on my understanding \\ of category theory, in celebration of his 80\xth{} birthday.}

\maketitle

\setcounter{tocdepth}{1}
\tableofcontents

\section{Introduction}

In the last twenty years, it has become increasingly clear that weak \dcs{} play a fundamental role in category theory~\cite{grandis1999limits,grandis2004adjoint,shulman2008framed}. However, due to its relative youth, several basic aspects of the theory of \dcs{} have remained underdeveloped. One such aspect is the theory of \emph{representability} for \dcs, extending the theory of representable presheaves for categories. The notion of \emph{presheaf} for \dcs{} was first studied in detail by \textcite{pare2011yoneda}, who showed that, for each \dc{} $\A$, there exists a \ff{} functor $Y_\A \colon \A \to \Lax(\A\opt, \Span)$, sending each object in $\A$ to a contravariant lax functor from $\A$ into the \dc{} of sets and spans; the appropriateness of this definition is justified by a corresponding Yoneda lemma. The generalisation from categories to \dcs{} is by no means routine: for a category $\b A$, the presheaves on $\b A$ assemble into a category $\Set^{\b A\op}$; on the other hand, $\Lax(\A\opt, \Span)$ is, in general, not a \pdc{} but merely a \emph{\vdc} -- a structure which generalises a \dc{} in the same way that a multicategory generalises a monoidal category. Thus, surprisingly, the theory of \dcs{} appears to lead to the study of \vdcs.

While the development of \textcite{pare2011yoneda} is insightful, the nature of the \vdc{} of presheaves is left rather mysterious. For a category $\b A$, the Yoneda embedding ${y_{\b A} \colon \b A \to \Set^{\b A\op}}$ corresponds, via the cartesian closed structure of the 2-category $\Cat$, to the hom-functor $\b A({-}, {-}) \colon \b A\op \times \b A \to \Set$. This gives a particularly elementary perspective on the Yoneda embedding, every category being equipped with a hom-functor essentially by definition. We might therefore hope that something similar is true for the double categorical Yoneda embedding $Y_\A \colon \A \to \Lax(\A\opt, \Span)$: namely, that it corresponds, via some notion of `transposition', to the lax hom-functor $\A({-}, {-}) \colon \A\opt \times \A \to \Span$. The present paper stemmed from the desire to understand to what extent this is true.

\subsection*{Exponentiability of \vdcs}

Given that, for $\Cat$, transposition is induced by cartesian closure, a first instinct might be to guess that the 2-category $\VDbl$ of \vdcs{} is cartesian closed and that, for each pair $\A$ and $\B$ of \dcs, the \vdc{} $\Lax(\A, \B)$ of lax functors forms the exponential $\B^\A$. This is false on both counts, but the thought nevertheless turns out to be fruitful. To understand what is going on, we take inspiration from the theory of multicategories. Just as monoidal categories may be seen as \dcs{} with a single object and single tight morphism, so too may multicategories be seen as \vdcs{} with a single object and single tight morphism, and in many respects the theory of \vdcs{} may thereby be seen to extend the theory of multicategories. The 2-category of multicategories is \emph{not} cartesian closed, but it does admit some exponentials: \textcite{pisani2014sequential} established that a multicategory $\M$ is exponentiable (\ie{} that $\ph \times M$ admits a right adjoint $\ph^\M$) if and only if it is \emph{promonoidal} in the sense of \textcite{day1970closed}. In particular, every monoidal category is promonoidal, hence exponentiable as a multicategory. Given monoidal categories $\M$ and $\b N$, one might thus expect that their exponential $\b N^\M$ in the category of multicategories to comprise the lax monoidal functors from $\M$ to $\b N$. This is not true; however, we can obtain the lax monoidal functors from $\M$ to $\b N$ as the \emph{monoids} in $\b N^\M$.

Guided by the one-object setting, we carry out a similar analysis for \vdcs. As for multicategories, while the 2-category $\VDbl$ is not cartesian closed, every \dc{} is exponentiable as a \vdc{} and so, given \dcs{} $\A$ and $\B$, we may form an exponential \vdc{} $\B^\A$. The analogue of a monoid in a multicategory is a monad in a \vdc\footnotemark{}, and for every \vdc{} $\X$, there is a \vdc{} $\Mod(\X)$ of monads and modules in $\X$.
\footnotetext{In this paper, a \emph{monad} in a \vdc{} will always refer to a \emph{loose monad}, rather than a \emph{tight monad} (which can only be defined in a \vdc{} with enough identity loose morphisms).}%
Our main theorem establishes the expected analogue\footnotemark{} of the multicategorical story.
\footnotetext{In fact, the virtual double categorical story is better behaved than the multicategorical one, as the monoids in a multicategory merely form a category, rather than a multicategory, whereas the monads in a \vdc{} form another \vdc.}

\begin{theorem*}[\cref{Lax-is-Mod}]
    \citeauthor{pare2011yoneda}'s \vdc{} $\Lax(\A, \B)$ is isomorphic to $\Mod(\B^\A)$.
\end{theorem*}

In answer to our original question, \citeauthor{pare2011yoneda}'s Yoneda embedding may be decomposed into the following composite,
\[\begin{tikzcd}
	\A & {\Mod(\A)} & {\Mod(\Span^{\A\opt})}
	\arrow[hook, from=1-1, to=1-2]
	\arrow[hook, from=1-2, to=1-3]
\end{tikzcd}\]
in which the first functor sends each object in $\A$ to the identity monad on that object, and in which the second functor is the image under $\Mod$ of the transposition $\A \ffto \Span^{\A\opt}$ of the lax hom-functor $\A({-}, {-}) \colon \A\opt \times \A \to \Span$. Thus, the Yoneda embedding is obtained, as desired, by transposition of the hom-functor, `up to $\Mod$'.

\subsection*{Enrichment and universality}

While our motivation for such a characterisation of $\Lax(\A, \B)$ was entirely conceptual, it turns out to have a number of practical implications. For instance, it follows easily (\cref{Dbl-is-enriched}) that the 2-category of \dcs{} and lax functors is enriched in the 2-category of \emph{normal} \vdcs{} (which are \vdcs{} that admit identity loose morphisms). Furthermore, the \vdc{} $\Lax(\A, \B)$ acquires a universal property from the combined universal properties of the $\Mod$ construction~\cite{cruttwell2010unified} and the exponential $\B^\A$, equipping the 2-category of \vdcs{} with a partially-closed skew-monoidal structure (\cref{partial-closed-structure}). As a consequence, it becomes unnecessary to manipulate $\Lax(\A, \B)$ directly; the benefits of this will be immediately apparent to readers familiar with previous work on \vdcs{} of lax functors~\cite{pare2011yoneda,pare2013composition,lambert2021discrete}.

\subsection*{Representability for \dcs}

Having answered our motivating question, we explore the extent to which this perspective simplifies and clarifies the theory of representability for \dcs{} developed by \textcite{pare2011yoneda}. One of the most important aspects of category theory is the relationship between families and bundles. For categories, this relationship takes the form of an equivalence between the category of presheaves on a fixed category $\b A$ and the category of discrete fibrations over $\b A$. The situation for \dcs{} is similar. \textcite{pare2011yoneda} observed that every lax functor $F \colon \A\opt \to \Span$ admits the construction of a \dc{} $\El(F)$ of elements, equipped with a projection functor $\pi_F \colon \El(F) \to \A$, which is a discrete fibration of \dcs. Subsequently, \textcite{lambert2021discrete}, established that, for $\A$ a \emph{strict} \dc, this assignment extends to an equivalence between the \vdc{} $\Lax(\A\opt, \Span)$ and a \vdc{} $\DFibDbl(\A)$ of discrete fibrations over $\A$. Unexpectedly, the assumption that $\A$ be strict is not simply a matter of convenience, but rather is crucial in \citeauthor{lambert2021discrete}'s proof strategy; we shall elaborate on this matter in \cref{elements-construction}. Later, \textcite{cruttwell2022double} extended the correspondence to permit $\A$ to be an arbitrary \pdc, but merely established an equivalence of categories rather than of \vdcs.

To demonstrate the utility of our characterisation of $\Lax(\A, \B)$, we give in \cref{elements-construction} a new proof of the correspondence between presheaves on \dcs{} and discrete fibrations of \dcs. Our strategy has several advantages over previous approaches. For instance, we establish an equivalence of \vdcs{}, rather than merely of categories, with respect to an arbitrary \pdc{} $\A$. Our correspondence thus strictly generalises the previous correspondences. Moreover, we go a step further and establish a two-sided version of the correspondence, relating lax functors $\B\opt \times \A \to \Span$ and two-sided discrete fibrations from $\A$ to $\B$. A feature of our proof strategy, in contrast to previous approaches~\cite{lambert2021discrete,cruttwell2022double}, is that it avoids substantial manual calculation by deriving the correspondence directly from that for categories.

\subsection*{Connections to bicategory theory}

Though this paper is primarily motivated by the theory of \dcs, our investigation nevertheless illuminates several constructions arising in bicategory theory. For one, the universal property we exhibit for $\Lax(\A, \B)$ is of interest even in the setting in which $\A$ and $\B$ are merely bicategories, since, in this case, it recovers the two-dimensional structure formed by lax functors between bicategories and their modules and multimodulations, studied by \citeauthor{cockett2003modules}~\cite{cockett2003morphisms,cockett2003modules}, which had not been known to have a universal property. For another, we show in \cref{local-cocompletion} that the local cocompletion of a bicategory~$\cl A$, as studied for instance by \textcite{kelly2002categories}, arises naturally from the \pdc{} $\Lax(\cl A\co, \Span)$, exhibiting a new universal property of the local cocompletion (\cref{local-cocompletion-via-exponentiation}). As an application, we show that contravariant lax functors into $\Span$ are examples of categories enriched in bicategories (\cref{lax-functors-are-graded-categories}).

\subsection{Outline of paper}

We begin in \cref{background-on-dcs} by recalling the basic definitions of \dc{} theory (both pseudo and virtual) that we will use throughout the paper.
In \cref{exponentiability}, we study exponentiability for \vdcs{} and, in particular, establish that pseudo double categories are exponentiable \vdcs{} (\cref{representable-vdcs-are-exponentiable}). In \cref{the-vdc-of-loose monads}, we recall the definition of the \vdc{} $\Mod(\X)$ of monads and modules in a \vdc{} $\X$, recall its universal property as the cofree normal \vdc{} on $\X$ (\cref{Mod-is-right-adjoint}), and study some of its properties, including its lax-idempotence (\cref{lax-idempotence}) and its interaction with the slice construction (\cref{slicing-over-a-monad}). In \cref{powers-and-copowers}, we recall the construction of the free normal \vdc{} on $\X$~\cite{fujii2025familial}, and establish some properties useful for the following sections. In \cref{lax-functors-as-monads}, we establish our main result, demonstrating that \citeauthor{pare2011yoneda}'s \vdc{} $\Lax(\A, \B)$ of lax functors between \pdcs{} arises as $\Mod(\B^\A)$ (\cref{Lax-is-Mod}). The remainder of the paper is dedicated to exploring consequences of this result. For instance, we show that the 2-category of \pdcs{} and lax functors is enriched in the 2-category of (normal) \vdcs{} (\cref{Dbl-is-enriched}), and exhibit a universal property of $\Lax(\A, \B)$ (\cref{partial-closed-structure}). In \cref{Yoneda-theory}, we study representability for presheaves of \dcs, and show that our perspective simplifies aspects of the development of \textcite{pare2011yoneda}. In \cref{exponentiability-revisited}, we utilise the Yoneda embedding for \dcs{} to give a complete characterisation of the exponentiable \vdcs. In \cref{local-cocompletion}, we explore the relationship between the Yoneda embedding for \dcs{} and local cocompletion for bicategories. Finally, in \cref{elements-construction}, we use the perspective to give a new and substantially simpler proof that contravariant lax functors from a \pdc{} $\A$ into $\Span$ are equivalent to discrete fibrations over $\A$, generalising the main theorem of \textcite{lambert2021discrete} from strict double categories to \pdcs, and then use the same techniques to establish the two-sided correspondence.

\subsection{Tight representability and loose representability}

Above, and in this paper as a whole, we focus on just one aspect of representability for \dcs. However, we find it helpful to describe briefly how this paper fits into the larger picture, by clarifying the relationship between two approaches for representability for \dcs{} that may be found in the literature.

A presheaf on a category is most helpfully considered as a special case of a more general concept: a \emph{distributor} between categories, which provides a notion of \emph{heteromorphism} between objects of different categories~\cite{benabou1973distributeurs}. Since there is a single notion of morphism between objects of a category, there is a single notion of distributor between categories. However, \dcs{} have two kinds of morphism -- the \emph{tight morphisms} (which compose strictly), and the \emph{loose morphisms} (which compose weakly) -- and correspondingly there are two notions of distributor between \dcs{}: \emph{tight distributors} and \emph{loose distributors}, respectively providing notions of \emph{tight heteromorphisms} and \emph{loose heteromorphisms} between objects of different \dcs\footnotemark{}. By specialising to (tight or loose) distributors with terminal domain or codomain, we obtain notions of \emph{tight (co)presheaf} and \emph{loose (co)presheaf}.
\footnotetext{Tight and loose distributors may be defined efficiently as (necessarily strict) functors of \dcs{} into the free-standing tight morphism and into the free-standing loose morphism respectively. The reader is invited to verify that these two definitions recover the notions of (tight) presheaf of \textcite{pare2011yoneda} and (loose) presheaf of \textcite{fiore2012double} respectively.}

This reveals there are two notions of representability for \dcs: \emph{tight representability}, which is concerned with the representability of tight presheaves (and, more generally, tight distributors); and \emph{loose representability}, which is concerned with the representability of loose presheaves (and, more generally, loose distributors). \citeauthor{pare2011yoneda}'s~\cite{pare2011yoneda} presheaves for \dcs{} -- namely, lax functors $\A\opt \to \Span$ -- are the \emph{tight} presheaves. Consequently, we focus in this paper purely on tight representability.

Loose representability has been considered in the literature by \textcite[\S3]{fiore2012double}, albeit only in the context of strict \dcs. The precise connection between the two approaches deserves further clarification.

\subsection{Acknowledgements}

This paper has been improved thanks to several people. Bryce Clarke, Bojana Femi\'c, and the two anonymous reviewers provided detailed and perspicacious comments on the paper. Kevin Carlson and Ea Thompson identified an oversight in a provisional definition of pro-double category. Steve Lack, Graham Manuell, and Dylan McDermott offered helpful suggestions. The author is also grateful to James Deikun and Keisuke Hoshino for general discussions regarding \vdcs.

\section{Pseudo double categories and virtual double categories}
\label{background-on-dcs}

We begin by recalling the basic definitions of pseudo and virtual \dcs{}, both for the benefit of the reader unfamiliar with these notions, and to establish our terminology and notation. The reader familiar with these concepts may prefer to skip directly to \cref{exponentiability}.

\subsection{Virtual \dcs}
\label{background-on-vdcs}

The central concept in this paper is that of a \emph{\vdc}, which is to a multicategory what a \pdc{} is to a monoidal category. Like a \dc{}, a \vdc{} is a structure with two kinds of morphism, which we call \emph{tight morphisms} and \emph{loose morphisms}. The objects and tight morphisms in a \vdc{} assemble into a category, so that tight morphisms compose strictly associatively and unitally. However, it is not generally possible to compose the loose morphisms in a \vdc. Instead, a \vdc{} has \emph{cells} that are framed at the top by a path of loose morphisms, representing a `virtual composite'. If you are encountering the definition for the first time, this may appear a little intimidating. However, in a \vdc, unlike in a \pdc, there is no coherence data, and it often turns out to be simpler to work with \vdcs{} than with \pdcs{} for this reason.

\begin{definition}[{\cites[61]{burroni1971tcategories}[Definition~2.1]{cruttwell2010unified}}]
    A \emph{\vdc{}} $\X$ comprises the following data.
    \begin{enumerate}
        \item A category $\u\X$ of \emph{objects} and \emph{tight morphisms}.
        We denote a tight morphism $f$ from an object $A$ to an object $B$ by an arrow $f \colon A \to B$; denote the composition of tight morphisms $f \colon A \to B$ and $g \colon B \to C$ both by $(f \d g) \colon A \to C$ and by $g f \colon A \to C$; and denote the identity of an object $A$ by $1_A \colon A \to A$, or simply by $=$ in pasting diagrams.
        \item For each pair of objects $A$ and $B$, a class of \emph{loose morphisms} from $A$ to $B$. We denote a loose morphism $p$ from $A$ to $B$ by an arrow with a vertical stroke $p \colon A \lto B$.
        \item For each chain of loose morphisms $p_1, \ldots, p_n$ ($n \geq 0$) and compatible tight morphisms $f, f'$ and loose morphism $q$ (together forming a \emph{frame}), a class of \emph{($n$-ary) cells} with the given frame.
        \[\begin{tikzcd}
            {A_0} & {A_1} & \cdots & {A_{n - 1}} & {A_n} \\
            B &&&& {B'}
            \arrow[""{name=0, anchor=center, inner sep=0}, "f"', from=1-1, to=2-1]
            \arrow["{p_1}"'{inner sep=.8ex}, "\shortmid"{marking}, from=1-2, to=1-1]
            \arrow["{p_2}"'{inner sep=.8ex}, "\shortmid"{marking}, from=1-3, to=1-2]
            \arrow["{p_{n - 1}}"'{inner sep=.8ex}, "\shortmid"{marking}, from=1-4, to=1-3]
            \arrow["{p_n}"'{inner sep=.8ex}, "\shortmid"{marking}, from=1-5, to=1-4]
            \arrow[""{name=1, anchor=center, inner sep=0}, "{f'}", from=1-5, to=2-5]
            \arrow["q"{inner sep=.8ex}, "\shortmid"{marking}, from=2-5, to=2-1]
            \arrow["\phi"{description}, draw=none, from=1, to=0]
        \end{tikzcd}\]
        Observe that our pasting diagrams are written from right-to-left (matching nondiagrammatic composition order of loose morphisms). Such a cell is \emph{globular} if $f$ and $f'$ are identities; we denote such a globular cell by $\phi \colon p_1, \ldots, p_n \tto q$.

        In pasting diagrammatic notation, we denote a nullary cell by a square of the following form\footnote{Some authors, \eg{} \cite{cruttwell2010unified}, prefer to render nullary cells as triangles. However, we find this to be both less uniform (for instance, such authors do not render binary cells as trapeziums) and less visually pleasing than rendering all cells as rectangles.}.
        \[\begin{tikzcd}
            A & A \\
            B & {B'}
            \arrow[""{name=0, anchor=center, inner sep=0}, "f"', from=1-1, to=2-1]
            \arrow[equals, from=1-2, to=1-1]
            \arrow[""{name=1, anchor=center, inner sep=0}, "{f'}", from=1-2, to=2-2]
            \arrow["p"{inner sep=.8ex}, "\shortmid"{marking}, from=2-2, to=2-1]
            \arrow["\phi"{description}, draw=none, from=1, to=0]
        \end{tikzcd}\]
        \item For every configuration of cells of the following shape,
        \[\begin{tikzcd}
            \cdot & \cdots & \cdot & \cdots & \cdot & \cdots & \cdot \\
            \cdot && \cdot & \cdots & \cdot && \cdot \\
            \cdot &&&&&& \cdot
            \arrow["\shortmid"{marking}, from=2-7, to=2-5]
            \arrow[""{name=0, anchor=center, inner sep=0}, "{f'}", from=1-7, to=2-7]
            \arrow[""{name=1, anchor=center, inner sep=0}, from=1-5, to=2-5]
            \arrow["{p_{m_n}^n}"', "\shortmid"{marking}, from=1-7, to=1-6]
            \arrow["{p_{1}^n}"', "\shortmid"{marking}, from=1-6, to=1-5]
            \arrow["{p_{m_1}^1}"', "\shortmid"{marking}, from=1-3, to=1-2]
            \arrow["{p_1^1}"', "\shortmid"{marking}, from=1-2, to=1-1]
            \arrow["\shortmid"{marking}, from=2-3, to=2-1]
            \arrow[""{name=2, anchor=center, inner sep=0}, from=1-3, to=2-3]
            \arrow[""{name=3, anchor=center, inner sep=0}, "f"', from=1-1, to=2-1]
            \arrow[""{name=4, anchor=center, inner sep=0}, "{g'}", from=2-7, to=3-7]
            \arrow["q", "\shortmid"{marking}, from=3-7, to=3-1]
            \arrow[""{name=5, anchor=center, inner sep=0}, "g"', from=2-1, to=3-1]
            \arrow["{\phi_n}"{description}, draw=none, from=0, to=1]
            \arrow["{\phi_1}"{description}, draw=none, from=2, to=3]
            \arrow["\psi"{description}, draw=none, from=4, to=5]
        \end{tikzcd}\]
        a cell,
        \[\begin{tikzcd}
            \cdot & \cdots & \cdot \\
            \cdot && \cdot
            \arrow["{p_{m_n}^n}"', "\shortmid"{marking}, from=1-3, to=1-2]
            \arrow["q", "\shortmid"{marking}, from=2-3, to=2-1]
            \arrow[""{name=0, anchor=center, inner sep=0}, "{f' \d g'}", from=1-3, to=2-3]
            \arrow[""{name=1, anchor=center, inner sep=0}, "{f \d g}"', from=1-1, to=2-1]
            \arrow["{p_1^1}"', "\shortmid"{marking}, from=1-2, to=1-1]
            \arrow["{(\phi_1, \ldots, \phi_n) \d \psi}"{description}, draw=none, from=0, to=1]
        \end{tikzcd}\]
        the \emph{composite}.
        \item For each loose morphism $p \colon A' \lto A$, a cell $1_p \colon p \tto p$, the \emph{identity} of $p$, denoted simply as $=$ in pasting diagrams.
        \[\begin{tikzcd}
            A & {A'} \\
            A & {A'}
            \arrow[""{name=0, anchor=center, inner sep=0}, equals, from=1-1, to=2-1]
            \arrow["p"'{inner sep=.8ex}, "\shortmid"{marking}, from=1-2, to=1-1]
            \arrow[""{name=1, anchor=center, inner sep=0}, equals, from=1-2, to=2-2]
            \arrow["p"{inner sep=.8ex}, "\shortmid"{marking}, from=2-2, to=2-1]
            \arrow["{=}"{description}, draw=none, from=1, to=0]
        \end{tikzcd}\]
    \end{enumerate}
    We shall sometimes denote by
    \[\begin{tikzcd}
        \cdot & \cdots & \cdot \\
        \cdot & \cdots & \cdot
        \arrow["\shortmid"{marking}, from=1-3, to=1-2]
        \arrow["\shortmid"{marking}, from=1-2, to=1-1]
        \arrow["\shortmid"{marking}, from=2-3, to=2-2]
        \arrow["\shortmid"{marking}, from=2-2, to=2-1]
        \arrow[""{name=0, anchor=center, inner sep=0}, Rightarrow, no head, from=1-3, to=2-3]
        \arrow[""{name=1, anchor=center, inner sep=0}, Rightarrow, no head, from=1-1, to=2-1]
        \arrow["{=}"{description}, draw=none, from=0, to=1]
    \end{tikzcd}\]
    the juxtaposition of identity cells
    \[\begin{tikzcd}
        \cdot & \cdot & \cdot & \cdot \\
        \cdot & \cdot & \cdot & \cdot
        \arrow["\shortmid"{marking}, from=1-4, to=1-3]
        \arrow["\shortmid"{marking}, from=2-2, to=2-1]
        \arrow[""{name=0, anchor=center, inner sep=0}, Rightarrow, no head, from=1-4, to=2-4]
        \arrow[""{name=1, anchor=center, inner sep=0}, Rightarrow, no head, from=1-1, to=2-1]
        \arrow["\shortmid"{marking}, from=1-2, to=1-1]
        \arrow["\shortmid"{marking}, from=2-4, to=2-3]
        \arrow["\cdots"{description}, draw=none, from=1-3, to=1-2]
        \arrow["\cdots"{description}, draw=none, from=2-3, to=2-2]
        \arrow[""{name=2, anchor=center, inner sep=0}, Rightarrow, no head, from=1-3, to=2-3]
        \arrow[""{name=3, anchor=center, inner sep=0}, Rightarrow, no head, from=1-2, to=2-2]
        \arrow["{=}"{description}, draw=none, from=0, to=2]
        \arrow["{=}"{description}, draw=none, from=3, to=1]
    \end{tikzcd}\]
    Composition of cells is required to be associative and unital in the evident manner~\cite[Definition~2.1]{cruttwell2010unified}: associativity is implicit in pasting diagrams. A \emph{virtual bicategory} is a \vdc{} in which every tight morphism is an identity.
\end{definition}

\begin{example}
    Every category $\b A$ forms a \vdc{} $\o{\b A}$ with no loose morphisms or cells.
\end{example}

\begin{example}
    \label{multicategory}
    Every multicategory $\M$ forms a \vdc{} $\Sigma\M$ with a single object and single tight morphism, the loose morphisms of $\Sigma\M$ being the objects of $\M$, and the cells of $\Sigma\M$ being the multimorphisms of $\M$.
\end{example}

\begin{example}
    \label{terminal-vdc}
    The terminal \vdc{} is denoted $\bbn 1$, and has a single object, single tight morphism, single loose morphism, and a single $n$-ary cell for each $n \in \N$.
\end{example}

\begin{example}
    \label{Span}
    The \vdc{} $\Span$ is defined as follows.
    \begin{enumerate}
        \item The underlying category is $\Set$, the category of sets and functions.
        \item A loose morphism from $A$ to $B$ is a span $A \from X \to B$.
        \item A cell with the following frame
        \[\begin{tikzcd}
            {A_0} & {X_1} & {A_1} & {X_2} & \cdots & {X_n} & {A_n} \\
            B &&& Y &&& {B'}
            \arrow["f"', from=1-1, to=2-1]
            \arrow["{p_1}"', from=1-2, to=1-1]
            \arrow["{q_1}", from=1-2, to=1-3]
            \arrow["{p_2}"', from=1-4, to=1-3]
            \arrow["{q_2}", from=1-4, to=1-5]
            \arrow["{p_n}"', from=1-6, to=1-5]
            \arrow["{q_n}", from=1-6, to=1-7]
            \arrow["{f'}", from=1-7, to=2-7]
            \arrow["r", from=2-4, to=2-1]
            \arrow["s"', from=2-4, to=2-7]
        \end{tikzcd}\]
        is a span morphism from the wide pullback, as follows.
        \[\begin{tikzcd}
            {A_0} & {X_1} & {X_1 \times_{A_1} X_2 \times_{A_2} \cdots \times_{A_n} X_n} & {X_n} & {A_n} \\
            B && Y && {B'}
            \arrow["f"', from=1-1, to=2-1]
            \arrow["{p_1}"', from=1-2, to=1-1]
            \arrow["{\pi_1}"', from=1-3, to=1-2]
            \arrow["{\pi_n}", from=1-3, to=1-4]
            \arrow["g"{description}, from=1-3, to=2-3]
            \arrow["{q_n}", from=1-4, to=1-5]
            \arrow["{f'}", from=1-5, to=2-5]
            \arrow["r", from=2-3, to=2-1]
            \arrow["s"', from=2-3, to=2-5]
        \end{tikzcd}\]
    \end{enumerate}
    The identity on a span $A \from X \to B$ is given by $1_X$, and the composition of cells is induced by the universal property of the pullback in the evident manner.
\end{example}

\begin{example}[{\cite[Example~2.7]{cruttwell2010unified}}]
    \label{Span(E)}
    More generally, for any category $\E$ with pullbacks, there is a \vdc{} $\Span(\E)$, whose underlying category is $\E$, in which a loose morphism is a span in $\E$, and whose cells and their composition is defined analogously to \cref{Span}. In particular, $\Span \defeq \Span(\Set)$.
\end{example}

\begin{example}
    \label{Dist}
    The \vdc{} $\Dist$ is defined as follows.
    \begin{enumerate}
        \item The underlying category is $\Cat$, the category of \emph{small} categories and functors.
        \item A loose morphism from $\b A$ to $\b B$ is a distributor from $\b A$ to $\b B$, \ie{} a functor $\b B\op \times \b A \to \Set$.
        \item A cell with the following frame
        \[\begin{tikzcd}
            {\b A_0} & \cdots & {\b A_n} \\
            {\b B} && {\b B'}
            \arrow[""{name=0, anchor=center, inner sep=0}, "f"', from=1-1, to=2-1]
            \arrow["{p_1}"'{inner sep=.8ex}, "\shortmid"{marking}, from=1-2, to=1-1]
            \arrow["{p_n}"'{inner sep=.8ex}, "\shortmid"{marking}, from=1-3, to=1-2]
            \arrow[""{name=1, anchor=center, inner sep=0}, "{f'}", from=1-3, to=2-3]
            \arrow["q"{inner sep=.8ex}, "\shortmid"{marking}, from=2-3, to=2-1]
            \arrow["\phi"{description}, draw=none, from=1, to=0]
        \end{tikzcd}\]
        comprises a family of functions
        \[\phi_{a_0, \ldots, a_n} \colon p_1(a_0, a_1) \times \cdots \times p_n(a_{n - 1}, a_n) \to q(f a_0, f' a_n)\]
		for each family of objects $a_0 \in \b A_0, \dots, a_n \in \b A_n$, satisfying naturality conditions~\cite[Definition~8.1]{arkor2024formal}.
    \end{enumerate}
    The identity on a distributor is given by the family of identity functions, and the composition of cells is induced by functoriality of the cartesian product.
\end{example}

\begin{example}[{\cite[Example~2.10]{cruttwell2010unified}}]
    \label{Dist(E)}
    More generally, for any category $\E$ with pullbacks, there is a \vdc{} $\Dist(\E)$, whose underlying category is $\Cat(\E)$ (the category of categories and functors internal to $\E$), in which a loose morphism is an internal distributor, and whose cells are internal natural transformations~\cite[\S2.4]{johnstone1977topos}. Since we will only refer to this example in passing, we will not spell it out, but will rather define $\Dist(\E)$ by means of a general construction on $\Span(\E)$ in \cref{Mod-Span}. In particular, $\Dist \defeq \Dist(\Set)$.
\end{example}

\begin{definition}[\cite{burroni1971tcategories}]
    \label{functor}
    A \emph{functor} of \vdcs{} is a homomorphism, having an assignment on objects, tight morphisms, loose morphisms, and cells that (strictly) preserves identities and composites of tight morphisms and cells. Every functor of \vdcs{} $F \colon \X \to \X'$ has an underlying functor $\u F \colon \u\X \to \u{\X'}$ of categories.
\end{definition}

At first, \cref{functor} may appear too strict, as we are asking for all structure to be preserved strictly. However, as we shall see in \cref{Dbl_l-in-VDbl}, such functors correspond to \emph{lax functors} between (non-virtual) \dcs. This illustrates the general principle we mentioned in the beginning of the section: that \vdcs{} provide a coherence-free approach to \dc{} theory.

\begin{definition}[{\cite[489]{dawson2006paths}}]
	\label{transformation}
	Let $F, G \colon \dc X \to \dc Y$ be functors of \vdcs. A \emph{transformation} $\Phi$ from $F$ to $G$ comprises the following data.
	\begin{enumerate}
		\item A tight morphism $\Phi_X \colon FX \to GX$ in $\Y$ for each object of $\X$.
		\item A cell $\Phi_p$ in $\Y$ for each loose morphism $p \colon X' \lto X$ of $\X$.
        \[\begin{tikzcd}
            FX & {FX'} \\
            GX & {GX'}
            \arrow[""{name=0, anchor=center, inner sep=0}, "{\Phi_X}"', from=1-1, to=2-1]
            \arrow["Fp"'{inner sep=.8ex}, "\shortmid"{marking}, from=1-2, to=1-1]
            \arrow[""{name=1, anchor=center, inner sep=0}, "{\Phi_{X'}}", from=1-2, to=2-2]
            \arrow["Gp"{inner sep=.8ex}, "\shortmid"{marking}, from=2-2, to=2-1]
            \arrow["{\Phi_p}"{description}, draw=none, from=0, to=1]
        \end{tikzcd}\]
	\end{enumerate}
    The family of tight morphisms is required to form a natural transformation $\u\Phi \colon \u F \tto \u G$, and the family of cells is required to be natural in the sense that, for each cell $\chi$ in $\X$ as follows,
    \[\begin{tikzcd}
        {X_0} & \cdots & {X_n} \\
        X && {X'}
        \arrow[""{name=0, anchor=center, inner sep=0}, "x"', from=1-1, to=2-1]
        \arrow["{p_1}"'{inner sep=.8ex}, "\shortmid"{marking}, from=1-2, to=1-1]
        \arrow["{p_n}"'{inner sep=.8ex}, "\shortmid"{marking}, from=1-3, to=1-2]
        \arrow[""{name=1, anchor=center, inner sep=0}, "{x'}", from=1-3, to=2-3]
        \arrow["p"{inner sep=.8ex}, "\shortmid"{marking}, from=2-3, to=2-1]
        \arrow["\chi"{description}, draw=none, from=0, to=1]
    \end{tikzcd}\]
    the following equation holds in $\Y$.
    \[
    \begin{tikzcd}[column sep=large]
        {FX_0} & \cdots & {FX_n} \\
        FX && {FX'} \\
        GX && {GX'}
        \arrow[""{name=0, anchor=center, inner sep=0}, "Fx"', from=1-1, to=2-1]
        \arrow["{Fp_1}"'{inner sep=.8ex}, "\shortmid"{marking}, from=1-2, to=1-1]
        \arrow["{Fp_n}"'{inner sep=.8ex}, "\shortmid"{marking}, from=1-3, to=1-2]
        \arrow[""{name=1, anchor=center, inner sep=0}, "{Fx'}", from=1-3, to=2-3]
        \arrow[""{name=2, anchor=center, inner sep=0}, "{\Phi_X}"', from=2-1, to=3-1]
        \arrow["Fp"{description}, from=2-3, to=2-1]
        \arrow[""{name=3, anchor=center, inner sep=0}, "{\Phi_{X'}}", from=2-3, to=3-3]
        \arrow["Gp"{inner sep=.8ex}, "\shortmid"{marking}, from=3-3, to=3-1]
        \arrow["{F\chi}"{description}, draw=none, from=0, to=1]
        \arrow["{\Phi_p}"{description}, draw=none, from=2, to=3]
    \end{tikzcd}
    \quad = \quad
    \begin{tikzcd}[column sep=large]
        {FX_0} & {FX_1} & {FX_{n - 1}} & {FX_n} \\
        {GX_0} & {GX_1} & {GX_{n - 1}} & {GX_n} \\
        GX &&& {GX'}
        \arrow[""{name=0, anchor=center, inner sep=0}, "{\Phi_{X_0}}"', from=1-1, to=2-1]
        \arrow["{Fp_1}"'{inner sep=.8ex}, "\shortmid"{marking}, from=1-2, to=1-1]
        \arrow[""{name=1, anchor=center, inner sep=0}, "{\Phi_{X_1}}"{description}, from=1-2, to=2-2]
        \arrow[""{name=2, anchor=center, inner sep=0}, "\cdots"{description}, draw=none, from=1-3, to=1-2]
        \arrow[""{name=3, anchor=center, inner sep=0}, "{\Phi_{X_{n - 1}}}"{description}, from=1-3, to=2-3]
        \arrow["{Fp_n}"'{inner sep=.8ex}, "\shortmid"{marking}, from=1-4, to=1-3]
        \arrow[""{name=4, anchor=center, inner sep=0}, "{\Phi_{X_n}}", from=1-4, to=2-4]
        \arrow[""{name=5, anchor=center, inner sep=0}, "Gx"', from=2-1, to=3-1]
        \arrow["{Gp_1}"{description}, from=2-2, to=2-1]
        \arrow[""{name=6, anchor=center, inner sep=0}, "\cdots"{description}, draw=none, from=2-3, to=2-2]
        \arrow["{Gp_n}"{description}, from=2-4, to=2-3]
        \arrow[""{name=7, anchor=center, inner sep=0}, "{Gx'}", from=2-4, to=3-4]
        \arrow["Gp"{inner sep=.8ex}, "\shortmid"{marking}, from=3-4, to=3-1]
        \arrow["{\Phi_{p_1}}"{description}, draw=none, from=0, to=1]
        \arrow["\cdots"{description}, draw=none, from=2, to=6]
        \arrow["{\Phi_{p_n}}"{description}, draw=none, from=3, to=4]
        \arrow["{G\chi}"{description}, draw=none, from=5, to=7]
    \end{tikzcd}
    \]
\end{definition}

\begin{proposition}[{\cite{dawson2006paths}\footnote{Note that \textcite{dawson2006paths} call \vdcs{} \emph{lax \dcs}, though, as discussed in \cite[\S4]{pare2011yoneda}, this name is better reserved for a notion of \dc{} in which loose morphisms may be composed, but only associatively and unitally up to noninvertible cells (called a \emph{slack \dc} in \cite{dawson2006paths}).}}]
    Virtual double categories, functors, and transformations form a 2-category $\VDbl$.
\end{proposition}

Consequently, we automatically obtain appropriate notions of adjunction and equivalence of \vdcs. In the following sections, we will make use of certain two-dimensional limits and colimits of \vdcs.

\begin{proposition}[{\cite[Remark~2.14]{dawson2006paths}}]
    \label{VDbl-is-bicomplete}
    $\VDbl$ is a complete and cocomplete 2-category.
\end{proposition}

It will be useful in particular to have an explicit description of copowers by the interval category $\b2 \defeq \{ 0 \to 1 \}$.

\begin{lemma}
    \label{copower-by-2}
    Let $\X$ be a \vdc{}. The copower $\b2 \copow \X$ is the \vdc{} described as follows.
    \begin{enumerate}
        \item An object is a pair $(i, X)$ where $i \in \{ 0, 1 \}$ and $X \in \X$.
        \item A tight morphism $(i, X)$ from $(j, Y)$, where $i \leq j$, is a tight morphism $X \to Y$ in $\X$; there are no tight morphisms between pairs with $i > j$.
        \item A loose morphism from $(i, X')$ to $(i, X)$ is a loose morphism $X' \lto X$ in $\X$; there are no loose morphisms between pairs with different indices.
        \item A cell is simply a cell with the same frame in $\X$ (forgetting the indices).
    \end{enumerate}
    The coprojections are given by the functors $\copi_i \colon \X \to \b2 \copow \X$, which are constantly $i$ on the first component and the identity on the second component, together with the canonical natural transformation $\copi_{<} \colon \copi_0 \tto \copi_1$ between them.
\end{lemma}

\begin{proof}
    The transformation $\copi_{<}$ induces a functor $\VDbl(\b2 \copow \X, \Y) \to \Cat(\b2, \VDbl(\X, \Y))$, which we must show to be invertible. Given a transformation $\Psi \colon F \tto G \colon \X \to \Y$, we define a functor $[\Psi] \colon \b2 \copow \X \to \Y$ by
    \begin{align*}
        [\Psi](0, X) & \defeq F(X) &
        [\Psi](1, X) & \defeq G(X) &
        [\Psi]({<}, X) & \defeq \Phi_X
    \end{align*}
    and analogously on the rest of the data. It is easy to verify that this defines a function $\Cat(\b2, \VDbl(\X, \Y))_0 \to \VDbl(\b2 \copow \X, \Y)_0$ inverse to the one induced by $\copi_{<}$, thus verifying the one-dimensional aspect of the universal property. Since we know from \cref{VDbl-is-bicomplete} that $\VDbl$ admits copowers, this suffices to identify the copower.
\end{proof}

\subsection{Pseudo \dcs}

While \vdcs{} do not admit composites of loose morphisms in general, we may characterise the existence of particular composites by a universal property, in the same way that we may characterise the existence of certain tensor products in a multicategory~\cite{hermida2000representable}.

\begin{definition}[{\cites[Definition~2.7]{dawson2006paths}[Definition~5.1]{cruttwell2010unified}}]
	\label{opcartesian}
    A cell
	\[\begin{tikzcd}
		\cdot & \cdots & \cdot \\
		\cdot && \cdot
		\arrow["{q_m}"', "\shortmid"{marking}, from=1-3, to=1-2]
		\arrow["{q_1}"', "\shortmid"{marking}, from=1-2, to=1-1]
		\arrow["q", "\shortmid"{marking}, from=2-3, to=2-1]
		\arrow[""{name=0, anchor=center, inner sep=0}, Rightarrow, no head, from=1-3, to=2-3]
		\arrow[""{name=1, anchor=center, inner sep=0}, Rightarrow, no head, from=1-1, to=2-1]
		\arrow["\opcart"{description}, draw=none, from=0, to=1]
	\end{tikzcd}\]
    in a \vdc{} is \emph{opcartesian} if any cell
	\[\begin{tikzcd}
		\cdot & \cdots & \cdot & \cdots & \cdot & \cdots & \cdot \\
		\cdot &&&&&& \cdot
		\arrow["{r_n}"', "\shortmid"{marking}, from=1-7, to=1-6]
		\arrow["{r_1}"', "\shortmid"{marking}, from=1-6, to=1-5]
		\arrow["{q_m}"', "\shortmid"{marking}, from=1-5, to=1-4]
		\arrow["{q_1}"', "\shortmid"{marking}, from=1-4, to=1-3]
		\arrow["{p_l}"', "\shortmid"{marking}, from=1-3, to=1-2]
		\arrow["{p_1}"', "\shortmid"{marking}, from=1-2, to=1-1]
		\arrow[""{name=0, anchor=center, inner sep=0}, "g", from=1-7, to=2-7]
		\arrow[""{name=1, anchor=center, inner sep=0}, "f"', from=1-1, to=2-1]
		\arrow["s", "\shortmid"{marking}, from=2-7, to=2-1]
		\arrow["\phi"{description}, draw=none, from=0, to=1]
	\end{tikzcd}\]
    factors uniquely therethrough:
	\[\begin{tikzcd}
		\cdot & \cdots & \cdot & \cdots & \cdot & \cdots & \cdot \\
		\cdot & \cdots & \cdot && \cdot & \cdots & \cdot \\
		\cdot &&&&&& \cdot
		\arrow["{r_n}"', "\shortmid"{marking}, from=1-7, to=1-6]
		\arrow["{r_1}"', "\shortmid"{marking}, from=1-6, to=1-5]
		\arrow["{q_m}"', "\shortmid"{marking}, from=1-5, to=1-4]
		\arrow["{q_1}"', "\shortmid"{marking}, from=1-4, to=1-3]
		\arrow["{p_l}"', "\shortmid"{marking}, from=1-3, to=1-2]
		\arrow["{p_1}"', "\shortmid"{marking}, from=1-2, to=1-1]
		\arrow["s", "\shortmid"{marking}, from=3-7, to=3-1]
		\arrow["q"{description}, from=2-5, to=2-3]
		\arrow[""{name=0, anchor=center, inner sep=0}, Rightarrow, no head, from=1-5, to=2-5]
		\arrow[""{name=1, anchor=center, inner sep=0}, "g", from=2-7, to=3-7]
		\arrow[""{name=2, anchor=center, inner sep=0}, "f"', from=2-1, to=3-1]
		\arrow[""{name=3, anchor=center, inner sep=0}, Rightarrow, no head, from=1-7, to=2-7]
		\arrow[""{name=4, anchor=center, inner sep=0}, Rightarrow, no head, from=1-1, to=2-1]
		\arrow["{r_n}"{description}, from=2-7, to=2-6]
		\arrow["{r_1}"{description}, from=2-6, to=2-5]
		\arrow["{p_l}"{description}, from=2-3, to=2-2]
		\arrow["{p_1}"{description}, from=2-2, to=2-1]
		\arrow[""{name=5, anchor=center, inner sep=0}, Rightarrow, no head, from=1-3, to=2-3]
		\arrow["{=}"{description}, draw=none, from=3, to=0]
		\arrow["{=}"{description}, draw=none, from=5, to=4]
		\arrow["\opcart"{description}, draw=none, from=0, to=5]
		\arrow["\check\phi"{description}, draw=none, from=1, to=2]
	\end{tikzcd}\]
    In this case, we call $q$ the \emph{(loose) composite} $q_1 \odot \cdots \odot q_m$ of $q_1, \ldots, q_m$.
	When $m = 0$, we call $q \colon A \lto A$ the \emph{loose identity} and denote it by $A(1, 1)$, or simply by $=\!\!\!|\!\!\!=$ in pasting diagrams.

	We denote a nullary cell with loose identity codomain by $\phi \colon f \tto g$.
	\[\begin{tikzcd}
		\cdot & \cdot \\
		\cdot & \cdot
		\arrow["\shortmid"{marking}, Rightarrow, no head, from=2-2, to=2-1]
		\arrow[""{name=0, anchor=center, inner sep=0}, "g", from=1-2, to=2-2]
		\arrow[""{name=1, anchor=center, inner sep=0}, "f"', from=1-1, to=2-1]
		\arrow[Rightarrow, no head, from=1-2, to=1-1]
		\arrow["\phi"{description}, draw=none, from=0, to=1]
	\end{tikzcd}\]
    A \vdc{} is \emph{representable} when it is equipped with a choice of loose composite for every chain $q_1, \ldots, q_m$ ($m \ge 0$).\footnote{Typically, representability merely asks for the \emph{existence} of loose composites (which are essentially unique), rather than for a \emph{choice}. We ask for a choice of loose composites to align more closely with the usual definition of \pdc{} (see \cref{Dbl_l-in-VDbl}).}
\end{definition}

Loose composites are unique up to unique globular isomorphism, and consequently loose composition is associative and unital up to coherent isomorphism, \eg{} we have $(p \odot q) \odot r \iso p \odot q \odot r \iso p \odot (q \odot r)$ when each of these composites exists. Loose composition is also functorial: given a chain of unary cells $\phi_1, \ldots, \phi_n$ (below left), assuming $p_1, \ldots, p_n$ and $q_1, \ldots, q_n$ admit loose composites, their universal properties induce a unary cell $\phi_1 \odot \cdots \odot \phi_n$ (below right). (The reader new to \vdcs{} is encouraged to try proving these properties directly themselves; abstractly, they follow from \cref{Dbl_l-in-VDbl}.)
\[
\begin{tikzcd}[column sep=large]
	\cdot & \cdot & \cdot & \cdot \\
	\cdot & \cdot & \cdot & \cdot
	\arrow[""{name=0, anchor=center, inner sep=0}, "{f_0}"', from=1-1, to=2-1]
	\arrow["{p_1}"'{inner sep=.8ex}, "\shortmid"{marking}, from=1-2, to=1-1]
	\arrow[""{name=1, anchor=center, inner sep=0}, "{f_1}"{description}, from=1-2, to=2-2]
	\arrow["\cdots"{description}, draw=none, from=1-3, to=1-2]
	\arrow[""{name=2, anchor=center, inner sep=0}, "{f_{n - 1}}"{description}, from=1-3, to=2-3]
	\arrow["{p_n}"'{inner sep=.8ex}, "\shortmid"{marking}, from=1-4, to=1-3]
	\arrow[""{name=3, anchor=center, inner sep=0}, "{f_n}", from=1-4, to=2-4]
	\arrow["{q_1}"{inner sep=.8ex}, "\shortmid"{marking}, from=2-2, to=2-1]
	\arrow["\cdots"{description}, draw=none, from=2-3, to=2-2]
	\arrow["{q_n}"{inner sep=.8ex}, "\shortmid"{marking}, from=2-4, to=2-3]
	\arrow["{\phi_1}"{description}, draw=none, from=1, to=0]
	\arrow["{\phi_n}"{description}, draw=none, from=3, to=2]
\end{tikzcd}
\quad \mapsto \quad
\begin{tikzcd}[column sep=7em]
	\cdot & \cdot \\
	\cdot & \cdot
	\arrow[""{name=0, anchor=center, inner sep=0}, "{f_0}"', from=1-1, to=2-1]
	\arrow["{p_1 \odot \cdots \odot p_n}"'{inner sep=.8ex}, "\shortmid"{marking}, from=1-2, to=1-1]
	\arrow[""{name=1, anchor=center, inner sep=0}, "{f_n}", from=1-2, to=2-2]
	\arrow["{p_1 \odot \cdots \odot q_n}"{inner sep=.8ex}, "\shortmid"{marking}, from=2-2, to=2-1]
	\arrow["{\phi_1 \odot \cdots \odot \phi_n}"{description}, draw=none, from=1, to=0]
\end{tikzcd}
\]

\begin{definition}[{\cites[Definition~2.7]{dawson2006paths}[Remark~5.8]{cruttwell2010unified}}]
    If the universal property of \cref{opcartesian} holds only for $l = 0$ and $n = 0$, we say that the cell is \emph{weakly opcartesian}, and that $q$ is a \emph{weak loose composite} or \emph{weak loose identity}. A \vdc{} is \emph{weakly representable} when it is equipped with a choice of weak loose composites for every chain of loose morphisms.
\end{definition}

Weak loose composites are also unique up to unique globular isomorphism, but weak loose composition is neither associative nor unital (see \cref{colax-dcs}).

\begin{remark}[Terminological warning]
    The term \emph{representability} is used in two senses in this paper. When speaking of \pdcs{}, it is in the sense of \emph{representable presheaves}; whereas, when speaking of \vdcs, it is in the sense of \cref{opcartesian}, \ie{} the existence of loose composites.
\end{remark}

If a \vdc{} admits loose identities, it has an \emph{underlying 2-category}, whose 2-cells are given by nullary cells with loose identity codomain~\cite[Proposition~6.1]{cruttwell2010unified}.

Pseudo \dcs{} and their morphisms are typically defined explicitly in terms of pseudocategories in the 2-category $\Cat$~\cites[\S7.1]{grandis1999limits}[Example~5.13]{arkor2024enhanced}. For our purposes, it will be more useful to view \pdcs{} as certain kinds of \vdcs, which is justified by the following result of \textcite{dawson2006paths}.

\begin{theorem}
    \label{Dbl_l-in-VDbl}
    Denote by $\Dbl_l$ the 2-category of \pdcs{}, lax functors, and transformations~\cite[\S7.1 -- 7.3]{grandis1999limits}. There is a \ff{} 2-functor $\Dbl_l \ffto \VDbl$ whose image comprises the representable \vdcs{}.
\end{theorem}

\begin{proof}
    The 2-category $\Dbl_l$ is the Kleisli 2-category for a lax-idempotent 2-comonad on $\Dbl_s$~\cite[Proposition~1.19]{dawson2006paths}, whereas the 2-category $\VDbl$ is the 2-category of strict coalgebras~\cite[Theorem~2.13]{dawson2006paths}. The image of the \ff{} inclusion of the Kleisli 2-category into the 2-category of coalgebras comprises the representable \vdcs{}~\cite[Proposition~2.8]{dawson2006paths}.
\end{proof}

\begin{remark}
    \label{colax-dcs}
    Similarly, \emph{colax \dcs} may be identified with the \vdcs{} that are weakly representable~\cites[\S2.9]{dawson2006paths}[Example~9.7]{cruttwell2010unified}. A colax \dc{} is a two-dimensional structure like a \pdc{}, but for which composition is neither associative nor unital. Instead, there is an $n$-ary composition operation for each $n \in \N$, and globular cells that introduce parenthesisation: for instance, given composable loose morphisms $p, q, r$ in a colax \dc, there are globular cells
    \[(p \odot q) \odot r \tfrom p \odot q \odot r \tto p \odot (q \odot r)\]
    but no cell between $(p \odot q) \odot r$ and $p \odot (q \odot r)$ in general.

    Examples of colax \dcs{} typically arise in the presence of colimit structure that is not preserved by composition. For instance, for every monoidal category $\b V$, there is a \vdc{} $\b V\h\dc{Mat}$ of sets and $\b V$-enriched matrices~\cite[Example~2.6]{cruttwell2010unified}. When $\b V$ is cocomplete and the tensor product preserves colimits in each variable, then $\b V\h\dc{Mat}$ is representable. However, if $\b V$ is cocomplete, but these colimits are not preserved by the tensor product, $\b V\h\dc{Mat}$ is merely weakly representable~\cite[Remark~5.8]{cruttwell2010unified}. We shall see another natural example of a colax \dc{} in \cref{power-of-Set}.
\end{remark}

In this paper, we \underline{define} a \emph{\pdc} to be a representable \vdc{} and a \emph{lax functor} to be a functor of \vdcs{} between representable \vdcs.

\begin{example}
    The \vdc{} $\Sigma\M$ arising from a multicategory $\M$ is representable if and only if $\M$ is representable in the sense of \textcite{hermida2000representable}, \ie{} it is induced by a monoidal category.
\end{example}

\begin{example}[{\cite[\S3.2]{grandis1999limits}}]
    $\Span$ is a \pdc. The loose identity associated to a set $A$ is the identity span $A =\!= A =\!= A$, and the loose composite of two spans is given by the pullback of the inner cospan. More generally, for any category $\E$ with pullbacks, $\Span(\E)$ is a \pdc.
\end{example}

\begin{example}[{\cite[\S3.1]{grandis1999limits}}]
    $\Dist$ is a \pdc. The loose identity associated to a small category $\b A$ is the hom distributor $\b A({-}, {-}) \colon \b A\op \times \b A \to \Set$, and the loose composite of two distributors $p \colon \b B \lto \b A$ and $q \colon \b C \lto \b B$ is given by the following coend.
    \[p \odot q \defeq \int^{b \in \b B} p({-}, b) \odot q(b, {-})\]
    More generally, for any category $\E$ with pullbacks and reflexive coequalisers that are stable under pullback, $\Dist(\E)$ is a \pdc{} (\eg{} see \cite[Proposition~2.43]{johnstone1977topos}).
\end{example}

\begin{example}
    For any \pdc{} $\A$, we may define a \pdc{} $\A\opt$, the \emph{tight opposite} of $\A$, whose underlying category is $(\u\A)\op$, whose loose morphisms are the same as those of $\A$, and in which a cell with the following frame in $\A\opt$
    \[\begin{tikzcd}
        {A_0} & \cdots & {A_n} \\
        B && {B'}
        \arrow["f"', from=1-1, to=2-1]
        \arrow["{p_1}"'{inner sep=.8ex}, "\shortmid"{marking}, from=1-2, to=1-1]
        \arrow["{p_n}"'{inner sep=.8ex}, "\shortmid"{marking}, from=1-3, to=1-2]
        \arrow["{f'}", from=1-3, to=2-3]
        \arrow["q"{inner sep=.8ex}, "\shortmid"{marking}, from=2-3, to=2-1]
    \end{tikzcd}\]
    is given by a cell with the following frame in $\A$.
    \[\begin{tikzcd}[column sep=large]
        B && {B'} \\
        {A_0} && {A_n}
        \arrow["f"', from=1-1, to=2-1]
        \arrow["q"'{inner sep=.8ex}, "\shortmid"{marking}, from=1-3, to=1-1]
        \arrow["{f'}", from=1-3, to=2-3]
        \arrow["{p_1 \odot \cdots \odot p_n}"{inner sep=.8ex}, "\shortmid"{marking}, from=2-3, to=2-1]
    \end{tikzcd}\qedshift\]
\end{example}

\subsection{Virtual equipments}
\label{virtual-equipments}

One of the motivating applications of \vdcs{} is to \fct, which is the axiomatic study of \emph{category theories} (such as enriched, internal, and fibred category theory). In practice, to effectively carry out \fct, it is necessary to assume some additional structure on a \vdc, allowing loose morphisms to be restricted along tight morphisms.

\begin{definition}[{\cite[Definition~7.1]{cruttwell2010unified}}]
	\label{cartesian-cell}
    A 2-cell
	\[\begin{tikzcd}
		\cdot & \cdot \\
		\cdot & \cdot
		\arrow["p"', "\shortmid"{marking}, from=1-2, to=1-1]
		\arrow["q", "\shortmid"{marking}, from=2-2, to=2-1]
		\arrow[""{name=0, anchor=center, inner sep=0}, "g", from=1-2, to=2-2]
		\arrow[""{name=1, anchor=center, inner sep=0}, "f"', from=1-1, to=2-1]
		\arrow["\cart"{description}, draw=none, from=0, to=1]
	\end{tikzcd}\]
    in a \vdc{} is \emph{cartesian} if any 2-cell
	\[\begin{tikzcd}
		\cdot & \cdots & \cdot \\
		\cdot && \cdot
		\arrow["q", "\shortmid"{marking}, from=2-3, to=2-1]
		\arrow[""{name=0, anchor=center, inner sep=0}, "{g' \d g}", from=1-3, to=2-3]
		\arrow[""{name=1, anchor=center, inner sep=0}, "{f' \d f}"', from=1-1, to=2-1]
		\arrow["{r_n}"', "\shortmid"{marking}, from=1-3, to=1-2]
		\arrow["{r_1}"', "\shortmid"{marking}, from=1-2, to=1-1]
		\arrow["\phi"{description}, draw=none, from=0, to=1]
	\end{tikzcd}\]
    factors uniquely therethrough:
	\[\begin{tikzcd}
		\cdot & \cdots & \cdot \\
		\cdot && \cdot \\
		\cdot && \cdot
		\arrow["q", "\shortmid"{marking}, from=3-3, to=3-1]
		\arrow["{r_n}"', "\shortmid"{marking}, from=1-3, to=1-2]
		\arrow[""{name=0, anchor=center, inner sep=0}, "g", from=2-3, to=3-3]
		\arrow[""{name=1, anchor=center, inner sep=0}, "f"', from=2-1, to=3-1]
		\arrow["p"{description}, from=2-3, to=2-1]
		\arrow[""{name=2, anchor=center, inner sep=0}, "{g'}", from=1-3, to=2-3]
		\arrow[""{name=3, anchor=center, inner sep=0}, "{f'}"', from=1-1, to=2-1]
		\arrow["{r_1}"', "\shortmid"{marking}, from=1-2, to=1-1]
		\arrow["\cart"{description}, draw=none, from=0, to=1]
		\arrow["\hat\phi"{description}, draw=none, from=2, to=3]
	\end{tikzcd}\]
    In this case, we call $p$ the \emph{restriction} $q(f, g)$.
    If $q$ is a loose identity $A(1, 1)$, we denote $p = A(1, 1)(f, g)$ simply by $A(f, g)$.
\end{definition}

Restrictions are unique up to unique globular isomorphism and are consequently pseudo functorial: given a loose-cell $p \colon D \lto C$ and tight-cells $f, f' \colon A \to C$ and $g, g' \colon B \to D$, each pair of 2-cells $\phi \colon f' \tto f$ and $\gamma \colon g \tto g'$ induces a 2-cell $p(\phi, \gamma) \colon p(f, g) \tto p(f', g')$, assuming both restrictions exist.
\[
\Bigg(\;
\begin{tikzcd}
	A & A \\
	C & C
	\arrow[equals, from=1-1, to=1-2]
	\arrow[""{name=0, anchor=center, inner sep=0}, "{f'}"', from=1-1, to=2-1]
	\arrow[""{name=1, anchor=center, inner sep=0}, "f", from=1-2, to=2-2]
	\arrow["\shortmid"{marking}, equals, from=2-1, to=2-2]
	\arrow["\phi"{description}, draw=none, from=0, to=1]
\end{tikzcd}
\quad,\quad
\begin{tikzcd}
	B & B \\
	D & D
	\arrow[equals, from=1-1, to=1-2]
	\arrow[""{name=0, anchor=center, inner sep=0}, "g"', from=1-1, to=2-1]
	\arrow[""{name=1, anchor=center, inner sep=0}, "{g'}", from=1-2, to=2-2]
	\arrow["\shortmid"{marking}, equals, from=2-1, to=2-2]
	\arrow["\gamma"{description}, draw=none, from=0, to=1]
\end{tikzcd}
\;\Bigg)
\qquad\mapsto\qquad
\begin{tikzcd}[column sep=large]
	A & B \\
	A & B
	\arrow[""{name=0, anchor=center, inner sep=0}, equals, from=1-1, to=2-1]
	\arrow["{p(f, g)}"'{inner sep=.8ex}, "\shortmid"{marking}, from=1-2, to=1-1]
	\arrow[""{name=1, anchor=center, inner sep=0}, equals, from=1-2, to=2-2]
	\arrow["{p(f', g')}"{inner sep=.8ex}, "\shortmid"{marking}, from=2-2, to=2-1]
	\arrow["{p(\phi, \gamma)}"{description}, draw=none, from=0, to=1]
\end{tikzcd}
\]

\emph{Virtual equipments} were introduced by \textcite{cruttwell2010unified} to axiomatise the structure fundamental to the study of \fct{} and have proven an effective formalism for that purpose~\cite{cruttwell2010unified,koudenburg2024formal,arkor2024formal,arkor2024relative,arkor2025nerve}.

\begin{definition}[{\cite[Definition~7.6]{cruttwell2010unified}}]
    \label{virtual-equipment}
    A \emph{\ve} is a \vdc{} that admits restrictions and loose identities.
\end{definition}

\begin{example}[{\cite[Examples~7.3]{cruttwell2010unified}}]
    $\Span$ is a \ve. Given a diagram of sets,
    \[\begin{tikzcd}
        A && B \\
        C & Y & D
        \arrow["f"', from=1-1, to=2-1]
        \arrow["g", from=1-3, to=2-3]
        \arrow["c", from=2-2, to=2-1]
        \arrow["d"', from=2-2, to=2-3]
    \end{tikzcd}\]
    its limit exhibits the restriction. More generally, for any category $\E$ with pullbacks, $\Span(\E)$ is a \ve, restrictions being formed in the same way.
\end{example}

\begin{example}[{\cite[Examples~7.3]{cruttwell2010unified}}]
    $\Dist$ is a \ve. Given a diagram of categories and distributors,
    \[\begin{tikzcd}
        {\b A} & {\b B} \\
        {\b C} & {\b D}
        \arrow["f"', from=1-1, to=2-1]
        \arrow["g", from=1-2, to=2-2]
        \arrow["p"{inner sep=.8ex}, "\shortmid"{marking}, from=2-2, to=2-1]
    \end{tikzcd}\]
    the distributor $p(f, g) \defeq (a, b) \mapsto p(fa, gb)$ exhibits the restriction. More generally, for any category $\E$ with pullbacks, $\Dist(\E)$ is a \ve, restrictions being formed in the same way.
\end{example}

\subsection{Fully faithful functors}

We shall be concerned throughout the paper with full faithfulness of functors of \vdcs.

\begin{definition}
    \label{ff}
    A functor of \vdcs{} $F \colon \X \to \Y$ is \emph{\ff} if its underlying functor $\u F \colon \u\X \to \u\Y$ of categories is \ff{} and if, for each frame in $\X$ as on the left below,
    \[
    \begin{tikzcd}
        {X_0} & \cdots & {X_n} \\
        X && {X'}
        \arrow["x"', from=1-1, to=2-1]
        \arrow["{p_1}"'{inner sep=.8ex}, "\shortmid"{marking}, from=1-2, to=1-1]
        \arrow["{p_n}"'{inner sep=.8ex}, "\shortmid"{marking}, from=1-3, to=1-2]
        \arrow["{x'}", from=1-3, to=2-3]
        \arrow["p"{inner sep=.8ex}, "\shortmid"{marking}, from=2-3, to=2-1]
    \end{tikzcd}
    \hspace{5em}
    \begin{tikzcd}
        {FX_0} & \cdots & {FX_n} \\
        FX && {FX'}
        \arrow["Fx"', from=1-1, to=2-1]
        \arrow["{Fp_1}"'{inner sep=.8ex}, "\shortmid"{marking}, from=1-2, to=1-1]
        \arrow["{Fp_n}"'{inner sep=.8ex}, "\shortmid"{marking}, from=1-3, to=1-2]
        \arrow["{Fx'}", from=1-3, to=2-3]
        \arrow["Fp"{inner sep=.8ex}, "\shortmid"{marking}, from=2-3, to=2-1]
    \end{tikzcd}
    \]
    the function sending each cell $\chi$ in $\X$ with the frame on the left above to a cell $F\chi$ in $\Y$ exhibits a bijection between cells in $\X$ with the frame on the left above, and cells in $\Y$ with the frame on the right above.
\end{definition}

Abstractly, a functor of \vdcs{} is \ff{} in the sense of \cref{ff} if and only if it is \ff{} as a tight morphism in the \ve{} of \vdcs{}~\cite[\S8]{cruttwell2010unified}. In particular, every \ff{} functor of \vdcs{} $F \colon \X \to \Y$ is representably \ff{} as a 1-cell in the 2-category $\VDbl$, \ie{} for each \vdc{} $\W$, the postcomposition functor $\VDbl(\W, F) \colon \VDbl(\W, \X) \to \VDbl(\W, \Y)$ is \ff. In general, the converse does not hold: conceptually, representable \ffness{} does not capture \ffness{} with respect to non-unary cells, because the data of a transformation (\cref{transformation}) only involves assignments of unary cells.

\begin{example}
    Denote by $\X$ the \vdc{} with a unique object, unique tight morphism, unique loose morphism, unique nullary cell, and unique unary cell, and no $n$-ary cells for $n > 1$. Denote by $\Y$ the \vdc{} with a unique object, unique tight morphism, unique loose morphism, and unique unary cell, two nullary cells $\alpha$ and $\beta$, and no $n$-ary cells for $n > 1$. There is a unique functor of \vdcs{} $A \colon \X \to \Y$ that sends the unique nullary cell of $\X$ to $\alpha$. This is not \ff, since the assignment on nullary cells is not bijective. However, it is representably \ff. To see this, observe that, since $\X$ is subterminal, given a \vdc{} $\W$, there is at most one functor $F \colon \W \to \X$ (a functor exists if and only if $\W$ has no $n$-ary cells for $n > 1$). Since $\X$ and $\Y$ both have no non-identity unary cells, there is exactly one transformation $F \tto F$ and exactly one transformation $AF \tto AF$. Thus $A$ is trivially representably \ff.
\end{example}

Despite the distinction between the notions of \ffness{} and of representable \ffness, there is a specific instance of interest in which they coincide. Recall from \cref{background-on-vdcs} that an adjunction of \vdcs{} is, by definition, an adjunction in the 2-category $\VDbl$ or, equivalently, an adjunction in the virtual equipment of \vdcs~\cite[Corollary~5.6]{arkor2024formal}. The following establishes that right adjoints between \vdcs{} are \ff{} if and only if they are representably \ff.

\begin{lemma}
    \label{coreflection}
    Let $\ell \adj r \colon B \to A$ be an adjunction in a \ve{}. The following are equivalent.
    \begin{enumerate}
        \item The counit of the adjunction is invertible.
        \item The right adjoint $r$ is \ff{}, in the sense that the canonical cell ${} \tto A(r, r)$ induced by restriction is opcartesian~\cite[Definition~3.27]{arkor2024formal}.
        \item The right adjoint $r$ is representably \ff{} in the underlying 2-category.
    \end{enumerate}
\end{lemma}

\begin{proof}
    (1 $\iff$ 2) follows from \cite[Lemma~2.23]{arkor2024formal} applied to the loose adjunction $B(1, \ell) \adj A(1, r)$ induced by $\ell \adj r$.
    (1 $\iff$ 3) follows by representability from the corresponding statement in $\Cat$, since the representably \ff{} 1-cells in $\Cat$ are precisely the \ff{} functors.
\end{proof}

\section{Exponentiable \vdcs}
\label{exponentiability}

Our investigation begins with the study of exponentiability in the 2-category $\VDbl$ of \vdcs{}. As we shall see, $\VDbl$ is not cartesian closed (\cref{VDbl-is-not-CC}), but there are nevertheless several interesting classes of exponentiable objects. To analyse exponentials of \vdcs, it will be useful to first spend a little time studying the relationship between \vdcs{} and graphs internal to categories.

\begin{definition}
    Denote by $\G \defeq \{ G_1 \rightrightarrows G_0 \}$ the free-standing parallel pair.
    \begin{enumerate}
        \item $\Grph \defeq [\G, \Set]$ is the category of graphs.
        \item $\Grph(\Cat) \defeq [\G, \Cat]$ is the 2-category of graphs internal to categories.
        \qedhere
    \end{enumerate}
\end{definition}

There are two perspectives on graphs internal to categories, which we make explicit to avoid potential confusion about our notational conventions.

\begin{proposition}
    The 2-category $\Grph(\Cat)$ is isomorphic to the 2-category $\Cat(\Grph)$ of categories internal to $\Grph$.
\end{proposition}

\begin{proof}
    Follows from the fact that limits in functor categories are computed componentwise. Abstractly, this is an instance of the symmetry of internalisation for two-dimensional limit sketches~\cite[Theorem~7.5]{arkor2024enhanced}.
\end{proof}

Explicitly, the objects of $\Grph(\Cat)$ comprise graphs of categories (below left), where $\A_0$ is a category of \emph{objects} and \emph{tight morphisms}, and $\A_1$ is a category of \emph{loose morphisms} and \emph{unary cells}. On the other hand, objects of $\Cat(\Grph)$ comprise internal categories (below right), where $(\A^\top)_0$ is a graph of objects and loose morphisms, and $(\A^\top)_1$ is a graph of tight morphisms and unary cells. The isomorphism between the two 2-categories may therefore be seen as a kind of transposition operation, analogous to that for strict \dcs{}~\cite[\S1.2]{grandis1999limits}.
\[
\begin{tikzcd}
	& {\A_1} \\
	{\A_0} && {\A_0}
	\arrow["{s_\A}"', from=1-2, to=2-1]
	\arrow["{t_\A}", from=1-2, to=2-3]
\end{tikzcd}
\hspace{6em}
\begin{tikzcd}
	& {(\A^\top)_1} \\
	{(\A^\top)_0} && {(\A^\top)_0}
	\arrow["{s_{(\A^\top)}}"', from=1-2, to=2-1]
	\arrow["{t_{(\A^\top)}}", from=1-2, to=2-3]
\end{tikzcd}
\]
In this paper, we will focus on the perspective of $\Grph(\Cat)$; our notation is chosen to reflect this, as well as to match the typical convention for \pdcs{} viewed as pseudocategories.

Every \vdc{} $\X$ has an underlying graph $\X\down$ internal to categories, given by forgetting the non-unary cells. Conversely, we may view a graph $\A$ internal to categories as a \vdc{} $\A\up$ with no non-unary cells. The precise relationship between categories, graphs internal to categories, and \vdcs{} is summarised by the following pair of adjunctions.\footnote{Abstractly, these adjunctions may be seen to arise from change of shape for generalised multicategories (see \cite[\S6.7]{leinster2004higher} and also \cref{generalised-multicategories}).}

\begin{proposition}
    \label{Cat-graphs-into-VDCs}
    There are coreflective 2-adjunctions as follows.
    \[
    \begin{tikzcd}[column sep=large]
        \Cat & {\Grph(\Cat)} & \VDbl
        \arrow[""{name=0, anchor=center, inner sep=0}, "{\Dis\ph}", shift left=2, hook, from=1-1, to=1-2]
        \arrow[""{name=1, anchor=center, inner sep=0}, "{\ph_0}", shift left=2, from=1-2, to=1-1]
        \arrow[""{name=2, anchor=center, inner sep=0}, "{\ph\up}", shift left=2, hook, from=1-2, to=1-3]
        \arrow[""{name=3, anchor=center, inner sep=0}, "{\ph\down}", shift left=2, from=1-3, to=1-2]
        \arrow["\dashv"{anchor=center, rotate=-90}, draw=none, from=0, to=1]
        \arrow["\dashv"{anchor=center, rotate=-90}, draw=none, from=2, to=3]
    \end{tikzcd}
    \qquad = \qquad
    \begin{tikzcd}
        \Cat & \VDbl
        \arrow[""{name=0, anchor=center, inner sep=0}, "{\o\ph}", shift left=2, hook, from=1-1, to=1-2]
        \arrow[""{name=1, anchor=center, inner sep=0}, "{\u\ph}", shift left=2, from=1-2, to=1-1]
        \arrow["\dashv"{anchor=center, rotate=-90}, draw=none, from=0, to=1]
    \end{tikzcd}
    \]
\end{proposition}

\begin{proof}
    The existence of the leftmost 2-adjunction follows from initiality of the empty category $\b0$. The rightmost 2-adjunction
    $\VDbl(\A\up, \X) \iso \Grph(\Cat)(\A, \X\down)$ is elementary. In both cases, the right adjoints are retractions of the left adjoints.
\end{proof}

\begin{remark}
    Neither $\Dis\ph \colon \Cat \to \Grph(\Cat)$ nor $\o\ph \colon \Cat \to \VDbl$ admits a left adjoint, since neither preserves the terminal object. We will show in \cref{chaotic-vdc} that $\u\ph \colon \VDbl \to \Cat$ admits a right adjoint.
\end{remark}

It follows from \cref{Cat-graphs-into-VDCs} that the forgetful functor $\u{\ph} \colon \VDbl \to \Cat$ is given by $\u{\ph} \iso \Cat(\b1, \u{\ph}) \iso \VDbl(\o{\b1}, {-})$, where $\b1$ is the terminal category\footnote{This is also observed in \cite[Remark~6.1]{fujii2025familial}.}. In what follows, we may refer to the underlying category of a \vdc{} $\X$ as either $\u\X$ or $\X_0$ (thereby leaving $\ph\down$ implicit), depending on the intuition we wish to convey.

\begin{definition}
    A \vdc{} is \emph{unary} if each of its cells is unary, \ie{} it is in the image of $\ph\up \colon \Grph(\Cat) \ffto \VDbl$.
\end{definition}

While exponentiability of general \vdcs{} is subtle, \emph{unary} \vdcs{} are much better behaved.

\begin{proposition}
    \label{Grph(Cat)-is-CC}
    The 2-category $\Grph(\Cat)$ is cartesian closed. Explicitly, given unary \vdcs{} $\A$ and $\X$, the exponential $\X^\A$ in $\Grph(\Cat)$ is given by the following unary \vdc{}.
    \begin{itemize}
        \item The underlying category is the functor category ${\X_0}^{\A_0}$.
        \item \label{graph-morphism} A loose morphism from $F$ to $G$ is a graph morphism.
        \[\begin{tikzcd}
            {\A_0} & {\A_1} & {\A_0} \\
            {\X_0} & {\X_1} & {\X_0}
            \arrow["F"', from=1-1, to=2-1]
            \arrow["s"', from=1-2, to=1-1]
            \arrow["t", from=1-2, to=1-3]
            \arrow["\Xi"{description}, from=1-2, to=2-2]
            \arrow["G", from=1-3, to=2-3]
            \arrow["s", from=2-2, to=2-1]
            \arrow["t"', from=2-2, to=2-3]
        \end{tikzcd}\]
        Explicitly, this comprises the following data.
        \begin{enumerate}
            \item For each loose morphism $p \colon X \lto Y$ in $\A$, a loose morphism $\Xi p \colon FX \lto GY$.
            \item For each cell in $\A$,
            \[\begin{tikzcd}
                Y & X \\
                {Y'} & {X'}
                \arrow[""{name=0, anchor=center, inner sep=0}, "y"', from=1-1, to=2-1]
                \arrow["p"'{inner sep=.8ex}, "\shortmid"{marking}, from=1-2, to=1-1]
                \arrow[""{name=1, anchor=center, inner sep=0}, "x", from=1-2, to=2-2]
                \arrow["{p'}"{inner sep=.8ex}, "\shortmid"{marking}, from=2-2, to=2-1]
                \arrow["\alpha"{description}, draw=none, from=1, to=0]
            \end{tikzcd}\]
            a cell in $\X$.
            \[\begin{tikzcd}
                GY & FX \\
                {GY'} & {FX'}
                \arrow[""{name=0, anchor=center, inner sep=0}, "Gy"', from=1-1, to=2-1]
                \arrow["{\Xi p}"'{inner sep=.8ex}, "\shortmid"{marking}, from=1-2, to=1-1]
                \arrow[""{name=1, anchor=center, inner sep=0}, "Fx", from=1-2, to=2-2]
                \arrow["{\Xi p'}"{inner sep=.8ex}, "\shortmid"{marking}, from=2-2, to=2-1]
                \arrow["{\Xi \alpha}"{description}, draw=none, from=1, to=0]
            \end{tikzcd}\]
        \end{enumerate}
        These are required to satisfy the following laws.
        \begin{enumerate}[resume]
            \item $\Xi(1_p) = 1_{\Xi(p)}$.
            \item $\Xi(\alpha \d \alpha') = \Xi(\alpha) \d \Xi(\alpha')$.
        \end{enumerate}
        \item A cell with the following frame
        \[\begin{tikzcd}
            G & F \\
            {G'} & {F'}
            \arrow[""{name=0, anchor=center, inner sep=0}, "\Gamma"', from=1-1, to=2-1]
            \arrow["\Xi"'{inner sep=.8ex}, "\shortmid"{marking}, from=1-2, to=1-1]
            \arrow[""{name=1, anchor=center, inner sep=0}, "\Phi", from=1-2, to=2-2]
            \arrow["{\Xi'}"{inner sep=.8ex}, "\shortmid"{marking}, from=2-2, to=2-1]
            \arrow["\xi"{description}, draw=none, from=0, to=1]
        \end{tikzcd}\]
        is a natural transformation $\xi \colon \Xi \tto \Xi'$ that commutes with $\Phi$ and $\Gamma$.
    \end{itemize}
\end{proposition}

\begin{proof}
    $\Grph(\Cat)$ is a presheaf 2-category, hence is cartesian closed. The explicit description of the exponential is given by evaluating the exponentials in the presheaf 2-category using the Yoneda lemma.
\end{proof}

Our strategy for studying exponentiability of general \vdcs{} stems from the observation that exponentials of \vdcs{} are closely related to exponentials of unary \vdcs.

\begin{definition}
    A 2-adjunction $L \adj R$ between cartesian monoidal 2-categories satisfies the \emph{reciprocity condition} if the canonical 1-cell $\tp{\varepsilon_X \c L\pi_1, L\pi_2} \colon L(RX \times Y) \to X \times LY$ is invertible for all objects $X$ and $Y$.
\end{definition}

Given a 2-adjunction $L \adj R$, it is a straightforward exercise to show that $R$ preserves those exponentials that exist if the 2-adjunction satisfies the reciprocity condition, and that the converse holds assuming the domain of $R$ is cartesian closed (\cf~\cite[6]{lawvere1970equality}).

\begin{lemma}
    \label{reciprocity}
    The 2-adjunctions of \Cref{Cat-graphs-into-VDCs} satisfy the reciprocity condition. Consequently, the 2-functors $\ph\down \colon \VDbl \to \Grph(\Cat)$ and $\ph_0 \colon \Grph(\Cat) \to \Cat$, and their composite $\u\ph \colon \VDbl \to \Cat$, preserve exponentials.
\end{lemma}

\begin{proof}
    The reciprocity condition for the first adjunction states that the canonical functor $(\X\down \times \A)\up \to \X \times \A\up$ is invertible for each \vdc{} $\X$ and unary \vdc{} $\A$. Since the $n$-ary cells in a product are given by pairs of $n$-ary cells in each component, and $\A\up$ has only unary cells, this is clear. The reciprocity condition for the second adjunction states that the canonical functor $(\Dis{(\X_0 \times \b A)}) \to \X \times (\Dis{\b A})$ is invertible for each unary \vdc{} $\X$ and category $\b A$. This is clear for essentially the same reason, since $\Dis{\b A}$ has no cells whatsoever.
\end{proof}

In particular, observe that, upon taking $\b A = \b1$ to be the terminal category in the reciprocity condition for the 2-adjunction $\o\ph \adj \u\ph \colon \VDbl \to \Cat$, we have that $\X \times \o{\b1} \iso \o{(\u\X)}$ for every \vdc{} $\X$, \ie{} taking the product of any \vdc{} $\X$ with the \vdc{} $\o{\b1}$ having a single tight morphism and no loose morphisms forgets the loose morphisms in $\X$.

We are now ready to prove the main theorem of this section. As we have mentioned (and will substantiate in \cref{VDbl-is-not-CC}), $\VDbl$ is not cartesian closed. However, there is an important class of exponentiable objects: namely, the \pdcs.

\begin{theorem}
    \label{representable-vdcs-are-exponentiable}
    Representable \vdcs{} are exponentiable, \ie{} for every representable \vdc{} $\A$, the 2-functor $\ph \times \A \colon \VDbl \to \VDbl$ admits a right adjoint $\ph^\A \colon \VDbl \to \VDbl$.
\end{theorem}

\begin{proof}
    Let $\A$ be a \pdc{} and let $\X$ be a \vdc. We shall define a \vdc{} $\X^\A$, guided by \cref{reciprocity}, which tells us that the unary \vdc{} $(\X^\A)\down$ underlying $\X^\A$ must be $(\X\down)^{(\A\down)}$. The definition of the multiary cells may be determined by considering functors into $\X^\A$ from the free-standing $n$-ary cell for each $n \in \N$. Thus, it is only necessary to exercise creativity in coming up with the appropriate definition of composition; the remaining data is determined.
    \begin{enumerate}
        \item The underlying category is given by ${\u\X}^{\u\A}$.
        \item A loose morphism from $F$ to $F'$ is given by a graph morphism as follows (see \cref{Grph(Cat)-is-CC}).
        \[\begin{tikzcd}
            {\A_0} & {\A_1} & {\A_0} \\
            {\X_0} & {\X_1} & {\X_0}
            \arrow["F"', from=1-1, to=2-1]
            \arrow["s"', from=1-2, to=1-1]
            \arrow["t", from=1-2, to=1-3]
            \arrow["\Xi"{description}, from=1-2, to=2-2]
            \arrow["{F'}", from=1-3, to=2-3]
            \arrow["s", from=2-2, to=2-1]
            \arrow["t"', from=2-2, to=2-3]
        \end{tikzcd}\]
        \item \label{cell-in-exponential} A cell with the following frame
        \[\begin{tikzcd}
            {F_0} & \cdots & {F_n} \\
            F && {F'}
            \arrow[""{name=0, anchor=center, inner sep=0}, "\Phi"', from=1-1, to=2-1]
            \arrow["{\Xi_1}"'{inner sep=.8ex}, "\shortmid"{marking}, from=1-2, to=1-1]
            \arrow["{\Xi_n}"'{inner sep=.8ex}, "\shortmid"{marking}, from=1-3, to=1-2]
            \arrow[""{name=1, anchor=center, inner sep=0}, "{\Phi'}", from=1-3, to=2-3]
            \arrow["\Xi"{inner sep=.8ex}, "\shortmid"{marking}, from=2-3, to=2-1]
            \arrow["\xi"{description}, draw=none, from=1, to=0]
        \end{tikzcd}\]
        comprises, for each chain of loose morphisms $p_1, \ldots, p_n$ in $\A$, a cell in $\X$,
        \[\begin{tikzcd}[column sep=large]
            {F_0A_0} & \cdots & {F_nA_n} \\
            {FA_0} && {F'A_n}
            \arrow[""{name=0, anchor=center, inner sep=0}, "{\Phi{A_0}}"', from=1-1, to=2-1]
            \arrow["{\Xi_1(p_1)}"'{inner sep=.8ex}, "\shortmid"{marking}, from=1-2, to=1-1]
            \arrow["{\Xi_n(p_n)}"'{inner sep=.8ex}, "\shortmid"{marking}, from=1-3, to=1-2]
            \arrow[""{name=1, anchor=center, inner sep=0}, "{\Phi'_{A_n}}", from=1-3, to=2-3]
            \arrow["{\Xi(p_1 \odot \cdots \odot p_n)}"{inner sep=.8ex}, "\shortmid"{marking}, from=2-3, to=2-1]
            \arrow["{\xi(p_1, \ldots, p_n)}"{description}, draw=none, from=1, to=0]
        \end{tikzcd}\]
        which is natural in the sense that, for each family of cells $\varpi_i$ in $\A$ ($0 \leq i < n$),
        \[\begin{tikzcd}
            {A_i'} & {A_{i + 1}'} \\
            {A_i} & {A_{i + 1}}
            \arrow[""{name=0, anchor=center, inner sep=0}, "{a_i}"', from=1-1, to=2-1]
            \arrow["{p_i'}"'{inner sep=.8ex}, "\shortmid"{marking}, from=1-2, to=1-1]
            \arrow[""{name=1, anchor=center, inner sep=0}, "{a_{i + 1}}", from=1-2, to=2-2]
            \arrow["{p_i}", from=2-2, to=2-1]
            \arrow["{\varpi_{i + 1}}"{description}, draw=none, from=0, to=1]
        \end{tikzcd}\]
        the following two cells in $\X$ are equal.
        \[
        \begin{tikzcd}
            {F_0A_0'} && {F_1A_1'} & \cdots & {F_{n - 1}A_{n - 1}'} && {F_nA_n'} \\
            {F_0A_0} && {F_1A_1} & \cdots & {F_{n - 1}A_{n - 1}} && {F_nA_n} \\
            {FA_0} &&&&&& {F'A_n}
            \arrow[""{name=0, anchor=center, inner sep=0}, "{F_0a_0}"', from=1-1, to=2-1]
            \arrow["{\Xi_1(p_1')}"'{inner sep=.8ex}, "\shortmid"{marking}, from=1-3, to=1-1]
            \arrow[""{name=1, anchor=center, inner sep=0}, "{F_1a_1}"{description}, from=1-3, to=2-3]
            \arrow["\shortmid"{marking}, from=1-4, to=1-3]
            \arrow["\cdots"{description}, draw=none, from=1-4, to=2-4]
            \arrow["\shortmid"{marking}, from=1-5, to=1-4]
            \arrow[""{name=2, anchor=center, inner sep=0}, "{F_{n - 1}a_{n - 1}}"{description}, from=1-5, to=2-5]
            \arrow["{\Xi_n(p_n')}"'{inner sep=.8ex}, "\shortmid"{marking}, from=1-7, to=1-5]
            \arrow[""{name=3, anchor=center, inner sep=0}, "{F_na_n}", from=1-7, to=2-7]
            \arrow[""{name=4, anchor=center, inner sep=0}, "{\Phi_{A_0}}"', from=2-1, to=3-1]
            \arrow["{\Xi_1(p_1)}"{description}, from=2-3, to=2-1]
            \arrow["\shortmid"{marking}, from=2-4, to=2-3]
            \arrow["\shortmid"{marking}, from=2-5, to=2-4]
            \arrow["{\Xi_n(p_n)}"{description}, from=2-7, to=2-5]
            \arrow[""{name=5, anchor=center, inner sep=0}, "{\Phi'_{A_n}}", from=2-7, to=3-7]
            \arrow["{\Xi(p_1 \odot \cdots \odot p_n)}"{inner sep=.8ex}, "\shortmid"{marking}, from=3-7, to=3-1]
            \arrow["{\Xi_1(\varpi_1)}"{description}, draw=none, from=0, to=1]
            \arrow["{\Xi_n(\varpi_n)}"{description}, draw=none, from=2, to=3]
            \arrow["{\xi(p_1, \ldots, p_n)}"{description}, draw=none, from=5, to=4]
        \end{tikzcd}
        \]
        \[
        \begin{tikzcd}
            {F_0A_0'} && {F_1A_1'} & \cdots & {F_{n - 1}A_{n - 1}'} && {F_nA_n'} \\
            {FA_0'} &&&&&& {F'A_n'} \\
            {FA_0} &&&&&& {F'A_n}
            \arrow[""{name=0, anchor=center, inner sep=0}, "{\Phi_{A_0'}}"', from=1-1, to=2-1]
            \arrow["{\Xi_1(p_1')}"'{inner sep=.8ex}, "\shortmid"{marking}, from=1-3, to=1-1]
            \arrow["\shortmid"{marking}, from=1-4, to=1-3]
            \arrow["\shortmid"{marking}, from=1-5, to=1-4]
            \arrow["{\Xi_n(p_n')}"'{inner sep=.8ex}, "\shortmid"{marking}, from=1-7, to=1-5]
            \arrow[""{name=1, anchor=center, inner sep=0}, "{\Phi'_{A_n'}}", from=1-7, to=2-7]
            \arrow[""{name=2, anchor=center, inner sep=0}, "{Fa_0}"', from=2-1, to=3-1]
            \arrow["{\Xi(p_1' \odot \cdots \odot p_n')}"{description}, from=2-7, to=2-1]
            \arrow[""{name=3, anchor=center, inner sep=0}, "{Fa_n}", from=2-7, to=3-7]
            \arrow["{\Xi(p_1 \odot \cdots \odot p_n)}"{inner sep=.8ex}, "\shortmid"{marking}, from=3-7, to=3-1]
            \arrow["{\xi(p_1', \ldots, p_n')}"{description}, draw=none, from=1, to=0]
            \arrow["{\Xi(\varpi_1 \odot \cdots \odot \varpi_n)}"{description}, draw=none, from=2, to=3]
        \end{tikzcd}
        \]
        \item The identity cell on $\Xi \colon F' \lto F$ is given by sending each loose morphism $p$ in $\A$ to the identity cell $1_{\Xi(p)}$ in $\X$.
        \item The composite of cells
        \[\begin{tikzcd}
            {F_0^1} & \cdots & {F_{n_1}^1} & \cdots & {F_0^m} & \cdots & {F_{n_m}^m} \\
            {F_0} && {F_1} & \cdots & {F_{m - 1}} && {F_m} \\
            F &&&&&& {F'}
            \arrow[""{name=0, anchor=center, inner sep=0}, "{\Phi_0}"', from=1-1, to=2-1]
            \arrow["{\Xi_1^1}"'{inner sep=.8ex}, "\shortmid"{marking}, from=1-2, to=1-1]
            \arrow["{\Xi_{n_1}^1}"'{inner sep=.8ex}, "\shortmid"{marking}, from=1-3, to=1-2]
            \arrow[""{name=1, anchor=center, inner sep=0}, "{\Phi_1}"{description}, from=1-3, to=2-3]
            \arrow["\shortmid"{marking}, from=1-4, to=1-3]
            \arrow["\shortmid"{marking}, from=1-5, to=1-4]
            \arrow[""{name=2, anchor=center, inner sep=0}, "{\Phi_{m - 1}}"{description}, from=1-5, to=2-5]
            \arrow["{\Xi_1^m}"'{inner sep=.8ex}, "\shortmid"{marking}, from=1-6, to=1-5]
            \arrow["{\Xi_{n_m}^m}"'{inner sep=.8ex}, "\shortmid"{marking}, from=1-7, to=1-6]
            \arrow[""{name=3, anchor=center, inner sep=0}, "{\Phi_m}", from=1-7, to=2-7]
            \arrow[""{name=4, anchor=center, inner sep=0}, "\Phi"', from=2-1, to=3-1]
            \arrow["{\Xi_1}"{description}, from=2-3, to=2-1]
            \arrow["\shortmid"{marking}, from=2-4, to=2-3]
            \arrow["\shortmid"{marking}, from=2-5, to=2-4]
            \arrow["{\Xi_m}"{description}, from=2-7, to=2-5]
            \arrow[""{name=5, anchor=center, inner sep=0}, "{\Phi'}", from=2-7, to=3-7]
            \arrow["\Xi"{inner sep=.8ex}, "\shortmid"{marking}, from=3-7, to=3-1]
            \arrow["{\xi^1}"{description}, draw=none, from=1, to=0]
            \arrow["\cdots"{description}, draw=none, from=1, to=2]
            \arrow["{\xi^m}"{description}, draw=none, from=2, to=3]
            \arrow["\xi"{description}, draw=none, from=4, to=5]
        \end{tikzcd}\]
        is given by the cell sending each chain $p_1^1, \ldots, p_{m_n}^m$ of loose morphisms in $\A$ to the following cell in $\X$.
        \[\begin{tikzcdscale}[column sep=large]
            {F_0^1A_0^1} & \cdots & {F_{n_1}^1 A_{n_1}^1} & \cdots & {F_0^m A_0^m} & \cdots & {F_{n_m}^m A_{n_m}^m} \\
            {F_0 A_0^1} && {F_1 A_{n_1}^1} & \cdots & {F_{m - 1} A_0^m} && {F_m A_{n_m}^m} \\
            {F A_0^1} &&&&&& {F' A_{n_m}^m} \\
            {F A_0^1} &&&&&& {F' A_{n_m}^m}
            \arrow[""{name=0, anchor=center, inner sep=0}, "{(\Phi_0)_{A_0^1}}"', from=1-1, to=2-1]
            \arrow["{\Xi_1^1(p_1^1)}"'{inner sep=.8ex}, "\shortmid"{marking}, from=1-2, to=1-1]
            \arrow["{\Xi_{n_1}^1(p_{n_1}^1)}"'{inner sep=.8ex}, "\shortmid"{marking}, from=1-3, to=1-2]
            \arrow[""{name=1, anchor=center, inner sep=0}, "{(\Phi_1)_{A_{n_1}^1}}"{description}, from=1-3, to=2-3]
            \arrow["\shortmid"{marking}, from=1-4, to=1-3]
            \arrow["\shortmid"{marking}, from=1-5, to=1-4]
            \arrow[""{name=2, anchor=center, inner sep=0}, "{(\Phi_{m - 1})_{A_0^m}}"{description}, from=1-5, to=2-5]
            \arrow["{\Xi_1^m (p_1^m)}"'{inner sep=.8ex}, "\shortmid"{marking}, from=1-6, to=1-5]
            \arrow["{\Xi_{n_m}^m (p_{n_m}^m)}"'{inner sep=.8ex}, "\shortmid"{marking}, from=1-7, to=1-6]
            \arrow[""{name=3, anchor=center, inner sep=0}, "{(\Phi_m)_{A_{n_m}^m}}", from=1-7, to=2-7]
            \arrow[""{name=4, anchor=center, inner sep=0}, "{\Phi_{A_0^1}}"', from=2-1, to=3-1]
            \arrow["{\Xi_1(p_1^1 \odot \cdots \odot p_{n_1}^1)}"{description}, from=2-3, to=2-1]
            \arrow["\shortmid"{marking}, from=2-4, to=2-3]
            \arrow["\shortmid"{marking}, from=2-5, to=2-4]
            \arrow["{\Xi_m(p_1^m \odot \cdots \odot p_{n_m}^m)}"{description}, from=2-7, to=2-5]
            \arrow[""{name=5, anchor=center, inner sep=0}, "{\Phi'_{A_{n_m}^m}}", from=2-7, to=3-7]
            \arrow[""{name=6, anchor=center, inner sep=0}, equals, from=3-1, to=4-1]
            \arrow["{\Xi((p_1^1 \odot \cdots \odot p_{n_1}^1) \odot \cdots \odot (p_1^m \odot \cdots \odot p_{n_m}^m))}"{description}, from=3-7, to=3-1]
            \arrow[""{name=7, anchor=center, inner sep=0}, equals, from=3-7, to=4-7]
            \arrow["{\Xi(p_1^1 \odot \cdots \odot p_{n_1}^1 \odot \cdots \odot p_1^m \odot \cdots \odot p_{n_m}^m)}", from=4-7, to=4-1]
            \arrow["{\xi^1(p_1^1, \ldots, p_{n_1}^1)}"{description}, draw=none, from=1, to=0]
            \arrow["\cdots"{description}, draw=none, from=1, to=2]
            \arrow["{\xi^m(p_1^m, \ldots, p_{n_m}^m)}"{description}, draw=none, from=2, to=3]
            \arrow["{\xi(p_1^1 \odot \cdots \odot p_{n_1}^1, \ldots, p_1^m \odot \cdots \odot p_{n_m}^m)}"{description}, draw=none, from=4, to=5]
            \arrow["{\Xi(\iso)}"{description}, draw=none, from=6, to=7]
        \end{tikzcdscale}\]
    \end{enumerate}
    That composition is unital and associative follows from unitality and associativity of composition in $\A$ and $\X$, together with naturality of the cells.

    We must show that, given a \vdc{} $\W$, there is an isomorphism of categories $\VDbl(\W \times \A, \X) \iso \VDbl(\W, \X^\A)$, 2-natural in $\W$ and $\X$. First, note that this isomorphism certainly holds on the underlying categories of objects and tight morphisms, \ie $\Cat(\u\W \times \u\A, \u\X) \iso \Cat(\u\W, \u\X^{\u\A})$, so it remains to check the assignments on the loose morphisms and cells.

    \begin{enumerate}
        \item \label{left-functor} A functor $F \colon \W \times \A \to \X$ comprises an assignment sending pairs of loose morphisms, $p \colon W' \lto W$ in $\W$ and $q \colon A' \lto A$ in $\A$, to a loose morphism $F(p, q) \colon F(W', A') \lto F(W, A)$ in $\X$; and an assignment sending pairs of a cell in $\W$ and a cell in $\A$,
        \[
        \begin{tikzcd}
            {W_1} & \cdots & {W_n} \\
            W && {W'}
            \arrow[""{name=0, anchor=center, inner sep=0}, "w"', from=1-1, to=2-1]
            \arrow["{p_1}"', "\shortmid"{marking}, from=1-2, to=1-1]
            \arrow["{p_n}"', "\shortmid"{marking}, from=1-3, to=1-2]
            \arrow[""{name=1, anchor=center, inner sep=0}, "{w'}", from=1-3, to=2-3]
            \arrow["p", "\shortmid"{marking}, from=2-3, to=2-1]
            \arrow["\omega"{description}, draw=none, from=0, to=1]
        \end{tikzcd}
        \hspace{4em}
        \begin{tikzcd}
            {A_1} & \cdots & {A_n} \\
            A && {A'}
            \arrow[""{name=0, anchor=center, inner sep=0}, "a"', from=1-1, to=2-1]
            \arrow["{q_1}"', "\shortmid"{marking}, from=1-2, to=1-1]
            \arrow["{q_n}"', "\shortmid"{marking}, from=1-3, to=1-2]
            \arrow[""{name=1, anchor=center, inner sep=0}, "{a'}", from=1-3, to=2-3]
            \arrow["q", "\shortmid"{marking}, from=2-3, to=2-1]
            \arrow["\alpha"{description}, draw=none, from=0, to=1]
        \end{tikzcd}
        \]
        to a cell in $\X$ as follows.
        \[\begin{tikzcd}[column sep=large]
            {F(W_1, A_1)} & \cdots & {F(W_n, A_n)} \\
            {F(W, A)} && {F(W', A')}
            \arrow[""{name=0, anchor=center, inner sep=0}, "{F(w, a)}"', from=1-1, to=2-1]
            \arrow["{F(p_1, q_1)}"', "\shortmid"{marking}, from=1-2, to=1-1]
            \arrow["{F(p_n, q_n)}"', "\shortmid"{marking}, from=1-3, to=1-2]
            \arrow[""{name=1, anchor=center, inner sep=0}, "{F(w', a')}", from=1-3, to=2-3]
            \arrow["{F(p, q)}", "\shortmid"{marking}, from=2-3, to=2-1]
            \arrow["{F(\omega, \alpha)}"{description}, draw=none, from=0, to=1]
        \end{tikzcd}\]
        \item \label{right-functor} A functor $G \colon \W \to \X^\A$ comprises an assignment on loose morphisms sending a loose morphism $p \colon W' \lto W$ in $\W$ to a morphism of spans,
        \[\begin{tikzcd}
            {\A_0} & {\A_1} & {\A_0} \\
            {\X_0} & {\X_1} & {\X_0}
            \arrow["{G(W')}"', from=1-1, to=2-1]
            \arrow["s"', from=1-2, to=1-1]
            \arrow["t", from=1-2, to=1-3]
            \arrow["{G(p)}"{description}, from=1-2, to=2-2]
            \arrow["{G(W)}", from=1-3, to=2-3]
            \arrow["s", from=2-2, to=2-1]
            \arrow["t"', from=2-2, to=2-3]
        \end{tikzcd}\]
        which sends each loose morphism $q \colon A' \lto A$ in $\A$ to a loose morphism $G(p)(q) \colon G(W')(A') \lto G(W)(A)$ in $\X$, and sends each unary cell in $\A$
        \[\begin{tikzcd}
            {A_0} & {A_1} \\
            A & {A'}
            \arrow[""{name=0, anchor=center, inner sep=0}, "a"', from=1-1, to=2-1]
            \arrow["q"'{inner sep=.8ex}, "\shortmid"{marking}, from=1-2, to=1-1]
            \arrow[""{name=1, anchor=center, inner sep=0}, "{a'}", from=1-2, to=2-2]
            \arrow["{q'}"{inner sep=.8ex}, "\shortmid"{marking}, from=2-2, to=2-1]
            \arrow["\alpha"{description}, draw=none, from=0, to=1]
        \end{tikzcd}\]
        to a unary cell in $\X$ as follows.
        \[\begin{tikzcd}[column sep=large]
            {G(W)(A_0)} & {G(W')(A_1)} \\
            {G(W)(A)} & {G(W')(A')}
            \arrow[""{name=0, anchor=center, inner sep=0}, "{G(W)(a)}"', from=1-1, to=2-1]
            \arrow["{G(p)(q)}"'{inner sep=.8ex}, "\shortmid"{marking}, from=1-2, to=1-1]
            \arrow[""{name=1, anchor=center, inner sep=0}, "{G(W')(a')}", from=1-2, to=2-2]
            \arrow["{G(p)(q')}"{inner sep=.8ex}, "\shortmid"{marking}, from=2-2, to=2-1]
            \arrow["{G(p)(\alpha)}"{description}, draw=none, from=0, to=1]
        \end{tikzcd}\]
        It also comprises an assignment of cells $\omega$ in $\W$ as above to a cell in $\X^\A$ with the following frame,
        \[\begin{tikzcd}[column sep=large]
            {G(W_1)({-})} & \cdots & {G(W_n)({-})} \\
            {G(W)({-})} && {G(W')({-})}
            \arrow[""{name=0, anchor=center, inner sep=0}, "{G(w)({-})}"', from=1-1, to=2-1]
            \arrow["{G(p_1)({-})}"'{inner sep=.8ex}, "\shortmid"{marking}, from=1-2, to=1-1]
            \arrow["{G(p_n)({-})}"'{inner sep=.8ex}, "\shortmid"{marking}, from=1-3, to=1-2]
            \arrow[""{name=1, anchor=center, inner sep=0}, "{G(w')({-})}", from=1-3, to=2-3]
            \arrow["{G(p)({-})}"{inner sep=.8ex}, "\shortmid"{marking}, from=2-3, to=2-1]
            \arrow["{G(\omega)({-})}"{description}, draw=none, from=0, to=1]
        \end{tikzcd}\]
        hence, for each chain of loose morphisms $q_1, \ldots, q_n$ in $\A$, a cell in $\X$ as follows.
        \[\begin{tikzcd}[column sep=large]
            {G(W_1)(A_1)} & \cdots & {G(W_n)(A_n)} \\
            {G(W)(A_1)} && {G(W')(A_n)}
            \arrow[""{name=0, anchor=center, inner sep=0}, "{G(w)(A_1)}"', from=1-1, to=2-1]
            \arrow["{G(p_1)(q_1)}"'{inner sep=.8ex}, "\shortmid"{marking}, from=1-2, to=1-1]
            \arrow["{G(p_n)(q_n)}"'{inner sep=.8ex}, "\shortmid"{marking}, from=1-3, to=1-2]
            \arrow[""{name=1, anchor=center, inner sep=0}, "{G(w')(A_n)}", from=1-3, to=2-3]
            \arrow["{G(p)(q_1 \odot \cdots \odot q_n)}"{inner sep=.8ex}, "\shortmid"{marking}, from=2-3, to=2-1]
            \arrow["{G(\omega)(q_1, \ldots, q_n)}"{description}, draw=none, from=0, to=1]
        \end{tikzcd}\]
    \end{enumerate}
    It is clear that the action of \eqref{left-functor} and \eqref{right-functor} above on loose morphisms and unary cells is the same. To establish the one-dimensional universal property of the exponential, it remains only to verify that their action on arbitrary cells is the same.

    Observe that, since $\A$ is representable, every cell $\alpha$ in $\A$ factors uniquely through an opcartesian cell as follows.
    \[\begin{tikzcd}
        {A_1} & \cdots & {A_n} \\
        {A_1} && {A_n} \\
        A && {A'}
        \arrow[""{name=0, anchor=center, inner sep=0}, Rightarrow, no head, from=1-1, to=2-1]
        \arrow["{q_1}"', "\shortmid"{marking}, from=1-2, to=1-1]
        \arrow["{q_n}"', "\shortmid"{marking}, from=1-3, to=1-2]
        \arrow[""{name=1, anchor=center, inner sep=0}, Rightarrow, no head, from=1-3, to=2-3]
        \arrow[""{name=2, anchor=center, inner sep=0}, "a"', from=2-1, to=3-1]
        \arrow["{q_1 \odot \cdots \odot q_n}"{description}, from=2-3, to=2-1]
        \arrow[""{name=3, anchor=center, inner sep=0}, "{a'}", from=2-3, to=3-3]
        \arrow["q", "\shortmid"{marking}, from=3-3, to=3-1]
        \arrow["\opcart"{description}, draw=none, from=0, to=1]
        \arrow["{\check\alpha}"{description}, draw=none, from=2, to=3]
    \end{tikzcd}\]
    Consequently, given $\omega$ and $\alpha$ as above, we have by \eqref{right-functor} a cell in $\X$ as follows.
    \[\begin{tikzcd}[column sep=huge]
        {G(W_1)(A_1)} & \cdots & {G(W_n)(A_n)} \\
        {G(W)(A_1)} && {G(W')(A_n)} \\
        {G(W)(A)} && {G(W')(A)}
        \arrow[""{name=0, anchor=center, inner sep=0}, "{G(w)(A_1)}"', from=1-1, to=2-1]
        \arrow["{G(p_1)(q_1)}"'{inner sep=.8ex}, "\shortmid"{marking}, from=1-2, to=1-1]
        \arrow["{G(p_n)(q_n)}"'{inner sep=.8ex}, "\shortmid"{marking}, from=1-3, to=1-2]
        \arrow[""{name=1, anchor=center, inner sep=0}, "{G(w')(A_n)}", from=1-3, to=2-3]
        \arrow[""{name=2, anchor=center, inner sep=0}, "{G(W)(a)}"', from=2-1, to=3-1]
        \arrow["{G(p)(q_1 \odot \cdots \odot q_n)}"{description}, from=2-3, to=2-1]
        \arrow[""{name=3, anchor=center, inner sep=0}, "{G(W')(a')}", from=2-3, to=3-3]
        \arrow["{G(p)(q)}"{inner sep=.8ex}, "\shortmid"{marking}, from=3-3, to=3-1]
        \arrow["{G(\omega)(q_1, \ldots, q_n)}"{description}, draw=none, from=0, to=1]
        \arrow["{G(p)(\check\alpha)}"{description}, draw=none, from=2, to=3]
    \end{tikzcd}\]
    Conversely, for each chain of loose morphisms $q_1, \ldots, q_n$ in $\A$, we have by \eqref{left-functor} a cell in $\X$ as follows.
    \[\begin{tikzcd}[column sep=large]
        {F(W_1, A_1)} & \cdots & {F(W_n, A_n)} \\
        {F(W, A_1)} && {F(W')(A_n)}
        \arrow[""{name=0, anchor=center, inner sep=0}, "{F(w, A_1)}"', from=1-1, to=2-1]
        \arrow["{F(p_1, q_1)}"'{inner sep=.8ex}, "\shortmid"{marking}, from=1-2, to=1-1]
        \arrow["{F(p_n, q_n)}"'{inner sep=.8ex}, "\shortmid"{marking}, from=1-3, to=1-2]
        \arrow[""{name=1, anchor=center, inner sep=0}, "{F(w', A_n)}", from=1-3, to=2-3]
        \arrow["{F(p, q_1 \odot \cdots \odot q_n)}"{inner sep=.8ex}, "\shortmid"{marking}, from=2-3, to=2-1]
        \arrow["{F(\omega, \opcart)}"{description}, draw=none, from=0, to=1]
    \end{tikzcd}\]
    That this correspondence is bijective follows from opcartesianness. Consequently, there is a bijection $\VDbl(\W \times \A, \X)_0 \iso \VDbl(\W, \X^\A)_0$ of sets, which is easily seen to be natural in $\W$ and $\X$, and hence $\X^\A$ satisfies the one-dimensional aspect of the universal property of an exponential.

    Finally, since $\ph \times \A$ preserves copowers by $\b2$ using the explicit description in \cref{copower-by-2}, the one-dimensional aspect of the universal property of $\X^\A$ implies the two-dimensional aspect by \cite[Proposition~3.1]{blackwell1989two}.
\end{proof}

\begin{remark}
    \label{generalised-multicategories}
    The reader familiar with \citeauthor{pisani2014sequential}'s work on exponentiable multicategories~\cite{pisani2014sequential} will observe similarities in our analysis of exponentiability for \vdcs. Our notation for the 2-adjunction $\ph\up \adj \ph\down \colon \VDbl \to \Grph(\Cat)$ is chosen to align with \citeauthor{pisani2014sequential}'s notation for the analogous 2-adjunction between categories and multicategories~\cites[\S2.1]{pisani2014sequential}[\S6.7]{leinster2004higher}.
    \[\begin{tikzcd}[column sep=large]
        \Cat & \Multicat
        \arrow[""{name=0, anchor=center, inner sep=0}, "{\ph\up}", shift left=2, hook, from=1-1, to=1-2]
        \arrow[""{name=1, anchor=center, inner sep=0}, "{\ph\down}", shift left=2, from=1-2, to=1-1]
        \arrow["\dashv"{anchor=center, rotate=-90}, draw=none, from=0, to=1]
    \end{tikzcd}\]
    However, note that, to extend some concepts from multicategories to \vdcs{}, the appropriate analogue of the 2-adjunction above is $\o\ph \adj \u\ph \colon \VDbl \to \Cat$ rather than $\ph\up \adj \ph\down \colon \VDbl \to \Grph(\Cat)$. For instance, see \cref{sequential-adjoint} and the discussion preceding it.
\end{remark}

Representability is merely a sufficient condition for exponentiability, not a necessary condition. To demonstrate this, we exhibit another class of exponentiable \vdcs{}, not contained in the class of representable \vdcs. We first observe that categories, in addition to being coreflective in \vdcs{} by \cref{Cat-graphs-into-VDCs}, are also reflective therein (\cf~\cite[\S6]{fujii2025familial}).

\begin{proposition}
    \label{chaotic-vdc}
    The forgetful 2-functor $\u{\ph} \colon \VDbl \to \Cat$ admits a \ff{} right adjoint $\Ch \colon \Cat \to \VDbl$ sending each category to the representable \vdc{} having the same underlying category, a unique loose morphism between every pair of objects, and a unique cell filling every frame.
\end{proposition}

\begin{proof}
    Trivial.
\end{proof}

Since $\Ch(\b A)$ is representable for each category $\b A$, it is exponentiable by \cref{representable-vdcs-are-exponentiable}. We may see that, in addition to the representable \vdcs, every loosely discrete \vdc{} on a category is exponentiable.\footnotemark{}
\footnotetext{This has also been observed by \textcite[Remark~6.1]{fujii2025familial}.}%

\begin{proposition}
    \label{oA-is-exponentiable}
    Let $\b A$ be a category. $\o{\b A}$ is exponentiable, but is representable if and only if $\b A = \b0$. Concretely, $\Y^{\o{\b A}} \iso \Ch(\u\Y^{\b A})$ for any \vdc{} $\Y$.
\end{proposition}

\begin{proof}
    We have the following chain of isomorphims, natural in \vdcs{} $\X$ and $\Y$.
    \begin{align*}
        \VDbl(\X \times \o{\b A}, \Y) & \iso \VDbl(\o{\u\X \times \b A}, \Y) \tag{\cref{reciprocity}} \\
            & \iso \Cat(\u\X \times \b A, \u\Y) \tag{\cref{Cat-graphs-into-VDCs}} \\
            & \iso \Cat(\u\X, \u\Y^{\b A}) \tag{$\Cat$ is cartesian closed} \\
            & \iso \VDbl(\X, \Ch(\u\Y^{\b A})) \tag{\cref{chaotic-vdc}}
    \end{align*}
    Finally, note that $\o{\b A}$ has no loose morphisms, so cannot be representable unless it has no objects (in which case it is trivially representable).
\end{proof}

In particular, we have $\X^{\o{\b1}} \iso \Ch(\u\X)$ for every \vdc{} $\X$, and so the 2-adjunction $\ph \times \o{\b1} \adj \ph^{\o{\b1}}$ recovers the 2-adjunction $\o{{\u\ph}} \adj \Ch(\u{\ph})$ arising from composing the 2-adjunctions of \cref{Cat-graphs-into-VDCs,chaotic-vdc}, the left adjoint having the structure of a 2-comonad on $\VDbl$, and the right adjoint having the structure of a 2-monad.

We shall return to the question of which \vdcs{} are exponentiable in \cref{exponentiability-revisited}, where we shall give a full characterisation.

\subsection{Relation to multicategories}
\label{relation-to-multicategories}

$\Multicat$ is a full sub-2-category of $\VDbl$: its image comprises the \vdcs{} whose underlying category is terminal (in other words, the \vdcs{} having a single object and single tight morphism as in \cref{multicategory}). It follows from the work of \textcite{pisani2014sequential} that every representable multicategory (\ie{} monoidal category) is exponentiable~\cite[\S2.9]{pisani2014sequential}. In fact, the exponentiability of monoidal categories is a special case of a more general phenomenon: that multicategories may be powered by exponentiable \vdcs{}, in the sense of being equipped with a coherently associative and unital functor $\pow \colon \VDbl\exp\op \times \Multicat \to \Multicat$~\cite[Definition~6.1]{mcdermott2022what}.

\begin{proposition}
    \label{multicategories-are-powered-over-VDCs}
    $\Multicat$ is powered over $\VDbl\exp$ (the full sub-2-category of $\VDbl$ spanned by exponentiable \vdcs{}) and the 2-functor $\Sigma \colon \Multicat \ffto \VDbl$ strictly preserves these powers. Explicitly, for a multicategory $\M$ and \pdc{} $\A$, the power $\A \pow \M \defeq (\Sigma\M)^\A$ has objects the functors $\A_1 \to \b M_1$ and as multimorphisms $F_1, \ldots, F_n \to G$ families of multimorphisms
    \[\{ F_1(p_1), \ldots, F_n(p_n) \to G(p_1 \odot \cdots \odot p_n) \}_{p_1, \ldots, p_n \in \A_n}\]
    natural in $p_1, \ldots, p_n$. Consequently, $\Sigma$ preserves exponentials by monoidal categories.
\end{proposition}

\begin{proof}
    Let $\M$ be a multicategory and let $\A$ be an exponentiable \vdc{}. Since the underlying category of $\Sigma\M$ is terminal, so is that of $(\Sigma\M)^\A$ by \cref{reciprocity}: we denote the corresponding multicategory $\A \pow \M$. We therefore obtain a functor $\pow \colon \VDbl\exp\op \times \Multicat \to \Multicat$. That this exhibits $\Multicat$ as powered follows from the fact that exponentiation exhibits $\VDbl$ as powered over $\VDbl\exp$, and $\Multicat$ is a full subcategory of $\VDbl$. Unwrapping the description of the exponential in \cref{representable-vdcs-are-exponentiable} gives the stated definition of $\A \pow \M$ when $\A$ is representable.

    For $\b A$ a monoidal category, viewed as a representable multicategory, $\Sigma\b A$ is also representable, and $\Sigma\M^{\Sigma\b A}$ coincides with $\Sigma(\M^{\b A})$ by inspection of the definition of exponentials of multicategories~\cite[\S2.9]{pisani2014sequential}.
\end{proof}

\textcite{pisani2014sequential} gives a complete characterisation of the exponentiable multicategories: they are precisely the promonoidal categories~\cite[Proposition~2.8]{pisani2014sequential}. Given a small promonoidal category $\b A$ and a cocomplete monoidal category $\M$, the exponential multicategory $\M^{\b A}$ is representable, and the corresponding monoidal structure coincides with \citeauthor{day1970closed}'s~\cite{day1970closed} convolution monoidal structure~\cite[Proposition~2.12]{pisani2014sequential}. Consequently, the operation of powering a multicategory by a \vdc{} may be viewed as a substantial generalisation of convolution for monoidal categories.

We expect $\Sigma$ to preserve and reflect exponentiability; this would follow straightforwardly from a concrete description of exponentials (see \cref{exponentiable-vdcs-conjecture}). For now, we merely show that $\Sigma$ reflects exponentiability.

\begin{lemma}
    \label{Sigma-reflects-exponentiability}
    $\Sigma$ reflects exponentiability: if $\b A$ is a multicategory for which $\Sigma\b A$ is exponentiable as a \vdc{}, then $\b A$ is exponentiable as a multicategory.
\end{lemma}

\begin{proof}
    Suppose the \vdc{} $\Sigma\b A$ is exponentiable. We have:
    \begin{align*}
        \Multicat(\M \times \b A, \b N) & \iso \VDbl(\Sigma(\M \times \b A), \Sigma\b N) \tag{$\Sigma$ is \ff} \\
            & \iso \VDbl(\Sigma\M \times \Sigma\b A, \Sigma\b N) \tag{$\Sigma$ preserves products} \\
            & \iso \VDbl(\Sigma\M, \Sigma\b N^{\Sigma\b A}) \tag{$\Sigma \b A$ is exponentiable} \\
            & \iso \VDbl(\Sigma\M, \Sigma(\Sigma\b A \pow \b N)) \tag{$\Sigma$ preserves powers} \\
            & \iso \Multicat(\M, \Sigma\b A \pow \b N) \tag{$\Sigma$ is \ff}
    \end{align*}
    which exhibits $\Sigma\b A \pow \b N$ as the exponential $\b N^{\b A}$.
\end{proof}

As a consequence, the failure of cartesian closure of multicategories implies the failure of cartesian closure of \vdcs.

\begin{corollary}
    \label{VDbl-is-not-CC}
    $\VDbl$ is not cartesian closed.
\end{corollary}

\begin{proof}
    If every \vdc{} were exponentiable, then every multicategory would also be exponentiable by \cref{Sigma-reflects-exponentiability}. However, there are non-exponentiable multicategories: for instance, no nonempty category, viewed as a multicategory with only unary morphisms, is exponentiable~\cite[Proposition~2.8]{pisani2014sequential}.
\end{proof}

It is worth noting that (a special case of) the power structure of \cref{multicategories-are-powered-over-VDCs} has been previously studied in the literature. In \cite{behr2023convolution}, \citeauthor{behr2023convolution} show that, given a small \pdc{} $\A$, the category of set-valued functors from the category $\A_1$, of loose morphisms and cells in $\A$, admits a colax monoidal structure that recovers the usual convolution monoidal structure of \textcite{day1970closed} when $\A$ is the delooping of a monoidal category. For the following, note that the correspondence between weakly representable \vdcs{} and colax \dcs{} (\cref{colax-dcs}) preserves the underlying categories of objects and tight morphisms and hence restricts to a correspondence between weakly representable multicategories and colax monoidal categories.

\begin{proposition}
    \label{power-of-Set}
    For a small \pdc{} $\A$, viewing $\Set$ as a cartesian multicategory, the power $\A \pow \Set$ is a weakly representable multicategory, and the corresponding colax monoidal structure on the category $\Set^{\A_1}$ is the convolution structure of \textcite[Theorem~3.5]{behr2023convolution}.
\end{proposition}

\begin{proof}
    A multimorphism $F_1, \ldots, F_n \to G$ in the convolution structure of \citeauthor{behr2023convolution} is given by a natural transformation
    \[\{ \int^{(p_1, \ldots, p_n) \in \A_n} \A_1(p_1 \odot \cdots \odot p_n, r) \times F_1(p_1) \times \cdots \times F_n(p_n) \to G(r) \}_{r \in \A_1}\]
    which, by the universal property of the coend, is equivalent to a natural family
    \[\{ \A_1(p_1 \odot \cdots \odot p_n, r) \times F_1(p_1) \times \cdots \times F_n(p_n) \to G(r) \}_{(p_1, \ldots, p_n) \in \A_n, r \in \A_1}\]
    which, by naturality, is equivalent to
    \[\{ F_1(p_1) \times \cdots \times F_n(p_n) \to G(p_1 \odot \cdots \odot p_n) \}_{(p_1, \ldots, p_n) \in \A_n}\]
    which are precisely the multimorphisms of $\A \pow \Set$.
\end{proof}

\begin{remark}
    \label{representability-of-exponential}
    We shall not study the representability of exponentials of \vdcs{} here, as it is tangential to our interests. However, the observation that the specific exponentials appearing in \cref{power-of-Set} are weakly representable suggests it would be worthwhile to investigate the general phenomenon. We conjecture that, for every small \pdc{} $\A$ and locally cocomplete\footnotemark{} and weakly representable \vdc{} $\X$, the exponential $\X^\A$ is again weakly representable. If $\A$ furthermore satisfies a suitable factorisation property, such as the \emph{AFP condition} of \textcite[Definition~3.1.1]{pare2013composition} or the \emph{$n$-cylindrical decomposition property} of \textcite[Definition~5.1]{behr2023convolution} (both of which are satisfied, for example, by \pdcs{} with companions and conjoints), and $\X$ admits non-nullary composites, we expect that $\X^\A$ also admits non-nullary composites. As observed in \cite[Example~4.5]{behr2023convolution}, loose identities rarely exist in exponentials, even assuming $\X$ admits them. While the absence of loose identities may seem undesirable, we will observe in \cref{representability-of-Lax} that this is less restrictive than might at first be imagined.
    \footnotetext{Local cocompleteness for \vdcs{} is defined in \cite[\S A]{arkor2026presheaves}.}%
\end{remark}

\section{The \vdc{} of loose monads}
\label{the-vdc-of-loose monads}

As outlined in the introduction, we will show in \cref{lax-functors-as-monads} that exponentiation of \vdcs{} is closely connected to \citeauthor{pare2011yoneda}'s \vdc{} of lax functors. To explain the connection, we must recall a second construction on \vdcs.

\begin{definition}[{\cite[Definition~5.3.1]{leinster2004higher}}]
    \label{LMnd}
    Let $\X$ be a \vdc{}. The \vdc{} $\Mod(\X)$ is defined as follows.
    \begin{enumerate}
        \item An object is a \emph{loose monad} in $\X$, comprising an object $\crr T \in \X$, a loose morphism $T \colon \crr T \lto \crr T$, and cells
        \[
        \begin{tikzcd}
            {\crr T} & {\crr T} & {\crr T} \\
            {\crr T} && {\crr T}
            \arrow[""{name=0, anchor=center, inner sep=0}, Rightarrow, no head, from=1-1, to=2-1]
            \arrow["T"', "\shortmid"{marking}, from=1-2, to=1-1]
            \arrow["T"', "\shortmid"{marking}, from=1-3, to=1-2]
            \arrow[""{name=1, anchor=center, inner sep=0}, Rightarrow, no head, from=1-3, to=2-3]
            \arrow["T", "\shortmid"{marking}, from=2-3, to=2-1]
            \arrow["{{\circ_T}}"{description}, draw=none, from=1, to=0]
        \end{tikzcd}
        \hspace{8em}
        \begin{tikzcd}
            {\crr T} & {\crr T} \\
            {\crr T} & {\crr T}
            \arrow[""{name=0, anchor=center, inner sep=0}, Rightarrow, no head, from=1-1, to=2-1]
            \arrow[Rightarrow, no head, from=1-2, to=1-1]
            \arrow[""{name=1, anchor=center, inner sep=0}, Rightarrow, no head, from=1-2, to=2-2]
            \arrow["T", "\shortmid"{marking}, from=2-2, to=2-1]
            \arrow["{{\I_T}}"{description}, draw=none, from=1, to=0]
        \end{tikzcd}
        \]
        subject to the following equations.
        \begin{align*}
            (\circ_T, 1_T) \d \circ_T & = (1_T, \circ_T) \d \circ_T &
            (\I_T, 1_T) \d \circ_T & = 1_T &
            (1_T, \I_T) \d \circ_T & = 1_T
        \end{align*}
        \item A tight morphism from $(\crr S, S, \circ_S, \I_S)$ to $(\crr T, T, \circ_T, \I_T)$ is a \emph{loose monad morphism}, comprising a tight morphism $\crr f \colon \crr S \to \crr T$ and a cell
        \[\begin{tikzcd}
            {\crr S} & {\crr S} \\
            {\crr T} & {\crr T}
            \arrow[""{name=0, anchor=center, inner sep=0}, "{{\crr f}}"', from=1-1, to=2-1]
            \arrow["S"', "\shortmid"{marking}, from=1-2, to=1-1]
            \arrow[""{name=1, anchor=center, inner sep=0}, "{{\crr f}}", from=1-2, to=2-2]
            \arrow["T", "\shortmid"{marking}, from=2-2, to=2-1]
            \arrow["f"{description}, draw=none, from=1, to=0]
        \end{tikzcd}\]
        subject to the following equations.
        \begin{align*}
            \circ_S \d f & = (f, f) \d \circ_T &
            \I_S \d f & = 1_{\crr f} \d \I_T
        \end{align*}
        \item A loose morphism is a \emph{loose monad module}, comprising a loose morphism $m \colon \crr{T'} \lto \crr T$ and cells
        \[
        \begin{tikzcd}
            {\crr T} & {\crr T} & {\crr{T'}} \\
            {\crr T} && {\crr{T'}}
            \arrow[""{name=0, anchor=center, inner sep=0}, Rightarrow, no head, from=1-1, to=2-1]
            \arrow["T"', "\shortmid"{marking}, from=1-2, to=1-1]
            \arrow["m"', "\shortmid"{marking}, from=1-3, to=1-2]
            \arrow[""{name=1, anchor=center, inner sep=0}, Rightarrow, no head, from=1-3, to=2-3]
            \arrow["m", "\shortmid"{marking}, from=2-3, to=2-1]
            \arrow["{{\lambda_m}}"{description}, draw=none, from=1, to=0]
        \end{tikzcd}
        \qquad
        \begin{tikzcd}
            {\crr T} & {\crr{T'}} & {\crr{T'}} \\
            {\crr T} && {\crr{T'}}
            \arrow[""{name=0, anchor=center, inner sep=0}, Rightarrow, no head, from=1-1, to=2-1]
            \arrow["m"', "\shortmid"{marking}, from=1-2, to=1-1]
            \arrow["{T'}"', "\shortmid"{marking}, from=1-3, to=1-2]
            \arrow[""{name=1, anchor=center, inner sep=0}, Rightarrow, no head, from=1-3, to=2-3]
            \arrow["m", "\shortmid"{marking}, from=2-3, to=2-1]
            \arrow["{{\rho_m}}"{description}, draw=none, from=1, to=0]
        \end{tikzcd}
        \]
        subject to the following equations.
        \begin{align*}
            (\circ_T, 1_m) \d \lambda_m & = (1_T, \lambda_m) \d \lambda_m &
            (\I_T, 1_m) \d \lambda_m & = 1_m \\
            (1_m, \circ_{T'}) \d \rho_m & = (\rho_m, 1_{T'}) \d \rho_m &
            (1_m, \I_{T'}) \d \rho_m & = 1_m \\
            (\lambda_m, 1_{T'}) \d \rho_m & = (1_T, \rho_m) \d \lambda_m
        \end{align*}
        \item A cell is \emph{loose monad transformation}, comprising a cell in $\X$
        \[\begin{tikzcd}
            {\crr{S_0}} & \cdots & {\crr{S_n}} \\
            {\crr T} && {\crr{T'}}
            \arrow[""{name=0, anchor=center, inner sep=0}, "{\crr{f}}"', from=1-1, to=2-1]
            \arrow["{m_1}"'{inner sep=.8ex}, "\shortmid"{marking}, from=1-2, to=1-1]
            \arrow["{m_n}"'{inner sep=.8ex}, "\shortmid"{marking}, from=1-3, to=1-2]
            \arrow[""{name=1, anchor=center, inner sep=0}, "{\crr{f'}}", from=1-3, to=2-3]
            \arrow["m"{inner sep=.8ex}, "\shortmid"{marking}, from=2-3, to=2-1]
            \arrow["\phi"{description}, draw=none, from=1, to=0]
        \end{tikzcd}\]
        subject to the following equations\footnotemark{}.
        \begin{align*}
            (f, \phi) \d \lambda_m & = (\phi, f') \d \rho_m \tag{$n = 0$} \\
            (\lambda_{m_1}, 1_{m_2}, \ldots, 1_{m_n}) \d \phi & = (f, \phi) \d \lambda_m \tag{$n \ge 1$} \\
            (1_{m_1}, \ldots, 1_{m_{n - 1}}, \rho_{m_n}) \d \phi & = (\phi, f') \d \rho_m \tag{$n \ge 1$} \\
            (1_{m_1}, \ldots, \lambda_{m_{i + 1}}, \ldots, 1_{m_n}) \d \phi & = (1_{m_1}, \ldots, \rho_{m_i}, \ldots, 1_{m_n}) \d \phi \tag{$1 \leq i < n$}
        \end{align*}
    \end{enumerate}
    \footnotetext{Note that \cite[Definition~5.3.1]{leinster2004higher} is incomplete, as it omits the coherence condition for nullary cells (as observed in \cite[Footnote~3]{arkor2025nerve}).}%
    Identities and composition is given componentwise.
\end{definition}

Many category-like structures arise by applying the $\Mod$ construction to a simpler \vdc. The prototypical example, stemming from observations of \textcite{benabou1967introduction} and \textcite{burroni1971tcategories}, is the \vdc{} of categories and distributors.

\begin{example}
    \label{Mod-Span}
    $\Mod(\Span) \iso \Dist$. More generally, for any category $\E$ with pullbacks, $\Mod(\Span(\E)) \iso \Dist(\E)$ (see \cref{Span(E),Dist(E)}).
\end{example}

Many more instances of \vdcs{} arising from the $\Mod$ construction may be found in the work of \textcite{cruttwell2010unified}; we shall see examples of a different flavour shortly.

\subsection{Normality}
\label{normality}

The $\Mod$ construction admits a convenient universal property: it is the cofree \emph{normal} \vdc{} construction. This property will be central to our study of lax functors in \cref{lax-functors-as-monads}.

\begin{definition}
    \label{normal-vdc}
    A \emph{normal} \vdc{} is a \vdc{} equipped with a choice of loose identities. A \emph{normal functor} strictly preserves the chosen loose identities. Denote by $\VDbln$ the 2-category of normal \vdcs{}, normal functors, and transformations. There is a forgetful 2-functor $\ph' \colon \VDbln \to \VDbl$.
\end{definition}

Our definition of normal functor may seem overly strict. However, any functor that preserves loose identities up to isomorphism is isomorphic to a functor that preserves loose identities strictly (note that an analogous property does \emph{not} hold for non-nullary composites). Thus there is no loss in generality in considering the stricter notion, which allows us to work with stricter universal properties.

Just as every monoid acts on itself, every loose monad is a module for itself: this exhibits $\Mod(\X)$ as a normal \vdc.

\begin{lemma}[{\cite[Proposition~5.5]{cruttwell2010unified}}]
    \label{Mod-is-normal}
    Let $\X$ be a \vdc{}. The \vdc{} $\Mod(\X)$ is normal: the loose identity on a loose monad $T$ in $\X$ is $T$ itself.
\end{lemma}

In fact, $\Mod$ is universal in this respect.

\begin{proposition}[{\cite[Proposition~5.14]{cruttwell2010unified}}]
    \label{Mod-is-right-adjoint}
    There is a 2-adjunction as follows.
    \[\begin{tikzcd}[column sep=large]
        \VDbln & \VDbl
        \arrow[""{name=0, anchor=center, inner sep=0}, "{\ph'}", shift left=2, from=1-1, to=1-2]
        \arrow[""{name=1, anchor=center, inner sep=0}, "\Mod", shift left=2, from=1-2, to=1-1]
        \arrow["\dashv"{anchor=center, rotate=-90}, draw=none, from=0, to=1]
    \end{tikzcd}\]
\end{proposition}

Note that \textcite{cruttwell2010unified} only establish the existence of a \emph{pseudo}adjunction, since their normal functors preserve loose identities only up to isomorphism. Since we will have need to analyse the unit and counit of this 2-adjunction further in what follows, we find it helpful to provide a self-contained proof of \cref{Mod-is-right-adjoint}.\footnote{\textcite[\S7.1]{fujii2025familial} share our preference for the 2-adjunction rather than the pseudoadjunction.}

\begin{proof}
    First, it is clear that $\Mod$ is functorial, since its definition is diagrammatic in nature. We define the unit
    \begin{equation}
        \label{Mnd-unit}
        \eta_\X \colon \X \to \Mod(\X')
    \end{equation}
    at a normal \vdc{} $\X$ to be the functor sending an object $A$ to the canonical loose monad structure on the loose identity $A(1, 1)$; a tight morphism $f \colon A \to B$ to the canonical monad morphism on $f$; a loose morphism $p$ to the canonical monad module on $p$; and a cell $\phi$ to the monad transformation $\phi$.

    Explicitly, the unit of the loose monad structure on $A(1, 1)$ is given by the nullary opcartesian cell defining $A(1, 1)$ and the multiplication is given by the mediating cell induced via the universal property of the loose identity on $A$ by the identity cell on $A(1, 1)$. Similarly, the cell component of the monad morphism structure on $f$ is given by the mediating cell induced via the universal property of the loose identity on $A$ by the identity cell on $f$ (itself obtained by precomposing the opcartesian cell defining $B$ by $f$). The cells of the monad module structure on $p$ are induced in the same way as the multiplication cell for $A(1, 1)$ by the identity cell on $p$. In each case, the equations follow from the universal properties of the opcartesian cells.

    The functor $\eta_\X$ is normal, since it sends a loose identity $A(1, 1) \colon A \lto A$ in $\X$ to the monad module $A(1, 1) \colon A(1, 1) \lto A(1, 1)$, which is the identity on the loose monad $A(1, 1)$. It is clear that $\eta$ is 2-natural with respect to strictly normal functors and transformations.

    We define the counit
    \begin{equation}
        \varepsilon_\Y \colon \Mod(\Y)' \to \Y
    \end{equation}
    at a \vdc{} $\Y$ to be the forgetful functor sending a loose monad $T$ to its carrier $\crr T$; a monad morphism $(\crr f, f)$ to the tight morphism $\crr f$; a monad module $m$ to the loose morphism $m$; and a monad transformation $\phi$ to the cell $\phi$. The family $\varepsilon$ is trivially 2-natural.

    The left triangle identity says that the carrier of $A(1, 1)$ is $A$. The right triangle identity says that if we view a loose monad $T$ as a monad module $T \colon T \lto T$, which forms a loose monad in $\Mod(\Y)'$, its underlying loose monad is precisely $T$.
    \[
    \begin{tikzcd}
        {\X'} & {\Mod(\X')'} \\
        & {\X'}
        \arrow["{{\eta_\X}'}", from=1-1, to=1-2]
        \arrow[equals, from=1-1, to=2-2]
        \arrow["{\varepsilon_{\X'}}", from=1-2, to=2-2]
    \end{tikzcd}
    \hspace{4em}
    \begin{tikzcd}[column sep=large]
        {\Mod(\Y)} & {\Mod(\Mod(\Y)')} \\
        & {\Mod(\Y)}
        \arrow["{\eta_{\Mod(\Y)}}", from=1-1, to=1-2]
        \arrow[equals, from=1-1, to=2-2]
        \arrow["{\Mod(\varepsilon_\Y)}", from=1-2, to=2-2]
    \end{tikzcd}
    \qedshift
    \]
\end{proof}

\subsection{Lax-idempotence}
\label{lax-idempotence}

The counit of the 2-adjunction in \cref{Mod-is-right-adjoint} gives, for each \vdc{} $\Y$, a forgetful functor $\varepsilon_\Y \colon \Mod(\Y)' \to \Y$, while the unit gives, for each \emph{normal} \vdc{} $\X$, a normal functor $\eta_\X \colon \X \to \Mod(\X')$. We might hope that these are related by adjointness, and indeed this is true\footnote{This has been independently observed (without proof) by \textcite[Remark~7.8]{fujii2025familial}.}.

\begin{definition}[{\cite[Theorem~10.3.4]{dostal2018two}}]
    A 2-adjunction $L \adj R$ with unit $\eta$ and counit $\varepsilon$ is \emph{lax-idempotent} if there is an adjunction $L\eta \adj \varepsilon L \colon LRL \tto L$ with identity unit.
\end{definition}

\begin{theorem}
    \label{Mod-is-lax-idempotent}
    The 2-adjunction $\ph' \adj \Mod$ is lax-idempotent.
\end{theorem}

\begin{proof}
    We must show that, for each normal \vdc{} $\X$, we have a coreflective adjunction of \vdcs{} as follows.
    \[\begin{tikzcd}
        {\X'} & {\Mod(\X')'}
        \arrow[""{name=0, anchor=center, inner sep=0}, "{{\eta_\X}'}", shift left=2, hook, from=1-1, to=1-2]
        \arrow[""{name=1, anchor=center, inner sep=0}, "{\varepsilon_{\X'}}", shift left=2, from=1-2, to=1-1]
        \arrow["\dashv"{anchor=center, rotate=-90}, draw=none, from=0, to=1]
    \end{tikzcd}\]
    The unit of the adjunction is the identity, the carrier of each identity loose monad $A(1, 1)$ on an object $A$ simply being $A$. The counit ${\eta_\X}' \c \varepsilon_{\X'} \tto 1$ is the transformation specified by, for each loose monad $S$ on $A$, the loose monad morphism $(1_A, \widecheck{\I_S}) \colon A(1, 1) \to S$ defined by the factorisation of the unit $\I_S$ through the nullary opcartesian cell defining $A(1, 1)$; and, for each module $p \colon S \lto T$, the identity cell on $p$. That these satisfy the equations for a loose monad morphism and monad transformation respectively follows by the same reasoning that units in multicategories are initial (\cf~\cite[Proposition~4.12]{arkor2024formal}). The naturality condition of \cref{transformation} is trivially satisfied, since $\varepsilon_{\X'}$ is faithful on cells.
    \[\begin{tikzcd}
        {A(1, 1)} & {B(1, 1)} \\
        S & T
        \arrow["p"{inner sep=.8ex}, "\shortmid"{marking}, from=1-1, to=1-2]
        \arrow[""{name=0, anchor=center, inner sep=0}, "{(1_A, \widecheck{\I_S})}"', from=1-1, to=2-1]
        \arrow[""{name=1, anchor=center, inner sep=0}, "{(1_B, \widecheck{\I_T})}", from=1-2, to=2-2]
        \arrow["p"'{inner sep=.8ex}, "\shortmid"{marking}, from=2-1, to=2-2]
        \arrow["{1_p}"{description}, draw=none, from=0, to=1]
    \end{tikzcd}\]
    The left triangle identity expresses that, for each object $A$, the unit of the loose monad $A(1, 1)$ is the identity. The right triangle identity expresses that, for each loose monad $S$ on $A$, the first component of $(1_A, \widecheck{\I_S})$ is the identity.
\end{proof}

\begin{remark}
    Since the 2-adjunction is lax-idempotent, so too is the induced 2-monad $\Mod({-}')$ on $\VDbln$.\footnote{Geoff Cruttwell and Michael Shulman have independently observed (in personal communication) that the 2-monad $\Mod({-}')$ is lax-idempotent.} This permits a simple characterisation of the pseudoalgebras for the 2-monad, which we shall show in future work are related to the \emph{exact \vdcs{}} of \cite[\S5]{arkor2025nerve} (\cf~\cite[Remark~5.4]{arkor2025nerve}).
\end{remark}

We record a useful consequence of lax-idempotence.

\begin{corollary}
    \label{Mod-unit-is-ff}
    For each normal \vdc{} $\X$, the unit component ${\eta_\X}' \colon \X' \to \Mod(\X')'$ \eqref{Mnd-unit} is \ff.
\end{corollary}

\begin{proof}
    Immediate from \cref{Mod-is-lax-idempotent,coreflection}, since ${\eta_\X}'$ is the left adjoint of a coreflective adjunction.
\end{proof}

\begin{lemma}
    \label{Mod-preserves-ffness}
    If a functor $F \colon \X \to \Y$ between \vdcs{} is \ff, then so is $\Mod(F)' \colon \Mod(\X)' \to \Mod(\Y)'$.
\end{lemma}

\begin{proof}
    The unit preservation condition for a loose monad morphism between identity loose monads implies that it is determined by the tight morphism component. Loose monad transformations are trivially determined by their underlying cells.
\end{proof}

\begin{proposition}
    \label{ff-transpose}
    If a functor $F \colon \X' \to \Y$ from a normal \vdc{} to a \vdc{} is \ff{}, then so is its transpose $\Mod(F)' \c {\eta_\X}' \colon \X' \to \Mod(\X')' \to \Mod(\Y)'$.
\end{proposition}

\begin{proof}
    By \cref{Mod-unit-is-ff,Mod-preserves-ffness}, the transpose is the composite of two \ff{} functors.
\end{proof}

\subsection{Slicing over a monad}
\label{slicing-over-a-monad}

We now have almost all the properties of the $\Mod$ construction that we need for what follows. One further observation will be relevant: when we analyse the relationship between presheaves and discrete fibrations for \dcs{} in \cref{elements-construction}, we shall need to understand how the $\Mod$ construction interacts with slicing. The slice of a \pdc{} over an object is well known~\cite[\S1.7]{grandis1999limits}; more generally, we may slice over a loose monad in a \pdc{}~\cite[\S2.5]{grandis1999limits}. Since the 2-category $\VDbl$ is complete (\cref{VDbl-is-bicomplete}), we may similarly slice over loose monads in a \vdc, observing that a functor of \vdcs{} $\bbn 1 \to \X$ expresses the data of a loose monad in $\X$ (see \cref{loose monads-are-lax-functors}).

\begin{definition}
    \label{slice-vdc}
    Let $\X$ be a \vdc{} and let $T$ be a loose monad in $\X$. The \emph{slice \vdc} $\X/T$ is the slice object in the 2-category of \vdcs, \ie{} the following comma object in $\VDbl$.
    \[\begin{tikzcd}
        {\X/T} & {\bbn 1} \\
        \X & \X
        \arrow["\unit", from=1-1, to=1-2]
        \arrow["\pi"', from=1-1, to=2-1]
        \arrow["T", from=1-2, to=2-2]
        \arrow[shorten <=6pt, shorten >=6pt, Rightarrow, from=2-1, to=1-2]
        \arrow[equals, from=2-1, to=2-2]
    \end{tikzcd}\]
    For any object $X \in \X$ admitting a loose identity, denote by $\X/X$ the slice $\X/X(1, 1)$.
\end{definition}

Let us observe in passing the following, which will be used in the following proposition.

\begin{lemma}
    \label{prime-creates-limits}
    The 2-functor $\ph' \colon \VDbln \to \VDbl$ creates limits. Consequently, $\VDbln$ is a complete 2-category.
\end{lemma}

\begin{proof}
    The first statement may be verified directly; we shall not spell out the details, as it will follow more abstractly from the fact that $\ph'$ is 2-monadic (\cref{F-is-left-adjoint}). The consequence follows from the fact $\VDbl$ is complete (\cref{VDbl-is-bicomplete}).
\end{proof}

When $\X$ is normal (so that all loose identities exist), the slice $\X'/X(1, 1)$ is given by the slice $\X/X$ in the 2-category of normal \vdcs, observing that a normal functor $\bbn 1 \to \X$ expresses the data of an object in $\X$, and that the forgetful functor $\ph' \colon \VDbln \to \VDbl$ preserves limits (\cref{prime-creates-limits}). More generally, the slice $\X'/T$ of a normal \vdc{} $\X$ over a loose monad $T$ is normal.

\begin{proposition}
    \label{slice-normal-vdc}
    Let $\X$ be a normal \vdc{} and let $T$ be a loose monad in $\X$. Denote by $\X/T$ the following comma object in $\VDbln$.
    \[\begin{tikzcd}
        {\X/T} & {\bbn 1} \\
        \X & {\Mod(\X')}
        \arrow["\unit", from=1-1, to=1-2]
        \arrow["\pi"', from=1-1, to=2-1]
        \arrow["T", from=1-2, to=2-2]
        \arrow[between={0.3}{0.7}, Rightarrow, from=2-1, to=1-2]
        \arrow["{\eta_\X}"', hook, from=2-1, to=2-2]
    \end{tikzcd}\]
    The following 2-cell exhibits the slice $\X'/T$ of \cref{slice-vdc}, exhibiting it as a normal \vdc.
    \[\begin{tikzcd}
        {(\X/T)'} & {\bbn 1} \\
        {\X'} & {\Mod(\X')'} \\
        {\X'} & {\X'}
        \arrow[from=1-1, to=1-2]
        \arrow["{\pi'}"', from=1-1, to=2-1]
        \arrow["{T'}", from=1-2, to=2-2]
        \arrow[between={0.3}{0.7}, Rightarrow, from=2-1, to=1-2]
        \arrow["{(\eta_\X)'}"{description}, hook, from=2-1, to=2-2]
        \arrow[""{name=0, anchor=center, inner sep=0}, equals, from=2-1, to=3-1]
        \arrow[""{name=1, anchor=center, inner sep=0}, "{\varepsilon_{\X'}}", from=2-2, to=3-2]
        \arrow[equals, from=3-1, to=3-2]
        \arrow["{=}"{description}, draw=none, from=0, to=1]
    \end{tikzcd}\]
\end{proposition}

\begin{proof}
    First, since $\ph'$ is continuous (\cref{prime-creates-limits}), applying $\ph'$ to the comma object diagram in $\VDbln$ produces a comma object in $\VDbl$. Second, since ${\eta_\X}' \adj \varepsilon_{\X'}$ by \cref{Mod-is-lax-idempotent}, postcomposing $\varepsilon_{\X'}$ (and, implicitly, pasting the unit of the adjunction, which is the identity) preserves the universal property of the comma object.
\end{proof}

We shall be interested in slicing \vdcs{} of the form $\Mod(\X)$, for which the following lemma will be useful.

\begin{lemma}[Pasting law for comma objects\footnote{While this lemma has surely been known since the advent of 2-category theory, we have been unable to find a classical reference. The bicategorical analogue appears, for instance, as \cite[Proposition~1.7]{carboni1994modulated}.}]
    \label{pasting-law}
    Given a diagram of the following shape in a 2-category, in which the right-hand square is a comma object,
    \[\begin{tikzcd}
        A & B & C \\
        D & E & F
        \arrow[from=1-1, to=1-2]
        \arrow[from=1-1, to=2-1]
        \arrow[from=1-2, to=1-3]
        \arrow[from=1-2, to=2-2]
        \arrow[from=1-3, to=2-3]
        \arrow[from=2-1, to=2-2]
        \arrow[between={0.3}{0.7}, Rightarrow, from=2-2, to=1-3]
        \arrow[from=2-2, to=2-3]
    \end{tikzcd}\]
    the left-hand square is a pullback if and only if the outer rectangle is a comma object.
\end{lemma}

We mention in passing the following useful consequence of the pasting law, which can be used to give an alternative proof of \cref{slice-as-pullback}.

\begin{corollary}
    \label{slice-and-pullback}
    Let $f \colon A \to B$ and $a \colon 1 \to A$ be 1-cells in a 2-category with a terminal object, and suppose that the slice object $B/fa$ exists. If $f$ is representably \ff{}, then the slice object $A/a$ exists if and only if the following pullback exists, in which case they are isomorphic.
    \[\begin{tikzcd}
        {A/a} & {B/fa} \\
        A & B
        \arrow[from=1-1, to=1-2]
        \arrow[""{name=0, anchor=center, inner sep=0}, "{\pi_a}"', from=1-1, to=2-1]
        \arrow[""{name=1, anchor=center, inner sep=0}, "{\pi_{fa}}", from=1-2, to=2-2]
        \arrow["f"', from=2-1, to=2-2]
        \arrow["\pb"{description}, draw=none, from=0, to=1]
    \end{tikzcd}\]
\end{corollary}

\begin{proof}
    Follows from applying the pasting law (\cref{pasting-law}) to the following diagram using that, since $f$ is representably \ff{}, postcomposition by $f$ preserves comma objects.
    \[\begin{tikzcd}
        {A/a} & {B/fa} & 1 \\
        A & B & B
        \arrow[from=1-1, to=1-2]
        \arrow["{\pi_a}"', from=1-1, to=2-1]
        \arrow["\unit", from=1-2, to=1-3]
        \arrow["{\pi_{fa}}"{description}, from=1-2, to=2-2]
        \arrow["fa", from=1-3, to=2-3]
        \arrow["f"', from=2-1, to=2-2]
        \arrow[between={0.3}{0.7}, Rightarrow, from=2-2, to=1-3]
        \arrow[equals, from=2-2, to=2-3]
    \end{tikzcd}\qedshift\]
\end{proof}

The following describes the structure of the slice \vdc{} in alternative terms, which shows our definition matches that of \textcite[Definition~4.30]{kawase2025double}.

\begin{proposition}
    \label{slice-as-pullback}
    Let $\X$ be a normal \vdc{} and let $T$ be a loose monad in $\X$. The following diagram forms a pullback in $\VDbln$.
    \[\begin{tikzcd}
        {\X/T} & {\Mod(\X')/T} \\
        \X & {\Mod(\X')}
        \arrow[hook, from=1-1, to=1-2]
        \arrow[""{name=0, anchor=center, inner sep=0}, "\pi"', from=1-1, to=2-1]
        \arrow[""{name=1, anchor=center, inner sep=0}, "\pi", from=1-2, to=2-2]
        \arrow["{\eta_\X}"', hook, from=2-1, to=2-2]
        \arrow["\pb"{description}, draw=none, from=0, to=1]
    \end{tikzcd}\]
\end{proposition}

\begin{proof}
    The outer rectangle in the following diagram forms a comma object by definition (\cref{slice-normal-vdc}), after which the statement follows from the pasting law (\cref{pasting-law}).
    \[\begin{tikzcd}
        {\X/T} & {\Mod(\X')/T} & \bbn1 \\
        \X & {\Mod(\X')} & {\Mod(\X')}
        \arrow[hook, from=1-1, to=1-2]
        \arrow["\pi"', from=1-1, to=2-1]
        \arrow["\unit", from=1-2, to=1-3]
        \arrow["\pi"{description}, from=1-2, to=2-2]
        \arrow["T", from=1-3, to=2-3]
        \arrow["{\eta_\X}"', hook, from=2-1, to=2-2]
        \arrow[between={0.3}{0.7}, Rightarrow, from=2-2, to=1-3]
        \arrow[equals, from=2-2, to=2-3]
    \end{tikzcd}\qedshift\]
\end{proof}

The following strengthens \cite[Corollary 4.34(ii)]{kawase2025double} (which additionally assumes that $\X$ is normal and admits restrictions along invertible tight morphisms). It should be compared to the classical result concerning slices of presheaf categories: since $\Mod$ is lax-idempotent, it behaves in a suitable sense like a free cocompletion of \vdcs{} (namely under \emph{collapses} of loose monads and their modules, \cf~\cite[Definition~5.2]{arkor2025nerve}).

\begin{proposition}
    \label{Mod-slice}
    Let $\X$ be a \vdc{} and let $T$ be a loose monad in $\X$. There is an isomorphism of normal \vdcs{} as follows.
    \[\begin{tikzcd}
        {\Mod(\X/T)} && {\Mod(\X)/T} \\
        & {\Mod(\X)}
        \arrow["\iso", from=1-1, to=1-3]
        \arrow["{\Mod(\pi)}"', from=1-1, to=2-2]
        \arrow["\pi", from=1-3, to=2-2]
    \end{tikzcd}\]
\end{proposition}

\begin{proof}
    $\Mod$ is a right 2-adjoint by \cref{Mod-is-right-adjoint}, and so preserves slice objects, observing that $\Mod(T) \colon \Mod(\bbn 1) \iso \bbn 1 \to \Mod(\X)$ picks out the identity module on the loose monad $T$, which is $T$ itself.
\end{proof}

\section{The free normal \vdc}
\label{powers-and-copowers}

We continue our short detour from exponentiability to further explore the relationship between normal \vdcs{} and \vdcs{}. In particular, we observe that, in addition to admitting a right adjoint, $\ph' \colon \VDbln \to \VDbl$ also admits a left adjoint. Viewing normality as a representability condition for \vdcs{}, which are a kind of generalised multicategory~\cite{burroni1971tcategories}, the existence of the left adjoint $\F \colon \VDbl \to \VDbln$ is suggested by the theory of representability for generalised multicategories~\cite{hermida2001coherent} (though there is a subtlety, in that normality is only a \emph{partial} representability condition). This construction is implicit in the work of \textcite{dawson2006paths} and was made explicit in the work of \textcite{fujii2025familial}.

\begin{definition}[{\cite[\S7.2]{fujii2025familial}}]
    Let $\X$ be a \vdc{}. The \vdc{} $\F(\X)$ is defined as follows.
    \begin{enumerate}
        \item The underlying category is $\u\X$.
        \item The loose morphisms from $X$ to $Y$ comprise the loose morphisms $X \lto Y$ in $\X$ and, when $X = Y$, a (freely added) endomorphism $X(1, 1) \colon X \lto X$.
        \item A cell with the following frame (where $p_1, \ldots, p_n, p$ are loose morphisms in $\X$, and a dashed arrow denotes a finite (possibly empty) chain of the freely added endomorphisms $X_i(1, 1)$),
        \[\begin{tikzcd}
            {X_0} & {X_0} & {X_1} & {X_1} & \cdots & {X_n} & {X_n} \\
            X &&&&&& {X'}
            \arrow["x"', from=1-1, to=2-1]
            \arrow[dashed, from=1-2, to=1-1]
            \arrow["{p_1}"'{inner sep=.8ex}, "\shortmid"{marking}, from=1-3, to=1-2]
            \arrow[dashed, from=1-4, to=1-3]
            \arrow["{p_2}"'{inner sep=.8ex}, "\shortmid"{marking}, from=1-5, to=1-4]
            \arrow["{p_n}"'{inner sep=.8ex}, "\shortmid"{marking}, from=1-6, to=1-5]
            \arrow[dashed, from=1-7, to=1-6]
            \arrow["{x'}", from=1-7, to=2-7]
            \arrow["p"{inner sep=.8ex}, "\shortmid"{marking}, from=2-7, to=2-1]
        \end{tikzcd}\]
        is given by a cell in $\X$ with the following frame (\ie{} given by removing the endomorphisms $X_i(1, 1)$ in the domain).
        \[\begin{tikzcd}
            {X_0} & {X_1} & \cdots & {X_n} \\
            X &&& {X'}
            \arrow["x"', from=1-1, to=2-1]
            \arrow["{p_1}"'{inner sep=.8ex}, "\shortmid"{marking}, from=1-2, to=1-1]
            \arrow["{p_2}"'{inner sep=.8ex}, "\shortmid"{marking}, from=1-3, to=1-2]
            \arrow["{p_n}"'{inner sep=.8ex}, "\shortmid"{marking}, from=1-4, to=1-3]
            \arrow["{x'}", from=1-4, to=2-4]
            \arrow["p"{inner sep=.8ex}, "\shortmid"{marking}, from=2-4, to=2-1]
        \end{tikzcd}\]
        Furthermore, for each tight morphism $x \colon X \to X'$, there is a unique cell with the following frame.
        \[\begin{tikzcd}
            X & \cdots & X \\
            {X'} && {X'}
            \arrow["x"', from=1-1, to=2-1]
            \arrow["{X(1, 1)}"'{inner sep=.8ex}, "\shortmid"{marking}, from=1-2, to=1-1]
            \arrow["{X(1, 1)}"'{inner sep=.8ex}, "\shortmid"{marking}, from=1-3, to=1-2]
            \arrow["x", from=1-3, to=2-3]
            \arrow["{X'(1, 1)}"{inner sep=.8ex}, "\shortmid"{marking}, from=2-3, to=2-1]
        \end{tikzcd}\]
    \end{enumerate}
    Identities and composites are inherited from $\X$. The \vdc{} $\F(\X)$ is normal, the unique cell
    \[\begin{tikzcd}[column sep=large]
        X & X \\
        X & X
        \arrow[equals, from=1-1, to=1-2]
        \arrow[""{name=0, anchor=center, inner sep=0}, equals, from=1-1, to=2-1]
        \arrow[""{name=1, anchor=center, inner sep=0}, equals, from=1-2, to=2-2]
        \arrow["{X(1, 1)}"{inner sep=.8ex}, "\shortmid"{marking}, from=2-2, to=2-1]
        \arrow["{!}"{description}, draw=none, from=0, to=1]
    \end{tikzcd}\]
    witnessing the opcartesian cell associated to the loose identity $X(1, 1)$ on an object $X$.
\end{definition}

\begin{proposition}
    \label{F-is-left-adjoint}
    There is a lax-idempotent and monadic 2-adjunction as follows.
    \[\begin{tikzcd}[column sep=large]
        \VDbln & \VDbl
        \arrow[""{name=0, anchor=center, inner sep=0}, "{\ph'}"', shift right=2, from=1-1, to=1-2]
        \arrow[""{name=1, anchor=center, inner sep=0}, "\F"', shift right=2, from=1-2, to=1-1]
        \arrow["\dashv"{anchor=center, rotate=-90}, draw=none, from=1, to=0]
    \end{tikzcd}\]
\end{proposition}

\begin{proof}
    That $\F$ defines a left 2-adjoint is stated in \cite[\S7.2]{fujii2025familial}. For lax-idempotence, observe that, since $\ph' \adj \Mod$ is lax-idempotent by \cref{Mod-is-lax-idempotent}, it induces a colax-idempotent 2-comonad. A 2-monad that is left-adjoint to a colax-idempotent 2-comonad is lax-idempotent. The $\F'$-algebras are consequently characterised in terms of existence of a left-adjoint retraction of the unit $\X \to \F(\X)'$, which corresponds precisely to the normality condition: the left adjoint provides the interpretation of the formal loose identities in $\F(\X)$ as loose morphisms in $\X$, as well as the interpretation of cells involving them; and the unit of the adjunction and triangle identities specify the nullary opcartesian cells in $\X$.
\end{proof}

Consequently, we have a 2-adjoint triple $\F \adj \ph' \adj \Mod$, inducing an adjoint 2-monad--2-comonad pair $\F' \adj \Mod'$. This situation is useful, as we may take advantage of the calculus of adjoint triples, properties of one adjunction being deducible from properties of the other~\cite{eilenberg1965adjoint}.

\begin{corollary}
    $\ph' \adj \Mod$ is 2-comonadic.
\end{corollary}

\begin{proof}
    Follows from 2-monadicity of $\F \adj \ph'$ by the $\Cat$-enriched analogue of \cite[Proposition~3.3]{eilenberg1965adjoint}.
\end{proof}

\begin{remark}
    \label{no-further-adjoints}
    $\F$ does not preserve the terminal object $\bbn 1$, so does not admit a further left adjoint. $\Mod$ does not admit a further right adjoint (although it does preserve some colimits, \eg{} coproducts). Suppose that it did, so that we have a natural isomorphism as follows for all \vdcs{} $\X$ and normal \vdcs{} $\Y$.
    \[\VDbln(\Mod(\X), \Y) \iso \VDbl(\X, R\Y)\]
    In particular, for any \vdc{} $\X$ containing no loose monads, we have
    \[\VDbl(\X, R\Y) \iso \VDbln(\Mod(\X), \Y) \iso \VDbln(\bbn 0, \Y) \iso \b1\]
    where $\bbn 0 \defeq \o{\b0}$ denotes the empty normal \vdc. Taking $\X$ variously to be the free-standing object, free-standing tight morphism, free-standing loose morphism, and free-standing $n$-ary cell for each $n \in \N$, none of which contain a loose monad, we deduce that $R\Y$ is necessarily the terminal \vdc{}. However, taking $\X$ to be any \vdc{} containing a loose monad and taking $\Y \defeq \bbn 0$, we have
    \[\VDbl(\X, R\bbn 0) \iso \VDbln(\Mod(\X), \bbn 0) \iso \b0\]
    which is a contradiction.
\end{remark}

\begin{remark}
    \label{terminal-vdc-via-F}
    The 2-functor $\Ch \colon \Cat \to \VDbl$ of \cref{chaotic-vdc} is similar in nature to the 2-functor $\F(\o\ph)' \colon \Cat \to \VDbl$, the difference being that the former adjoins a unique loose morphism $!_{a, a'} \colon a \lto a'$ between every pair of objects $a, a'$, whereas the latter does so only for $a = a'$. Consequently, for a category $\b A$, we have $\Ch(\b A) \iso \F(\o{\b A})'$ if and only if $\b A$ has at most one object. In particular, $\F(\o{\b1})' \iso \bbn 1$, since $\Ch$ is a right adjoint and thus preserves the terminal object.
\end{remark}

As we observed in \cref{VDbl-is-bicomplete},  $\VDbl$ is known to be complete and cocomplete as a 2-category and hence admit powers and copowers by categories. However, the methodology used to establish this fact relies upon the construction of limits in 2-categories of coalgebras and colimits in 2-categories of algebras~\cite{dawson2006paths}, and therefore does not produce concrete descriptions of either. In contrast, the adjoint triple between \vdcs{} and normal \vdcs{} gives a concrete description of both powers and copowers (in particular, recovering the concrete description of copowers by $\b2$ in \cref{copower-by-2}).

\begin{proposition}
    \label{power-and-copower}
    The 2-category $\VDbl$ is powered and copowered over $\Cat$. For a category $\b A$ and \vdcs{} $\X$ and $\Y$:
    \begin{align*}
        \b A \copow \X & \iso \F(\o{\b A})' \times \X &
        \b A \pow \Y & \iso \Y^{\F(\o{\b A})'}
    \end{align*}
\end{proposition}

\begin{proof}
    To see that the proposed description of the power is well formed, observe that $\F(\o{\b A})'$ is representable, its only loose morphisms being loose identities, hence is exponentiable by \cref{representable-vdcs-are-exponentiable}. We have $\VDbl(\X, \Y^{\F(\o{\b A})'}) \iso \VDbl(\F(\o{\b A})' \times \X, \Y)$, and so it suffices to verify the proposed description of the copower. Since we already know that $\VDbl$ is copowered (\cref{VDbl-is-bicomplete}), it is necessary only to verify the one-dimensional aspect of the universal property, since this identifies the copower up to isomorphism.

    We have a functor $\b A \to \VDbl(\X, \F(\o{\b A})' \times \X) \iso \VDbl(\X, \F(\o{\b A})') \times \VDbl(\X, \X)$ sending $a \in \b A$ to the pair comprising the constant functor on $a$ and the identity functor on $\X$; we must show that the induced function $\VDbl(\F(\o{\b A})' \times \X, \Y)_0 \to \Cat(\b A, \VDbl(\X, \Y))_0$ is invertible. A functor $F_{\ph} \colon \b A \to \VDbl(\X, \Y)$ comprises, for each object $a \in \b A$, a functor $F_a \colon \X \to \Y$ and, for each morphism $\alpha \colon a \to a'$, a transformation $F_\alpha \colon F_a \tto F_{a'}$, functorial in $a \in \b A$. Given this data, we define a functor $F_{({-}_1)}({-}_2) \colon \F(\o{\b A})' \times \X \to \Y$, sending $a \in \b A$ and $X \in \X$ to $F_a(X) \in \Y$. A loose morphism from $(a', X')$ to $(a, X)$ in $\F(\o{\b A})' \times \X$, which necessarily has $a = a'$ and comprises a loose morphism $p \colon X' \lto X$ in $\X$, is sent to the loose morphism $F_a(p) \colon F_a(X') \lto F_a(X)$ in $\Y$. The action on tight morphisms and cells is analogous. It is easy to verify that this assignment is inverse to the function $\VDbl(\F(\o{\b A})' \times \X, \Y)_0 \to \Cat(\b A, \VDbl(\X, \Y))_0$.
\end{proof}

It follows from \cref{power-and-copower} that the embedding $\F(\o\ph)' \colon \Cat \to \VDbl$ plays an analogous role for \vdcs{} that the cocartesian (\aka{} sequential) multicategory construction plays for multicategories \cites[Proposition~2.1]{pisani2013remarks}[Corollary~3.4]{pisani2014sequential}. For instance, recall from \cite[Corollary~3.10]{pisani2014sequential} that the cocartesian multicategory construction admits both a left adjoint and a right adjoint: in particular, it is left-adjoint to the category of monoids construction. These facts have analogues in the setting of \vdcs. To describe the left adjoint, we make use of the notion of a generalised congruence on a category and its quotient~\cite[Definition~3.6 \& \S3.9]{bednarczyk1999generalized}.

\begin{definition}
    \label{loose-connected-components}
    Given a \vdc{} $\X$, denote by $\u Q_\X \colon \u\X \to \LoCo(\X)$ the quotient functor with respect to the generalised congruence $\sim$ on the category $\u\X$, for which $\sim$ is generated by the relation: $X \sim X'$ if there exists a loose morphism $X' \lto X$ in $\X$, and $x \sim x'$ if there exists a cell with a frame as follows in $\X$.
    \[\begin{tikzcd}
        {X_0} & \cdots & {X_n} \\
        X && {X'}
        \arrow["x"', from=1-1, to=2-1]
        \arrow["{p_1}"'{inner sep=.8ex}, "\shortmid"{marking}, from=1-2, to=1-1]
        \arrow["{p_n}"'{inner sep=.8ex}, "\shortmid"{marking}, from=1-3, to=1-2]
        \arrow["{x'}", from=1-3, to=2-3]
        \arrow["p"{inner sep=.8ex}, "\shortmid"{marking}, from=2-3, to=2-1]
    \end{tikzcd}\qedshift\]
\end{definition}

\begin{proposition}
    \label{sequential-adjoint}
    The 2-functor $\F(\o\ph)' \colon \Cat \to \VDbl$ is \ff{} and admits both a left and a right adjoint.
    \[\LoCo \adj \F(\o\ph)' \adj \u{\Mod\ph'} \colon \VDbl \to \Cat\]
\end{proposition}

\begin{proof}
    First, \ffness{} follows because the action on loose morphisms is uniquely determined. The right adjoint is given by composition of the 2-adjunctions of \cref{Cat-graphs-into-VDCs,Mod-is-right-adjoint,F-is-left-adjoint}.
    \[\begin{tikzcd}[column sep=huge]
        \Cat & \VDbl & \VDbl
        \arrow[""{name=0, anchor=center, inner sep=0}, "{\o\ph}", shift left=2, hook, from=1-1, to=1-2]
        \arrow[""{name=1, anchor=center, inner sep=0}, "{\u\ph}", shift left=2, from=1-2, to=1-1]
        \arrow[""{name=2, anchor=center, inner sep=0}, "{\F\ph'}", shift left=2, from=1-2, to=1-3]
        \arrow[""{name=3, anchor=center, inner sep=0}, "{\Mod\ph'}", shift left=2, from=1-3, to=1-2]
        \arrow["\dashv"{anchor=center, rotate=-90}, draw=none, from=0, to=1]
        \arrow["\dashv"{anchor=center, rotate=-90}, draw=none, from=2, to=3]
    \end{tikzcd}\]

    For the left adjoint, observe that, given a \vdc{} $\X$, the functor ${\u Q_\X \colon \u\X \to \LoCo(\X)}$ of categories underlies a unique functor $Q_\X \colon \X \to \F(\o{\LoCo(\X)})'$ of \vdcs{}, since the only loose morphisms and the only cells in the latter are identities.

    Given a functor $F \colon \X \to \F(\o{\b Y})'$ of \vdcs{}, the underlying functor $\u F \colon \u\X \to \b Y$ of categories induces a generalised congruence~\cite[Definition~3.2]{bednarczyk1999generalized}. Given a loose morphism $p \colon X' \lto X$ in $\X$, we necessarily have $F(X') = F(X)$, because the only loose morphisms in $\F(\o{\b Y})'$ are identities; similarly, given a cell $\chi$ framed by tight morphisms $x$ and $x'$, we necessarily have $F(x) = F(x')$. Consequently, the generalised congruence of \cref{loose-connected-components} is a sub-congruence of the generalised congruence induced by $\u F$, so that, by \cite[Corollary~3.11]{bednarczyk1999generalized}, there is a unique functor $\tilde F \colon \LoCo(\X) \to \b Y$ rendering the following diagram on the left commutative; and which is consequently also the unique functor rendering the following diagram on the right commutative.
    \[
    \begin{tikzcd}
        {\LoCo(\X)} & \\
        {\u\X} & {\b Y}
        \arrow["{\tilde F}", dashed, from=1-1, to=2-2]
        \arrow["{\u Q_\X}", from=2-1, to=1-1]
        \arrow["{\u F}"', from=2-1, to=2-2]
    \end{tikzcd}
    \hspace{6em}
    \begin{tikzcd}
        {\F(\o{\LoCo(\X)})'} & \\
        {\u\X} & {\F(\o{\b Y})'}
        \arrow["{\F(\o{\tilde F})'}", dashed, from=1-1, to=2-2]
        \arrow["{Q_\X}", from=2-1, to=1-1]
        \arrow["F"', from=2-1, to=2-2]
    \end{tikzcd}
    \]
    To show this satisfies the two-dimensional universal property, it suffices by \cite[Proposition~3.1]{blackwell1989two} to show that $\LoCo$ preserves copowers by $\b2$. Using \cref{power-and-copower}, preservation amounts to invertibility of the canonical functor $\LoCo(\F(\o{\b 2})' \times \X) \to \b2 \times \LoCo(\X)$ for all \vdcs{} $\X$, which holds because the generalised congruence on $\u{\F(\o{\b 2})' \times \X}$ is determined by that on $\u\X$, since $\F(\o{\b 2})$ has no nontrivial loose morphisms.
\end{proof}

\begin{corollary}
    $\F(\o\ph)' \colon \Cat \to \VDbl$ preserves exponentials.
\end{corollary}

\begin{proof}
    Since $\F(\o\ph)'$ is right-adjoint by \cref{sequential-adjoint}, it preserves powers, \ie{} $\F(\o{\b X^{\b A}})'$ exhibits the power $\b A \pow \F(\o{\b X})'$. By \cref{power-and-copower}, the power is precisely the exponential $(\F(\o{\b X})')^{\F(\o{\b A})'}$.
\end{proof}

\begin{remark}
    Since $\VDbl$ is copowered over $\Cat$, we obtain, for each fixed \vdc{} $\X$, a 2-adjunction $\ph \copow \X \adj \VDbl(\X, {-}) \colon \VDbl \to \Cat$.
    \begin{enumerate}
        \item When $\X = \o{\b1}$, this recovers the free--forgetful 2-adjunction $\o\ph \adj \u\ph$ of \cref{Cat-graphs-into-VDCs}.
        \item When $\X = \bbn 1$, this recovers the 2-adjunction $\F(\o\ph)' \adj \u{\Mod\ph'}$ of \cref{sequential-adjoint}.
        \qedhere
    \end{enumerate}
\end{remark}

Therefore, $\Cat$ is embedded coreflectively in $\VDbl$ in two ways, identifying the categories as both the `loosely empty' \vdcs{} (having no loose morphisms) and the `loosely discrete' \vdcs{} (having only loose identities); and is embedded reflectively in two ways, identifying the categories as both the `loosely discrete' \vdcs{} and the `loosely codiscrete' \vdcs{} (by \cref{chaotic-vdc}). Note that $\u{\F(\o\ph)'} \colon \Cat \to \Cat$ is equal to the identity, so that $\u\ph \colon \VDbl \to \Cat$ provides a retraction of each of the reflective and coreflective inclusions.

\section{Lax functors as monads}
\label{lax-functors-as-monads}

We are now in position to explain the relationship between exponentiation for \vdcs{} and the \vdc{} of lax functors introduced by \textcite{pare2011yoneda}. As a warm up, we observe that the category of loose monads and monad morphisms in a \vdc{} $\X$ is given by the category of functors from the terminal \vdc{} $\bbn 1$ into $\X$.

\begin{lemma}
    \label{loose monads-are-lax-functors}
    $\VDbl(\bbn 1, \X) \iso \u{\Mod(\X)'}$.
\end{lemma}

\begin{proof}
    This is immediate from the fact that $\bbn 1 \iso \bb F(\o{\b1})'$ (\cref{terminal-vdc-via-F}) together with \cref{sequential-adjoint}. For concreteness, we spell out the correspondence. A functor $\bbn1 \to \X$ sends the unique object of $\bbn1$ to an object $\crr T$ in $\X$; the unique loose morphism to $T \colon \crr T \lto \crr T$; and the unique $n$-ary cell to an $n$-ary cell of the following form.
    \[\begin{tikzcd}
        {\crr T} & {\crr T} & \cdots & {\crr T} & {\crr T} \\
        {\crr T} &&&& {\crr T}
        \arrow[""{name=0, anchor=center, inner sep=0}, equals, from=1-1, to=2-1]
        \arrow["T"'{inner sep=.8ex}, "\shortmid"{marking}, from=1-2, to=1-1]
        \arrow["T"'{inner sep=.8ex}, "\shortmid"{marking}, from=1-3, to=1-2]
        \arrow["T"'{inner sep=.8ex}, "\shortmid"{marking}, from=1-4, to=1-3]
        \arrow["T"'{inner sep=.8ex}, "\shortmid"{marking}, from=1-5, to=1-4]
        \arrow[""{name=1, anchor=center, inner sep=0}, equals, from=1-5, to=2-5]
        \arrow["T"{inner sep=.8ex}, "\shortmid"{marking}, from=2-5, to=2-1]
        \arrow["{{\circ^n_T}}"{description}, draw=none, from=1, to=0]
    \end{tikzcd}\]
    This clearly suffices to define a loose monad in $\X$, defining the multiplication cell $\c_T \defeq \c^2_T$ and the unit cell $\I_T \defeq \c^0_T$. Conversely, given a loose monad in $\X$, we may define an $n$-ary multiplication cell $\c^n_T$ in terms of repeated binary multiplication $\c_T$ (for $n > 1$) or the unit cell $\I_T$ (for $n = 0$); the associativity and unitality equations imply that this choice is unique. The correspondence between transformations and loose monad morphisms is then essentially immediate.
\end{proof}

The following observation, which extends the above from an isomorphism of categories to an isomorphism of \vdcs{}, provides our original motivation for considering exponentiation of \vdcs.

\begin{theorem}
    \label{Lax-is-Mod}
    Let $\A$ be a \pdc{} and let $\X$ be a \vdc. The \vdc{} $\Lax(\A, \X)$ of \cite[Theorem~4.3]{pare2011yoneda}\;\footnotemark{} is isomorphic to $\Mod(\X^\A)'$.
    \footnotetext{\citeauthor{pare2011yoneda} only defined $\Lax(\A, \X)$ under the additional assumption that $\X$ is representable, but the definition extends in the evident way.}%
\end{theorem}

The reader who does not have a copy of \cite{pare2011yoneda} to hand is invited to unwind the definition of $\Mod(\X^\A)'$, whereupon they will arrive at \citeauthor{pare2011yoneda}'s definition; we shall not reproduce the explicit description here, which spans more than five pages \ibid.

\begin{proof}
    First, we observe that the underlying categories are isomorphic. Using \cref{Dbl_l-in-VDbl} for the first isomorphism, \cref{representable-vdcs-are-exponentiable} for the second isomorphism, and \cref{loose monads-are-lax-functors} for the final isomorphism, we have the following.
    \[\u{\Lax(\A, \X)} \iso \VDbl(\A, \X) \iso \VDbl(\bbn 1, \X^\A) \iso \u{\Mod(\X^\A)'}\]
    Next, we must check the data of the loose morphisms and cells. A loose morphism in $\Mod(\X^\A)$ is given by a module between loose monads. Explicitly, it is given by a loose morphism in $\X^\A$, \ie a graph morphism, equipped with two cells in $\X^\A$ corresponding to the left and right actions. Comparing with the data and axioms for a module in \cite[Definition~3.2]{pare2011yoneda}, (M1 \& M2) express the data of a graph morphism; (M3) expresses the data of the left and right actions; (M4) expresses functoriality of the graph morphism (see \cref{graph-morphism}); (M5) expresses naturality of the left and right actions; (M6) expresses the compatibility of the right action with the multiplication, the compatibility of the left action with the right action, and the compatibility of the left action with the multiplication; (M7) expresses the compatibility of the right action with the unit, and the compatibility of the left action with the unit.

    A cell in $\Mod(\X^\A)$ is given by a cell in $\X^\A$ satisfying several compatibility conditions. Comparing with the data and axioms for a multimodulation in \cite[Definition~4.2]{pare2011yoneda}, (mm1) expresses the data of a cell; (mm2) expresses naturality; (mm3\textsubscript{l}) expresses to the compatibility condition for the left action with itself; (mm3\textsubscript{r}) expresses to the compatibility condition for the right action with itself; (mm3\textsubscript{i}) expresses to the compatibility condition for the left action with the right action; (mm3\textsubscript{0}) expresses the nullary compatibility condition.

    \textcite{pare2011yoneda} does not define composition of cells in $\Lax(\A, \X)$, therein deferring them to \cite{pare2013composition}, where they are also not defined explicitly. However, it is clear that they must coincide with those of $\Mod(\X^\A)$.
\end{proof}

Motivated by \cref{Lax-is-Mod}, we may extend \citeauthor{pare2011yoneda}'s construction to exponentiable \vdcs.

\begin{definition}
    Let $\A$ and $\X$ be \vdcs{} and suppose that $\A$ is exponentiable. Define the \vdc{} $\Lax(\A, \X) \defeq \Mod(\X^\A)$.
\end{definition}

Beyond being purely of conceptual interest, this observation has many practical consequences. For one, it provides a simpler construction of $\Lax(\A, \B)$ that makes concrete computations substantially easier; for another, it provides a universal property of $\Lax(\A, \B)$ that means that concrete computations are often rendered unnecessary. In the remainder of the paper, we will explore some of these consequences.

First, normality of the \vdc{} of lax functors follows as immediate corollary of \cref{Lax-is-Mod}.

\begin{corollary}[{\cites[Theorem~4.3]{pare2011yoneda}[Theorem~5.1.10]{pare2013composition}}]
    \label{Lax-is-normal}
    $\Lax(\A, \X)$ is normal.
\end{corollary}

\begin{proof}
    Immediate from \cref{Lax-is-Mod,Mod-is-normal}.
\end{proof}

\begin{remark}
    \label{representability-of-Lax}
    If $\X$ admits non-nullary composites, restrictions, and local reflexive coequalisers, then $\Mod(\X)$ is representable: this is stated in \cite[Theorem~A.1]{cruttwell2010unified} under the further assumptions that $\X$ admits arbitrary composites and local coequalisers, but inspection of the proof reveals the stronger assumptions to be unnecessary. It is straightforward to show that, if $\X$ admits restrictions and local reflexive coequalisers, then so too does each exponential $\X^\A$. Consequently, supposing the conjecture in \cref{representability-of-exponential} holds, it would follow that, given a \pdc{} $\A$ satisfying a suitable factorisation condition and a locally cocomplete \pdc{} $\X$ with restrictions, the \vdc{} $\Lax(\A, \X)$ is representable (and is furthermore locally cocomplete with restrictions), which would recover \cite[Theorem~4.0.1]{pare2013composition}. Our expectation is that this proof strategy would result in a substantial simplication of the proof \ibid, given the relative simplicity of $\X^\A$ compared to $\Lax(\A, \X)$.
\end{remark}

In the following subsections, we will explore some of the more interesting consequences of \cref{Lax-is-Mod}.

\subsection{Enrichment of double categories}
\label{enrichment}

The reader of \cite{pare2011yoneda} may be led to wonder whether the assignment $(\A, \B) \mapsto \Lax(\A, \B)$ of pairs of \pdcs{} to the normal \vdc{} of lax functors between them defines an enrichment of \pdcs{} in normal \vdcs{}. While verifying this by hand is not a difficult exercise, it does involve some tedious calculations. On the other hand, it is a straightforward corollary of \cref{Lax-is-Mod}. We require a preparatory lemma.

\begin{lemma}
    \label{enriched-subcategory}
    Let $(\cl D, \otimes, I)$ be a symmetric monoidal bicategory and let $\C \ffto \cl D$ be a full sub-bicategory. Suppose that, for each $A, B \in \C$, the internal hom $[A, B]$ exists in $\cl D$. Then $\C$ is canonically equipped with a $\cl D$-enriched structure.
\end{lemma}

\begin{proof}
    The enriched structure on $\C$ is defined in the same way as the canonical enriched structure on a closed monoidal bicategory~\cite[Proposition~3.9.6]{lack1995algebra}, with $\C(A, B) \defeq [A, B]$. Full faithfulness ensures that the underlying bicategory of the enriched bicategory coincides with $\C$, since $\cl D(I, [A, B]) \equiv \cl D(A, B) \equiv \C(A, B)$.
\end{proof}

\begin{corollary}
    \label{Dbl-is-enriched}
    The 2-category $\Dbl_l$ of pseudo double categories and lax functors admits the following enrichments.
    \begin{enumerate}
        \item In the cartesian monoidal 2-category $\VDbl$ via $\Dbl_l(\A, \B) \defeq \B^\A$.
        \item In the cartesian monoidal 2-category $\VDbln$ via $\Dbl_l(\A, \B) \defeq \Lax(\A, \B)$.
    \end{enumerate}
\end{corollary}

\begin{proof}
    First, the inclusion $\Dbl_l \to \VDbl$ is a \ff{} 2-functor (\cref{Dbl_l-in-VDbl}). (1) thus follows from \cref{enriched-subcategory} using that representable \vdcs{} are exponentiable (\cref{representable-vdcs-are-exponentiable}). For (2), we use change of base along $\Mod$, using that $\Mod$ is right-adjoint and so preserves products, along with the characterisation of $\Lax(\A, \B)$ in \cref{Lax-is-Mod}. It is clear that the underlying 2-category is precisely $\Dbl_l$, since $\VDbln(\bbn 1, \Lax(\A, \B)) \iso \VDbl(\A, \B)$.
\end{proof}

\begin{remark}
    Bicategories enriched in the cartesian monoidal 2-category $\VDbln$ are generalisations of the locally cubical bicategories of \cite[Definition~11]{garner2009low}, which are bicategories enriched in the cartesian monoidal 2-category $\Dbl_p$ of \pdcs{} and pseudo functors (which, by the remark at the start of \cref{normality}, is biequivalent to the 2-category of \pdcs{} and \emph{normal} pseudo functors). Note that each of the 2-functors in the following pseudo-commutative square preserves products and so, by change of base for enriched bicategories~\cite[\S13.2]{garner2016enriched}, every locally cubical bicategory induces bicategories enriched in $\Dbl_l$, $\VDbln$, and $\VDbl$. However, only change of base along the vertical functors below preserves underlying 2-categories (since normal functors from $\bbn 1$ pick out objects, whereas non-normal functors from $\bbn 1$ pick out loose monads).
    \[\begin{tikzcd}
        {\Dbl_p} & {\Dbl_l} \\
        \VDbln & \VDbl
        \arrow[from=1-1, to=1-2]
        \arrow[from=1-1, to=2-1]
        \arrow[hook, from=1-2, to=2-2]
        \arrow[from=2-1, to=2-2]
    \end{tikzcd}\qedshift\]
\end{remark}

\subsection{Skew-closed structure}
\label{skew-closed-structure}

The universal properties of exponentiation and of $\Mod$ supply $\Lax$ with a universal property of its own.

\begin{lemma}
    \label{UP-of-Lax}
    For each normal \vdc{} $\W$, exponentiable \vdc{} $\A$, and \vdc{} $\X$, we have a 2-natural isomorphism of categories:
    \[\VDbl(\U\W \times \A, \X) \iso \VDbln(\W, \Lax(\A, \X))\]
    Consequently, for each \vdc{} $\W$, exponentiable \vdc{} $\A$, and \vdc{} $\X$, we have a 2-natural isomorphism of categories:
    \[\VDbl(\U{(\F\W)} \times \A, \X) \iso \VDbl(\W, \U{\Lax(\A, \X)})\]
\end{lemma}

\begin{proof}
    For the first isomorphism, we use \cref{Lax-is-Mod} together with the following 2-adjunction.
    \[\begin{tikzcd}[column sep=large]
        \VDbln & \VDbl & \VDbl
        \arrow[""{name=0, anchor=center, inner sep=0}, "{\ph'}", shift left=2, from=1-1, to=1-2]
        \arrow[""{name=1, anchor=center, inner sep=0}, "\Mod", shift left=2, from=1-2, to=1-1]
        \arrow[""{name=2, anchor=center, inner sep=0}, "{\ph \times \A}", shift left=2, from=1-2, to=1-3]
        \arrow[""{name=3, anchor=center, inner sep=0}, "{\ph^\A}", shift left=2, from=1-3, to=1-2]
        \arrow["\dashv"{anchor=center, rotate=-90}, draw=none, from=0, to=1]
        \arrow["\dashv"{anchor=center, rotate=-90}, draw=none, from=2, to=3]
    \end{tikzcd}\]
    For the second isomorphism, we precompose the adjunction $\F \adj \U\ph$.
\end{proof}

\begin{remark}
    In particular, for a category $\b A$ and \vdcs{} $\X$ and $\Y$, where $\X$ is exponentiable, we have the following chain of natural isomorphisms,
    \begin{align*}
        \VDbl(\F(\o{\b A})' \times \X, \Y) & \iso \VDbl(\o{\b A}, \Lax(\X, \Y)') \tag{\cref{UP-of-Lax}} \\
            & \iso \Cat(\b A, \u{\Lax(\X, \Y)'}) \tag{\cref{Cat-graphs-into-VDCs}} \\
            & \iso \Cat(\b A, \VDbl(\X, \Y)) \tag{\cref{Lax-is-Mod}}
    \end{align*}
    recovering the description of the copower $\b A \copow \X$ in \cref{power-and-copower}. However, note that this does not give an alternative proof of \cref{power-and-copower}, as it imposes the stronger assumption that $\X$ be exponentiable.
\end{remark}

\Cref{UP-of-Lax} indicates that $\Lax({-}, {-})$ forms something resembling a closed monoidal structure on the 2-category of \vdcs. Since not every \vdc{} is exponentiable, such a monoidal structure will be at best partially closed, but we may still hope that the monoidal structure itself is totally defined. This is, in fact, very close to being the case: while the closed structure expressed by $\Lax({-}, {-})$ does not correspond to a monoidal structure, it does correspond to a \emph{left-skew monoidal} structure in the sense of \textcite{szlachanyi2012skew}.

\begin{theorem}
    \label{partial-closed-structure}
    The assignment $(\W, \X) \mapsto \U{(\F\W)} \times \X$ equips $\VDbl$ with the structure of a right-normal left-skew monoidal 2-category $(\VDbl, \skt, \bbn 1)$. Furthermore, for each exponentiable \vdc{} $\A$, we have $\ph \skt \A \adj \U{\Lax(\A, {-})}$.
\end{theorem}

\begin{proof}
    Every 2-monad admits a unique colax monoidal structure with respect to cartesian monoidal structure, and so the given assignment defines a left-skew monoidal structure by the $\Cat$-enriched analogue of \cite[Example~3.8]{lack2012skew}. Right-normality follows because $(\F\W)' \times \bbn 1 \iso (\F\W)'$. That this skew-monoidal structure is right-closed with respect to exponentiable \vdcs{} follows from \cref{UP-of-Lax}.
\end{proof}

\begin{corollary}
    The 2-category $\Dbl_l$ is equipped with right-normal left-skew monoidal structure $(\skt, \bbn 1)$.
\end{corollary}

\begin{proof}
    Since $\Dbl_l$ is a full sub-2-category of $\VDbl$ (\cref{Dbl_l-in-VDbl}), it suffices to check that $\skt$ restricts to pseudo double categories. This is trivial, as $\F$ preserves representability, as does the cartesian product.
\end{proof}

\begin{remark}
    One might hope that the universal property of \cref{UP-of-Lax} extends from an isomorphism of categories to an isomorphism of \vdcs{}, \ie{} that the tensor--hom adjunction holds internally. However, while we do obtain a functor
    \[\Lax(\A \skt \B, \X) \to \Lax(\A, \Lax(\B, \X)')\]
    for $\A$ and $\B$ exponentiable, the functor is not invertible, since $\skt$ is merely skew monoidal (in particular, since the associator is non-invertible).

    It is worth observing that we also have a functor
    \[\Lax(\A, \Lax(\B, \X)') \to \Lax(\A \times \B, \X)\]
    for $\A$ and $\B$ exponentiable. This is obtained by the following sequence of 2-natural transformations,
    \begin{iffseq}
        \W \to \Lax(\A, \Lax(\B, \X)')' \\
        \W \skt \A \to \Lax(\B, \X)' \\
        (\W \skt \A) \skt \B \to \X \internaldblbackslash\cmidrule{1-1}
        \W \skt (\A \times \B) \to \X \\
        \W \to \Lax(\A \times \B, \X)'
    \end{iffseq}
    where the noninvertible step is given by
    precomposition with the functor
    \[\W \skt (\A \times \B) = (\F\W)' \times (\A \times \B) \iso ((\F\W)' \times \A) \times \B \to \F((\F\W)' \times \A)' \times \B = (\W \skt \A) \skt \B\]
    induced by the unit of the adjunction $\F \adj \ph'$. In particular, taking $\A = \B = \bbn 1$, we obtain a functor
    \[\Mod(\Mod(\X)') \iso \Lax(\bbn 1, \Lax(\bbn 1, \X)') \to \Lax(\bbn 1 \times \bbn 1, \X) \iso \Mod(\X)\]
    which coincides with the whiskering $\Mod \c \varepsilon_\X$ (\cref{Mod-is-right-adjoint}). This suggests a connection to the theory of distributive laws for lax functors between 2-categories and between strict \dcs~\cite{faul20212dimensional,femic2023bifunctor}.
\end{remark}

\subsection{Exponentiable normal \vdcs}

The structure of exponentials in $\VDbl$ may at first appear surprising, since the objects of an exponential are not functors of \vdcs, in contrast to the behaviour of exponentials in familiar 2-categories like $\Cat$. This distinction is encapsulated in the difference between the \vdc{} $\o{\b1}$, which represents the forgetful functor $\u\ph \colon \VDbl \to \Cat$, and the \vdc{} $\bbn1$, which is the terminal object in $\VDbl$. For normal \vdcs{}, there is no such distinction: the forgetful functor $\u\ph \colon \VDbln \to \Cat$ is indeed represented by $\bbn1$ (which is terminal in $\VDbln$ as well as $\VDbl$).
Consequently, exponentials of normal \vdcs{} behave more familiarly than exponentials of arbitrary \vdcs. In fact, we may obtain a description of exponentials of normal \vdcs{} from our description of \vdcs{} of lax functors.

\begin{definition}
    Let $\A$ be an exponentiable \vdc{} and let $\X$ be a \vdc{}. Denote by $\LaxN(\A, \X)$ the normal full sub-\vdc{} of $\Lax(\A', \X')$ spanned by the normal functors.
\end{definition}

\begin{theorem}
    \label{representable-normal-vdcs-are-exponentiable}
    Let $\A$ be a \pdc{} and let $\X$ be a normal \vdc. $\LaxN(\A, \X)$ exhibits the exponential $\X^\A$ in the 2-category $\VDbln$.
\end{theorem}

\begin{proof}
    For each normal \vdc{} $\W$, normal functors $G \colon \W \to \LaxN(\A, \X)$ are precisely normal functors $\W \to \Lax(\A', \X')$ for which, for each object $W \in \W$, the functor $G(W) \colon \A \to \X$ is normal, \ie{} for each object $A \in \A$, $G(W)(A(1, 1)) = (G(W)(A))(1, 1)$. However, normal functors $G \colon \W \to \Lax(\A', \X')$ are, by \cref{UP-of-Lax}, in 2-natural bijection with functors $F \colon (\W \times \A)' \iso \W' \times \A' \to \X'$. Normality of $G(W)$ then corresponds to normality of the functor $F(W, {-}) \colon \A' \to \X'$, which is precisely the condition that $F(W(1, 1), A(1, 1)) = F(W, A)(1, 1)$. Consequently, for each $G(W)$ to be normal is precisely for $F$ to be normal. The two-dimensional universal property of the exponential follows from the fact that $\VDbln$ is a locally full sub-2-category of $\VDbl$.
\end{proof}

This description of exponentials of normal \vdcs{} allows us to internalise the adjunction $\ph' \adj \Mod$ of \cref{Mod-is-right-adjoint}, at least with respect to exponentiable \vdcs.

\begin{corollary}
    \label{Mod-is-internal-right-adjoint}
    For every \pdc{} $\A$ and \vdc{} $\X$, there is an isomorphism of normal \vdcs{} $\Lax(\A', \X) \iso \LaxN(\A, \Mod(\X))$.
\end{corollary}

\begin{proof}
    We have the following isomorphism, 2-natural in $\W$.
    \begin{align*}
        \VDbln(\W, \Lax(\A', \X)) & \iso \VDbl(\W' \times \A', \X) \tag{\cref{UP-of-Lax}} \\
        & \iso \VDbl((\W \times \A)', \X) \tag{$\ph'$ is continuous} \\
        & \iso \VDbln(\W \times \A, \Mod(\X)) \tag{\cref{Mod-is-right-adjoint}} \\
        & \iso \VDbln(\W, \LaxN(\A, \Mod(\X))) \tag{\cref{representable-normal-vdcs-are-exponentiable}}
    \end{align*}
\end{proof}

\subsection{Exponentiable \ve s}

As mentioned in \cref{virtual-equipments}, a class of normal \vdcs{} of particular interest are the \ve s. Given that \ve s provide a convenient setting for \fct, it is useful to have techniques to construct new \ve s from existing ones. As we shall show, exponentiation provides one such technique.

\begin{proposition}
    \label{restrictions-in-exponential}
    Let $\A$ be a representable \vdc{} and let $\X$ be a \vdc. If $\X$ admits restrictions, then so does $\X^\A$.
\end{proposition}

\begin{proof}
    Consider a diagram of the following shape in $\X^\A$.
    \[\begin{tikzcd}
        F & {F'} \\
        G & {G'}
        \arrow["\Phi"', from=1-1, to=2-1]
        \arrow["{\Phi'}", from=1-2, to=2-2]
        \arrow["\Xi"{inner sep=.8ex}, "\shortmid"{marking}, from=2-2, to=2-1]
    \end{tikzcd}\]
    We define a loose morphism $\Xi(\Phi, \Phi')$ in $\X^\A$ as follows, using the description of the latter in \cref{graph-morphism}.
    \begin{enumerate}
        \item Each loose morphism $p \colon X \lto Y$ in $\A$ is sent to the loose morphism $\Xi(\Phi_Y, \Phi'_X)(p)$ defined as the following restriction in $\X$.
        \[\begin{tikzcd}[column sep=6em]
            FY & {F'X} \\
            GY & {G'X}
            \arrow[""{name=0, anchor=center, inner sep=0}, "{\Phi_Y}"', from=1-1, to=2-1]
            \arrow["{\Xi(\Phi_Y, \Phi'_X)(p)}"'{inner sep=.8ex}, "\shortmid"{marking}, from=1-2, to=1-1]
            \arrow[""{name=1, anchor=center, inner sep=0}, "{\Phi'_X}", from=1-2, to=2-2]
            \arrow["{\Xi p}"{inner sep=.8ex}, "\shortmid"{marking}, from=2-2, to=2-1]
            \arrow["\cart"{description}, draw=none, from=0, to=1]
        \end{tikzcd}\]
        \item Each cell in $\A$
        \[\begin{tikzcd}
            Y & X \\
            {Y'} & {X'}
            \arrow[""{name=0, anchor=center, inner sep=0}, "y"', from=1-1, to=2-1]
            \arrow["p"'{inner sep=.8ex}, "\shortmid"{marking}, from=1-2, to=1-1]
            \arrow[""{name=1, anchor=center, inner sep=0}, "x", from=1-2, to=2-2]
            \arrow["{p'}"{inner sep=.8ex}, "\shortmid"{marking}, from=2-2, to=2-1]
            \arrow["\alpha"{description}, draw=none, from=1, to=0]
        \end{tikzcd}\]
        induces a unique cell $\Xi(\Phi_Y, \Phi'_X)(\alpha)$ satisfying the following equation.
        \[
        \begin{tikzcd}[column sep=6em]
            FY & {F'X} \\
            GY & {G'X} \\
            {GY'} & {G'X'}
            \arrow[""{name=0, anchor=center, inner sep=0}, "{\Phi_Y}"', from=1-1, to=2-1]
            \arrow["{\Xi(\Phi_Y, \Phi'_X)(p)}"'{inner sep=.8ex}, "\shortmid"{marking}, from=1-2, to=1-1]
            \arrow[""{name=1, anchor=center, inner sep=0}, "{\Phi'_X}", from=1-2, to=2-2]
            \arrow[""{name=2, anchor=center, inner sep=0}, "Gy"', from=2-1, to=3-1]
            \arrow["{\Xi p}"{description}, from=2-2, to=2-1]
            \arrow[""{name=3, anchor=center, inner sep=0}, "{G'x}", from=2-2, to=3-2]
            \arrow["{\Xi p'}"{inner sep=.8ex}, "\shortmid"{marking}, from=3-2, to=3-1]
            \arrow["\cart"{description}, draw=none, from=0, to=1]
            \arrow["{\Xi\alpha}"{description}, draw=none, from=2, to=3]
        \end{tikzcd}
        \quad = \quad
        \begin{tikzcd}[column sep=7em]
            FY & {F'X} \\
            {FY'} & {F'X'} \\
            {GY'} & {G'X'}
            \arrow[""{name=0, anchor=center, inner sep=0}, "Fy"', from=1-1, to=2-1]
            \arrow["{\Xi(\Phi_Y, \Phi'_X)(p)}"'{inner sep=.8ex}, "\shortmid"{marking}, from=1-2, to=1-1]
            \arrow[""{name=1, anchor=center, inner sep=0}, "{F'x}", from=1-2, to=2-2]
            \arrow[""{name=2, anchor=center, inner sep=0}, "{\Phi_{Y'}}"', from=2-1, to=3-1]
            \arrow["{\Xi p'(\Phi_{Y'}, \Phi'_{X'})}"{description}, from=2-2, to=2-1]
            \arrow[""{name=3, anchor=center, inner sep=0}, "{\Phi_{X'}}", from=2-2, to=3-2]
            \arrow["{\Xi p'}"{inner sep=.8ex}, "\shortmid"{marking}, from=3-2, to=3-1]
            \arrow["{\Xi(\Phi_Y, \Phi'_X)(\alpha)}"{description}, draw=none, from=0, to=1]
            \arrow["\cart"{description}, draw=none, from=2, to=3]
        \end{tikzcd}
        \]
    \end{enumerate}
    That this assignment preserves identities and composites of cells follows from the aforementioned uniqueness. The cartesian cells in (1) defines a cell of the following shape in $\X^\A$.
    \[\begin{tikzcd}[column sep=large]
        F & {F'} \\
        G & {G'}
        \arrow["\Phi"', from=1-1, to=2-1]
        \arrow["{\Xi(\Phi, \Phi')}"'{inner sep=.8ex}, "\shortmid"{marking}, from=1-2, to=1-1]
        \arrow["{\Phi'}", from=1-2, to=2-2]
        \arrow["\Xi"{inner sep=.8ex}, "\shortmid"{marking}, from=2-2, to=2-1]
    \end{tikzcd}\]
    This cell is cartesian in $\X^\A$, the universal property following from the universal properties of each cartesian cell (1) in $\X$, using the description of the exponential in \cref{representable-vdcs-are-exponentiable}.
\end{proof}

There is little hope for $\X^\A$ to be a \ve{} since, as discussed in \cref{representability-of-exponential}, the exponential seldom admits loose identities. However, $\Lax(\A, \X) = \Mod(\X^\A)$ is a \ve{} as soon as $\X^\A$ admits restrictions.

\begin{proposition}[{\cite[Proposition~7.4]{cruttwell2010unified}}]
    \label{restrictions-in-Mod}
    Let $\X$ be a \vdc. The forgetful functor $\varepsilon_\X \colon \Mod(\X)' \to \X$ creates restrictions. Consequently, if $\X$ admits restrictions, then so does $\Mod(\X)$.
\end{proposition}

\begin{proof}
    The claim is stated in slightly less generality \ibid, but inspection of the proof reveals it to verify the statement.
\end{proof}

Note that normal functors between \ve s automatically preserve restrictions~\cite[Theorem~7.24]{cruttwell2010unified}. This justifies the following definition.

\begin{definition}
    Denote by $\VEquip$ the full and locally full sub-2-category of $\VDbln$ (\cref{normal-vdc}) spanned by the \ve s (\cref{virtual-equipment}).
\end{definition}

\begin{corollary}
    Let $\A$ and $\X$ be \ve s. If $\A$ is representable, then $\LaxN(\A, \X)$ exhibits the exponential $\X^\A$ in the 2-category $\VEquip$.
\end{corollary}

\begin{proof}
    First, note that the forgetful 2-functor $\VEquip \ffto \VDbln$ creates products, restrictions in products being formed componentwise.
    Since $\VEquip$ is a full and locally full sub-2-category of $\VDbl$, the claim will follow as soon as we show that $\LaxN(\A, \X)$ is a \ve{} when $\X$ is. From \cref{restrictions-in-exponential,restrictions-in-Mod}, we know that $\Lax(\A, \X)$ is a \ve. Since $\LaxN(\A, \X) \to \Lax(\A', \X')$ is full on loose morphisms and cells, $\LaxN(\A, \X)$ is consequently also a \ve.
\end{proof}

\section{Presheaves for \dcs}
\label{Yoneda-theory}

In this section, we examine some implications of \cref{Lax-is-Mod} for the \vdc{} $\Lax(\A\opt, \Span)$ of presheaves on a \dc{} $\A$, which was studied extensively by \textcite{pare2011yoneda}. We begin by spelling out the definition of the exponential $\Span^{\A\opt}$ using \cref{representable-vdcs-are-exponentiable}, which admits some simplification over the general case.

\begin{definition}
    For a \vdc{} $\X$, denote by $\X_n \iso \X_1 \times_{\X_0} \cdots \times_{\X_0} \X_1$ the category whose objects are chains of loose morphisms of length $n$, and whose morphisms comprise $n$ cells in $\X$ with compatible source and target.
    \[\begin{tikzcd}
        {X_0} & {X_1} & \cdots & {X_n} \\
        {X_0'} & {X_1'} & \cdots & {X_n'}
        \arrow[""{name=0, anchor=center, inner sep=0}, "{x_0'}"', from=1-1, to=2-1]
        \arrow["{p_1}"'{inner sep=.8ex}, "\shortmid"{marking}, from=1-2, to=1-1]
        \arrow[""{name=1, anchor=center, inner sep=0}, "{x_1'}"{description}, from=1-2, to=2-2]
        \arrow["{p_2}"'{inner sep=.8ex}, "\shortmid"{marking}, from=1-3, to=1-2]
        \arrow[""{name=2, anchor=center, inner sep=0}, "\cdots"{description}, draw=none, from=1-3, to=2-3]
        \arrow["{p_n}"'{inner sep=.8ex}, "\shortmid"{marking}, from=1-4, to=1-3]
        \arrow[""{name=3, anchor=center, inner sep=0}, "{x_n'}", from=1-4, to=2-4]
        \arrow["{p_1'}"{inner sep=.8ex}, "\shortmid"{marking}, from=2-2, to=2-1]
        \arrow["{p_2'}"{inner sep=.8ex}, "\shortmid"{marking}, from=2-3, to=2-2]
        \arrow["{p_n'}"{inner sep=.8ex}, "\shortmid"{marking}, from=2-4, to=2-3]
        \arrow["{\chi_1}"{description}, draw=none, from=0, to=1]
        \arrow["{\chi_2}"{description}, draw=none, from=1, to=2]
        \arrow["{\chi_n}"{description}, draw=none, from=2, to=3]
    \end{tikzcd}\]
    Supposing $\X$ is representable, there is a functor ${\odot}_n \colon \X_n \to \X$ sending each chain of loose morphisms $p_1, \ldots, p_n$ to the loose composite $p_1 \odot \cdots \odot p_n$, and each chain of cells $\chi_1, \ldots, \chi_n$ to the loose composite $\chi_1 \odot \cdots \chi_n$ (\cref{Dbl_l-in-VDbl}).
\end{definition}

\begin{example}
    \label{Span-A-op}
    Observe that the category $\Span_1$ of spans and span morphisms is isomorphic to $\Set^{\{ \cdot \from \cdot \to \cdot \}}$, the category of $\Set$-valued functors from the free-standing span, and that the inclusion $1 \to \{ \cdot \from \cdot \to \cdot \}$ that picks out the middle object of the span induces a functor $m \colon \Set^{\{ \cdot \from \cdot \to \cdot \}} \to \Set$.

    For a \pdc{} $\A$, the underlying category of $\Span^{\A\opt}$ is the presheaf category $\Set^{(\A_0)\op}$. By the description of $\Span_1$ above, a loose morphism from $F'$ to $F$ is given by a functor and a pair of natural transformations as follows.
    \[\begin{tikzcd}
        {\A_0\op} & {\A_1\op} & {\A_0\op} \\
        \Set & \Set & \Set
        \arrow[""{name=0, anchor=center, inner sep=0}, "F"', from=1-1, to=2-1]
        \arrow["{t\op}"', from=1-2, to=1-1]
        \arrow["{s\op}", from=1-2, to=1-3]
        \arrow[""{name=1, anchor=center, inner sep=0}, "\Xi"{description}, from=1-2, to=2-2]
        \arrow[""{name=2, anchor=center, inner sep=0}, "{F'}", from=1-3, to=2-3]
        \arrow[equals, from=2-2, to=2-1]
        \arrow[equals, from=2-3, to=2-2]
        \arrow["{\Xi_t}"', between={0.3}{0.7}, Rightarrow, from=1, to=0]
        \arrow["{\Xi_s}", between={0.3}{0.7}, Rightarrow, from=1, to=2]
    \end{tikzcd}\]
    A cell with the following frame (left) is given by a natural transformation (right).
    \[
    \begin{tikzcd}
        {F_0} & \cdots & {F_n} \\
        F && {F'}
        \arrow[""{name=0, anchor=center, inner sep=0}, "\Phi"', from=1-1, to=2-1]
        \arrow["{\Xi_1}"', "\shortmid"{marking}, from=1-2, to=1-1]
        \arrow["{\Xi_n}"', "\shortmid"{marking}, from=1-3, to=1-2]
        \arrow[""{name=1, anchor=center, inner sep=0}, "{\Phi'}", from=1-3, to=2-3]
        \arrow["\Xi", "\shortmid"{marking}, from=2-3, to=2-1]
        \arrow["\xi"{description}, draw=none, from=1, to=0]
    \end{tikzcd}
    \hspace{4em}
    \begin{tikzcd}
        {\A_n\op} && {\A_1\op} \\
        {\Span_n} & {\Span_1} & \Set
        \arrow["{\odot_n\op}", from=1-1, to=1-3]
        \arrow[""{name=0, anchor=center, inner sep=0}, "{\tp{\Xi_1, \ldots, \Xi_n}}"', from=1-1, to=2-1]
        \arrow[""{name=1, anchor=center, inner sep=0}, "\Xi", from=1-3, to=2-3]
        \arrow["{\odot_n}"', from=2-1, to=2-2]
        \arrow["m"', from=2-2, to=2-3]
        \arrow["\xi", between={0.3}{0.7}, Rightarrow, from=0, to=1]
    \end{tikzcd}
    \]
    such that the following equations hold.
    \[
    \begin{tikzcd}[column sep=large]
        {\A_n\op} & {\A_1\op} & {\A_0\op} \\
        {\Span_n} & \Set & \Set
        \arrow["{\odot_n\op}", from=1-1, to=1-2]
        \arrow[""{name=0, anchor=center, inner sep=0}, "{\tp{\Xi_1, \ldots, \Xi_n}}"{description}, from=1-1, to=2-1]
        \arrow["{s\op}", from=1-2, to=1-3]
        \arrow[""{name=1, anchor=center, inner sep=0}, "\Xi"{description}, from=1-2, to=2-2]
        \arrow[""{name=2, anchor=center, inner sep=0}, "{F'}", from=1-3, to=2-3]
        \arrow["{m \c \odot_n}"', from=2-1, to=2-2]
        \arrow[equals, from=2-2, to=2-3]
        \arrow["\xi", between={0.4}{0.7}, Rightarrow, from=0, to=1]
        \arrow["{\Xi_s}", between={0.3}{0.7}, Rightarrow, from=1, to=2]
    \end{tikzcd}
    \quad = \quad
    \begin{tikzcd}[column sep=large]
        {\A_n\op} & {\A_1\op} & {\A_0\op} \\
        {\Span_n} & \Set & \Set
        \arrow["{\odot_n\op}", from=1-1, to=1-2]
        \arrow[""{name=0, anchor=center, inner sep=0}, "{\tp{\Xi_1, \ldots, \Xi_n}}"{description}, from=1-1, to=2-1]
        \arrow["{s\op}", from=1-2, to=1-3]
        \arrow[""{name=1, anchor=center, inner sep=0}, "{\Xi_n}"{description}, from=1-2, to=2-2]
        \arrow[""{name=1p, anchor=center, inner sep=0}, phantom, from=1-2, to=2-2, start anchor=center, end anchor=center]
        \arrow[""{name=2, anchor=center, inner sep=0}, "{F'}", curve={height=-12pt}, from=1-3, to=2-3]
        \arrow[""{name=3, anchor=center, inner sep=0}, "{F_n}"{description}, curve={height=12pt}, from=1-3, to=2-3]
        \arrow[""{name=3p, anchor=center, inner sep=0}, phantom, from=1-3, to=2-3, start anchor=center, end anchor=center, curve={height=12pt}]
        \arrow["{m \c \odot_n}"', from=2-1, to=2-2]
        \arrow[equals, from=2-2, to=2-3]
        \arrow["{\pi_n}", between={0.4}{0.7}, Rightarrow, from=0, to=1]
        \arrow["{(\Xi_n)_s}", between={0.2}{0.8}, Rightarrow, from=1p, to=3p]
        \arrow["{\Phi'}", between={0.2}{0.8}, Rightarrow, from=3, to=2]
    \end{tikzcd}
    \]
    \[
    \begin{tikzcd}[column sep=large]
        {\A_n\op} & {\A_1\op} & {\A_0\op} \\
        {\Span_n} & \Set & \Set
        \arrow["{\odot_n\op}", from=1-1, to=1-2]
        \arrow[""{name=0, anchor=center, inner sep=0}, "{\tp{\Xi_1, \ldots, \Xi_n}}"{description}, from=1-1, to=2-1]
        \arrow["{t\op}", from=1-2, to=1-3]
        \arrow[""{name=1, anchor=center, inner sep=0}, "\Xi"{description}, from=1-2, to=2-2]
        \arrow[""{name=2, anchor=center, inner sep=0}, "F", from=1-3, to=2-3]
        \arrow["{m \c \odot_n}"', from=2-1, to=2-2]
        \arrow[equals, from=2-2, to=2-3]
        \arrow["\xi", between={0.4}{0.7}, Rightarrow, from=0, to=1]
        \arrow["{\Xi_t}", between={0.3}{0.7}, Rightarrow, from=1, to=2]
    \end{tikzcd}
    \quad = \quad
    \begin{tikzcd}[column sep=large]
        {\A_n\op} & {\A_1\op} & {\A_0\op} \\
        {\Span_n} & \Set & \Set
        \arrow["{\odot_n\op}", from=1-1, to=1-2]
        \arrow[""{name=0, anchor=center, inner sep=0}, "{\tp{\Xi_1, \ldots, \Xi_n}}"{description}, from=1-1, to=2-1]
        \arrow["{t\op}", from=1-2, to=1-3]
        \arrow[""{name=1, anchor=center, inner sep=0}, "{\Xi_1}"{description}, from=1-2, to=2-2]
        \arrow[""{name=1p, anchor=center, inner sep=0}, phantom, from=1-2, to=2-2, start anchor=center, end anchor=center]
        \arrow[""{name=2, anchor=center, inner sep=0}, "F", curve={height=-12pt}, from=1-3, to=2-3]
        \arrow[""{name=3, anchor=center, inner sep=0}, "{F_0}"{description}, curve={height=12pt}, from=1-3, to=2-3]
        \arrow[""{name=3p, anchor=center, inner sep=0}, phantom, from=1-3, to=2-3, start anchor=center, end anchor=center, curve={height=12pt}]
        \arrow["{m \c \odot_n}"', from=2-1, to=2-2]
        \arrow[equals, from=2-2, to=2-3]
        \arrow["{\pi_1}", between={0.4}{0.7}, Rightarrow, from=0, to=1]
        \arrow["{(\Xi_1)_t}", between={0.2}{0.8}, Rightarrow, from=1p, to=3p]
        \arrow["\Phi", between={0.2}{0.8}, Rightarrow, from=3, to=2]
    \end{tikzcd}
    \]
    Above, we have denoted by $\tp{\Xi_1, \ldots, \Xi_n}$ the functor sending a chain of loose morphisms $p_1, \ldots, p_n$ in $\A$ to the following chain of spans.
    \[\begin{tikzcd}
        & {\Xi_1(p_1)} && {\Xi_2(p_2)} && \cdots \\
        {F_0(X_0)} && {F_1(X_1)} && {F_2(X_2)} && {F_n(X_n)}
        \arrow["{\Xi_t(p_1)}"', from=1-2, to=2-1]
        \arrow["{\Xi_s(p_1)}", from=1-2, to=2-3]
        \arrow["{\Xi_t(p_2)}"', from=1-4, to=2-3]
        \arrow["{\Xi_s(p_2)}", from=1-4, to=2-5]
        \arrow["{\Xi_t(p_3)}"', from=1-6, to=2-5]
        \arrow["{\Xi_s(p_n)}", from=1-6, to=2-7]
    \end{tikzcd}\]
    The identity cell on $\Xi$ is given by the identity natural transformation on $\Xi \colon {\A_1}\op \to \Set$.
    \[\begin{tikzcd}
        {\A_1\op} && {\A_1\op} \\
        {\Span_1} & {\Span_1} & \Set
        \arrow[equals, from=1-1, to=1-3]
        \arrow[""{name=0, anchor=center, inner sep=0}, "\Xi"', from=1-1, to=2-1]
        \arrow[""{name=1, anchor=center, inner sep=0}, "\Xi", from=1-3, to=2-3]
        \arrow[equals, from=2-1, to=2-2]
        \arrow["m"', from=2-2, to=2-3]
        \arrow["{=}"{description}, draw=none, from=0, to=1]
    \end{tikzcd}\]
    Composition is given as follows.
    \[\begin{tikzcd}[column sep=large]
        {\A_{m_1 + \cdots + m_n}\op} && {\A_n\op} && {\A_1\op} \\
        {\Span_{m_1 + \cdots + m_n}} && {\Span_n} & {\Span_1} & \Set
        \arrow["{\tp{\odot_{m_1}, \ldots, \odot_{m_n}}\op}", from=1-1, to=1-3]
        \arrow[""{name=0, anchor=center, inner sep=0}, "{\tp{\Xi_1^1, \ldots, \Xi_{m_n}^n}}"', from=1-1, to=2-1]
        \arrow["{\odot_n\op}", from=1-3, to=1-5]
        \arrow[""{name=1, anchor=center, inner sep=0}, "{\tp{\Xi_1, \ldots, \Xi_n}}"{description}, from=1-3, to=2-3]
        \arrow[""{name=2, anchor=center, inner sep=0}, "\Xi", from=1-5, to=2-5]
        \arrow["{\tp{\odot_{m_1}, \ldots, \odot_{m_n}}}"', from=2-1, to=2-3]
        \arrow["{\odot_n}"', from=2-3, to=2-4]
        \arrow["m"', from=2-4, to=2-5]
        \arrow["{\tp{\xi_1, \ldots, \xi_n}}", between={0.4}{0.6}, Rightarrow, from=0, to=1]
        \arrow["\xi", between={0.4}{0.6}, Rightarrow, from=1, to=2]
    \end{tikzcd}\qedshift\]
\end{example}

\begin{example}
    Let $\E$ be a small category, viewed as a loosely discrete strict \dc{} $\dc E \defeq \F(\o\E)'$, so that $\E = \dc E_0 = \dc E_1 = \dc E_n$ ($n \ge 0$). By \cref{power-and-copower}, $\Span^{\dc E}$ is the power $\E \pow \Span$. Consequently, either by direct computation using \cref{Span-A-op}, or by observing that the 2-functor $\Span\ph \colon \Cat_{\tx{pb}} \to \Dbl_l$ preserves powers (\cf~\cite[Proposition~3.27]{dawson2010span}) -- where $\Cat_{\tx{pb}}$ is the 2-category of categories with pullbacks, pullback-preserving functors, and natural transformations -- we have that $\Span^{\dc E}$ is representable and is equivalent to the \pdc{} $\Span(\Set^\E)$. By \cref{Mod-Span,Lax-is-Mod}, we have
    \[\Dist(\Set^\E) \iso \Mod(\Span(\Set^\E)) \iso \Mod(\Span^{\dc E}) \iso \Lax(\dc E, \Span)\]
    whose underlying isomorphism of 2-categories, by \cref{sequential-adjoint}, recovers the symmetry between categories internal to presheaf categories and $\Cat$-valued presheaves.
    \[\Cat(\Set^\E) \iso \VDbl(\dc E, \Span) \iso \Cat^\E \qedhere\]
\end{example}

Just as, to every category $\b A$, there is an associated hom-functor $\b A({-}, {-}) \colon \b A\op \times \b A \to \Set$ sending every pair of objects to the set of morphisms between them, to every \pdc{} $\A$, there is an associated lax hom-functor $\A({-}, {-}) \colon \A\opt \times \A \to \Span$ sending every pair of objects to the set of tight morphisms between them, and every pair of loose morphisms to the set of cells between them.

\begin{proposition}[{\cite[\S2.1]{pare2011yoneda}}]
    \label{hom-functor}
    For every \pdc{} $\A$, there is a lax hom-functor $\A({-}, {-}) \colon \A\opt \times \A \to \Span$.
\end{proposition}

\begin{proof}
    The action on objects and tight morphisms is given by $\A_0({-}, {-})$ and the action on loose morphisms and cells is given by $\A_1({-}, {-})$ together with the source and target projections, as in the following diagram. Below, $s_{{-}, {-}}$ and $t_{{-}, {-}}$ are postcomposition by $s$ and $t$ respectively.
    \[\begin{tikzcd}[column sep=large]
        {\A_0\op \times \A_0} & {\A_1\op \times \A_1} & {\A_0\op \times \A_0} \\
        \Set & \Set & \Set
        \arrow[""{name=0, anchor=center, inner sep=0}, "{\A_0({-}, {-})}"', from=1-1, to=2-1]
        \arrow["{s\op \times s}"', from=1-2, to=1-1]
        \arrow["{t\op \times t}", from=1-2, to=1-3]
        \arrow[""{name=1, anchor=center, inner sep=0}, "{\A_1({-}, {-})}"{description}, from=1-2, to=2-2]
        \arrow[""{name=2, anchor=center, inner sep=0}, "{\A_0({-}, {-})}", from=1-3, to=2-3]
        \arrow[equals, from=2-1, to=2-2]
        \arrow[equals, from=2-2, to=2-3]
        \arrow["{s_{{-}, {-}}}"{description}, shorten <=19pt, shorten >=13pt, Rightarrow, from=1, to=0]
        \arrow["{t_{{-}, {-}}}"{description}, shorten <=19pt, shorten >=13pt, Rightarrow, from=1, to=2]
    \end{tikzcd}\]
    For each pair of objects $A$ and $A'$, the unitor is given by the following functor.
    \[\begin{tikzcd}[column sep=large]
        & {\A_0(A, A')} \\
        {\A_0(A, A')} & {\A_1(\I_A, \I_{A'})} & {\A_0(A, A')}
        \arrow[equals, from=1-2, to=2-1]
        \arrow["{\I_{\ph}}"{description}, from=1-2, to=2-2]
        \arrow[equals, from=1-2, to=2-3]
        \arrow["{s_{\I_A, \I_{A'}}}", from=2-2, to=2-1]
        \arrow["{t_{\I_A, \I_{A'}}}"', from=2-2, to=2-3]
    \end{tikzcd}\]
    For each pair of composable pairs of loose morphisms $A \xlfrom p B \xlfrom q C$ and $A' \xlfrom {p'} B' \xlfrom {q'} C'$ in $\A$, the compositor is given by the following functor.
    \[\begin{tikzcd}
        {\A_0(A, A')} & {\A_1(p, p')} & {\A_1(p, p') \times_{\A_0(B, B')} \A_1(q, q')} & {\A_1(q, q')} & {\A_0(C, C')} \\
        {\A_0(A, A')} && {\A_1(p \odot q, p' \odot q')} && {\A_0(C, C')}
        \arrow[equals, from=1-1, to=2-1]
        \arrow["{s_{p, p'}}"', from=1-2, to=1-1]
        \arrow["{\pi_1}"', from=1-3, to=1-2]
        \arrow["{\pi_2}", from=1-3, to=1-4]
        \arrow["{\ph \odot \ph}"{description}, from=1-3, to=2-3]
        \arrow["{t_{p, p'}}", from=1-4, to=1-5]
        \arrow[equals, from=1-5, to=2-5]
        \arrow["{s_{p \odot q, p' \odot q'}}", from=2-3, to=2-1]
        \arrow["{t_{p \odot q, p' \odot q'}}"', from=2-3, to=2-5]
    \end{tikzcd}\]
    That this forms a lax functor follows from coherence for the \pdc{} $\A$.
\end{proof}

By transposing the hom-functor, we obtain a functor of \vdcs{} as follows.
\begin{equation}
    \label{exponential-embedding}
    \A({-}_2, {-}_1) \colon \A \to \Span^{\A\opt}
\end{equation}
Since the underlying category of $\Span^{\A\opt}$ is the category of presheaves on the category $\u\A$ underlying $\A$, we automatically obtain a Yoneda lemma with respect to the objects and tight morphisms of $\A$, given simply by the ordinary Yoneda lemma with respect to $\u\A$. However, we also have a Yoneda lemma with respect to the loose morphisms and cells of $\A$; the idea is contained in \cite[Theorem~4.4]{pare2011yoneda}, but the following proposition is simpler than the theorem \ibid, as it concerns cells in $\Span^{\A\opt}$ rather than cells in $\Lax(\A\opt, \Span)$.

\begin{proposition}[Yoneda lemma]
    \label{Yoneda-lemma}
    Cells in $\Span^{\A\opt}$ of the following form
    \[\begin{tikzcd}[column sep=large]
        {\A({-}, A_0)} & \cdots & {\A({-}, A_n)} \\
        F && {F'}
        \arrow[""{name=0, anchor=center, inner sep=0}, "{a_0}"', from=1-1, to=2-1]
        \arrow["{\A({-}, p_1)}"'{inner sep=.8ex}, "\shortmid"{marking}, from=1-2, to=1-1]
        \arrow["{\A({-}, p_n)}"'{inner sep=.8ex}, "\shortmid"{marking}, from=1-3, to=1-2]
        \arrow[""{name=1, anchor=center, inner sep=0}, "{a_n}", from=1-3, to=2-3]
        \arrow["\Xi"{inner sep=.8ex}, "\shortmid"{marking}, from=2-3, to=2-1]
        \arrow["\alpha"{description}, draw=none, from=0, to=1]
    \end{tikzcd}\]
    are in natural bijection with elements of $\Xi(p_1 \odot \cdots \odot p_n)$.
\end{proposition}

\begin{proof}
    By the Yoneda lemma for the category $\u\A$, the tight morphisms framing $\alpha$ correspond to elements $a_0 \in F(A_0)$ and $a_n \in F'(A_n)$. By definition (\cref{Span-A-op}), the cell $\alpha$ corresponds to a natural transformation as follows.
    \[\A({-}, p_1) \times_{\A({-}, A_1)} \A({-}, p_2) \times_{\A({-}, A_2)} \cdots \times_{\A({-}, A_{n - 1})} \A({-}, p_n) \quad \tto \quad \Xi\]
    Since $\A_n \iso \A_1 \times_{\A_0} \cdots \times_{\A_0} \A_1$, the domain of this natural transformation is isomorphic to the representable presheaf $\A_n({-}, (p_1, \ldots, p_n))$, and so the statement follows from the Yoneda lemma with respect to presheaves on $\A_n$.
\end{proof}

\begin{corollary}
    Every chain of representables $\A({-}, p_1), \ldots, \A({-}, p_n)$ in $\Span^{\A\opt}$ admits a weak composite $\A({-}, p_1 \odot \cdots \odot p_n)$.
\end{corollary}

\begin{proof}
    Taking $\Xi = \A({-}, p_1 \odot \cdots \odot p_n)$ and $a_0$ and $a_n$ to be identities in \cref{Yoneda-lemma}, we get a globular cell of the appropriate shape, which is weakly opcartesian by another application of \cref{Yoneda-lemma}.
\end{proof}

This gives us all the pieces needed to recover the Yoneda embedding of \textcite[Theorem~4.8]{pare2011yoneda}: as desired, it is obtained by transposing the hom-functor, and then applying the loose monads and modules construction.

\begin{proposition}
    \label{presheaf-embedding}
    For each \pdc{} $\A$, there is a \ff{} functor $\A({-}_2, {-}_1) \colon \A \to \Span^{\A\opt}$ of \vdcs, hence also a \ff{} normal functor $Y_\A \colon \A \to \Lax(\A\opt, \Span)$.
\end{proposition}

\begin{proof}
    As described above \eqref{exponential-embedding}, the functor $\A({-}_2, {-}_1)$ is obtained by transposing the lax-hom functor of \cref{hom-functor}. Full faithfulness on tight morphisms follows from \ffness{} of the Yoneda embedding $\u\A \ffto \Set^{\u\A\op}$, while \ffness{} on cells follows from \cref{Yoneda-lemma}. Therefore, by \cref{Mod-is-right-adjoint,Lax-is-Mod}, we obtain a normal functor of \vdcs{} $\A \to \Mod(\Span^{\A\opt}) \iso \Lax(\A\op, \Span)$, which is \ff{} by \cref{ff-transpose}.
\end{proof}

Explicitly, the embedding $Y_\A \colon \A \to \Lax(\A\opt, \Span)$ sends an object $A \in \A$ to the image under $\A({-}_2, {-}_1) \colon \A \to \Span^{\A\opt}$ of the identity loose monad $A(1, 1)$ on $A$.
The Yoneda lemmas for $\Lax(\A\opt, \Span)$~\cite[Theorems~2.3, 3.18 \& 4.4]{pare2011yoneda} thereby follow readily from the Yoneda lemmas for $\Span^{\A\opt}$. We leave this as an instructive exercise for the reader. (For a hint, note that the Yoneda lemma for lax functors invokes both the ordinary Yoneda lemma and \cref{Yoneda-lemma}; whereas the Yoneda lemma for multimodulations makes use of the property that cells in $\Lax(\A\opt, \Span)$ are precisely cells in $\Span^{\A\opt}$ satisfying some additional properties.) We suggest this indicates that $\A({-}_2, {-}_1) \colon \A \to \Span^{\A\opt}$ may be the more fundamental of the two `Yoneda embeddings' for \dcs.

\begin{remark}
    In \cite{frohlich2024yoneda}, \citeauthor{frohlich2024yoneda} study a variant of the Yoneda embedding with respect to \emph{normal} lax functors $\A\opt \to \Dist$. However, by \cref{Mod-is-internal-right-adjoint}, we have an isomorphism of normal \vdcs{},
    \[\Lax(\A\opt, \Span) \iso \LaxN(\A\opt, \Mod(\Span)) \iso \LaxN(\A\opt, \Dist)\]
    under which the theory of \citeauthor{frohlich2024yoneda} is entirely equivalent to \citeauthor{pare2011yoneda}'s~\cite{pare2011yoneda}. For instance their normal lax hom-functor $\A\opt \times \A \to \Dist$~\cite[Construction~3.4.1]{frohlich2024yoneda} is obtained from \cref{hom-functor} by transposition, and their Yoneda lemma~\cite[Theorem~4.2]{frohlich2024yoneda} follows from \citeauthor{pare2011yoneda}'s likewise.
\end{remark}

We shall need one other construct associated to the Yoneda embedding for what follows.

\begin{lemma}
    \label{nerve}
    Every functor $F \colon \A \to \B$ of \vdcs{} with $\A$ exponentiable induces a functor $N_F \colon \B \to \Span^{\A\opt}$, the \emph{nerve} of $F$.
\end{lemma}

\begin{proof}
    Functors $\B \to \Span^{\A\opt}$ are in 2-natural bijection with functors $\A\opt \times \B \to \Span$, the latter of which may be induced by $F$ via the following composite.
    \[\begin{tikzcd}[column sep=large]
        {\A\opt \times \B} & {\B\opt \times \B} & \Span
        \arrow["{F\opt \times \B}", from=1-1, to=1-2]
        \arrow["{\B({-}, {-})}", from=1-2, to=1-3]
    \end{tikzcd}\]
    (Observe that, when $\B$ is also exponentiable, the nerve is isomorphic to the following composite.)
    \[\begin{tikzcd}[column sep=huge]
        \B & {\Span^{\B\opt}} & {\Span^{\A\opt}}
        \arrow["{\B({-}_2, {-}_1)}", from=1-1, to=1-2]
        \arrow["{\Span^{F\opt}}", from=1-2, to=1-3]
    \end{tikzcd}\qedshift\]
\end{proof}

\section{Exponentiability revisited}
\label{exponentiability-revisited}

In \cref{exponentiability}, we presented several sufficient conditions for a \vdc{} to be exponentiable: most importantly, in \cref{representable-vdcs-are-exponentiable}, we showed that every representable \vdc{} is exponentiable. In this section, we give a full characterisation of exponentiability for \vdcs{} by reducing the exponentiability of an arbitrary \vdc{} $\A$ to the existence of a specific exponential, namely $\Span^\A$. This is similar to characterisations of exponentiability for other multicategory-like structures: for instance, a locale $A$ is exponentiable if and only if the exponential $S^A$ exists, where $S$ is the Sierpi\'nski space~\cite{hyland1979function}; a topos $\b A$ is exponentiable if and only if the exponential $\b S^{\b A}$ exists, where $\b S$ is the object classifier~\cite{johnstone1982continuous}; and a multicategory $\b M$ is exponentiable if and only if the exponential $\Set^{\b M}$ exists (as we shall prove in \cref{exponentiable-multicategory}). Our proof strategy in this section follows the characterisation of exponentiable locales in \textcite[\S3]{townsend2006categorical} and, in particular, invokes the exponentiability of representable \vdcs{} from \cref{exponentiability}.

First, we observe an elementary commutation property between exponentials and equalisers.

\begin{lemma}
    \label{exponentials-and-equalisers}
    Let $A, X, Y$ be objects in a 2-category $\C$ for which the exponentials $X^A$ and $Y^A$ exist. Suppose that the following diagram exhibits an equaliser.
    \[\begin{tikzcd}
        E & X & Y
        \arrow["e", from=1-1, to=1-2]
        \arrow["f", shift left, from=1-2, to=1-3]
        \arrow["g"', shift right, from=1-2, to=1-3]
    \end{tikzcd}\]
    The parallel pair $f^A, g^A \colon X^A \rightrightarrows Y^A$ admits an equaliser if and only if the exponential $E^A$ exists, in which case $e^A \colon E^A \to X^A$ exhibits the equaliser.
\end{lemma}

\begin{proof}
    Applying the Yoneda embedding to the equaliser diagram, and taking the value at $\ph \times A$, we obtain the following equaliser diagram in $\Cat^{\C\op}$, in which the two rightmost 2-presheaves are representable by assumption.
    \[\begin{tikzcd}[column sep=huge]
        {\C({-} \times A, E)} & {\C({-} \times A, X)} & {\C({-} \times A, Y)}
        \arrow["{\C({-} \times A, e)}", from=1-1, to=1-2]
        \arrow["{\C({-} \times A, f)}", shift left, from=1-2, to=1-3]
        \arrow["{\C({-} \times A, g)}"', shift right, from=1-2, to=1-3]
    \end{tikzcd}\]
    Existence of the requisite exponential, which is equivalent to the representability of $\C({-} \times A, E)$, is thus equivalent to the existence of the requisite equaliser.
\end{proof}

The key lemma in this section is the following, which states that, as soon as objects in a 2-category can be embedded into exponentials of a fixed object $S$, the exponentiability of a given object $A$ can be reduced to the existence of the exponential $S^A$.

\begin{lemma}
    \label{exponential-via-subobject}
    Let $\C$ be a 2-category with binary products and equalisers. Fix an object $S \in \C$. Suppose that, for every object $X \in \C$, there is an object $E_X \in \C$ such that (1) the exponential $S^{E_X}$ exists; (2) $X$ is a regular subobject of $S^{E_X}$. Then an object $A \in \C$ is exponentiable if and only if the exponential $S^A$ exists.
\end{lemma}

\begin{proof}
    First note that, for each object $X \in \C$, we have, by assumption, an equaliser diagram $X \monoto S^{E_X} \rightrightarrows Y$ for some object $Y \in \C$. Again by assunption, we have a regular monomorphism $Y \monoto S^{E_Y}$. Since equaliser diagrams are closed under postcomposition by monomorphisms, $X \monoto S^{E_X} \rightrightarrows S^{E_Y}$ is also an equaliser diagram.

    Now suppose that the exponential $S^A$ exists. Since exponentiable objects are closed under products, \cref{exponentials-and-equalisers} implies that the equaliser of $S^{E_X \times A} \rightrightarrows S^{E_Y \times A}$ is the exponential $X^A$. The converse is trivial.
\end{proof}

It remains to verify the assumptions of \cref{exponential-via-subobject} in the case of the 2-category $\VDbl$.

\begin{lemma}
    \label{subobject-classifier}
    $\VDbl$ admits a classifier $t \colon \bbn 1 \to \Omega$ for functors that are injective on objects, injective on loose morphisms, and \ff.
\end{lemma}

\begin{proof}
    Define the \vdc{} $\Omega$ as follows.
    \begin{itemize}
        \item The underlying category is the codiscrete category on a pair of objects, denoted $\bot$ and $\top$.
        \item There are unique loose morphisms $\bot \lto \bot$, $\bot \lto \top$, and $\top \lto \bot$. There are two loose endomorphisms on $\top$, denoted $f \colon \top \lto \top$ and $t \colon \top \lto \top$.
        \item There is a unique cell for every frame.
    \end{itemize}
    Define the functor $t \colon \bbn 1 \to \Omega$ to pick the trivial loose monad $t \colon \top \lto \top$. A functor $\chi \colon \B \to \Omega$ is specified by assigning, to each object, either $\top$ or $\bot$; and, to each loose morphism whose domain and codomain have been assigned $\top$, either $t$ or $f$. This is precisely the data of a functor into $\B$, obtained by pulling $\chi$ back along $t$, that is injective on objects, injective on loose morphisms, and \ff. Since $\Omega$ is codiscrete on tight morphisms and cells, every transformation between such functors is given by pulling back a transformation between functors $\B \to \Omega$.
\end{proof}

\begin{theorem}[{\cite[Theorems~1.21 \& 2.12]{dawson2006paths}}]
    \label{path-construction}
    The inclusion of the 2-category of strict \dcs{} and strict functors into $\VDbl$ admits a comonadic left adjoint, inducing a lax-idempotent 2-monad $\P$ on $\VDbl$.
\end{theorem}

In particular, for each \vdc{} $\X$, the component of the unit $\eta_\X \colon \X \to \P(\X)$ is the identity on the underlying categories, injective on loose morphisms, and \ff{} on cells, so that every \vdc{} embeds \ff ly into a strict \dc.

\begin{theorem}
    \label{exponentiable-iff-Span-exponential-exists}
    A \vdc{} $\A$ is exponentiable if and only if the exponential $\Span^\A$ exists.
\end{theorem}

\begin{proof}
    We appeal to \cref{exponential-via-subobject}, with respect to the \vdc{} $\Span$ in $\VDbl$, recalling that $\VDbl$ is complete by \cref{VDbl-is-bicomplete}. Let $\X$ be a \vdc. To satisfy the remaining assumption, observe that, since $\P(\X)$ is a strict \dc{} (in particular, representable), it is exponentiable by \cref{representable-vdcs-are-exponentiable}. By \cref{presheaf-embedding,path-construction}, for every \vdc{}, we have a \ff{} functor $\X \ffto \P(\X) \ffto \Span^{\P(\X)\opt}$ that is injective on objects and loose morphisms, hence a regular monomorphism by \cref{subobject-classifier}.
\end{proof}

\begin{remark}
    In a preliminary version of this paper, \cref{exponentiable-iff-Span-exponential-exists} was presented as a conjecture (see \cref{exponentiable-vdcs-conjecture}). Since then, Kevin Carlson and Ea Thompson have announced an independent proof of the statement~\cite{thompson2025exponentiable}, which should be viewed as contemporaneous to our own.
\end{remark}

In passing, we note that this proof strategy also gives a characterisation of exponentiability for multicategories. As we have mentioned in \cref{relation-to-multicategories}, and shall discuss in more detailed shortly, \textcite[Proposition~2.12]{pisani2014sequential} established that the exponentiable multicategories are precisely the promonoidal categories of \textcite{day1970closed}. Given that promonoidal categories may be characterised as the cocontinuously monoidal presheaf categories~\cite{day1970closed}, the following characterisation is not surprising, but does appears to be new.

\begin{proposition}
    \label{exponentiable-multicategory}
    A multicategory $\M$ is exponentiable if and only if the exponential $\Set^\M$ exists.
\end{proposition}

\begin{proof}
    The proof is essentially identical to that of \cref{exponentiable-iff-Span-exponential-exists}, instead using the construction of the free strict monoidal category on a multicategory~\cite[\S7]{hermida2000representable}, and classifying \ff{} functors that are injective on objects via the codiscrete multicategory on two objects.
\end{proof}

\subsection{An exponentiability conjecture}
\label{exponentiability-conjecture}

The characterisation of exponentiable \vdcs{} in \cref{exponentiable-iff-Span-exponential-exists}, while complete, is not particularly conducive to verification in specific instances. It would therefore be useful to have an intrinsic characterisation of exponentiability for \vdcs. While we shall not do so here, we shall draw on the connection between \vdcs{} and multicategories to make a conjecture regarding such a characterisation.

Multicategories may be characterised as the normal colax comonoids in the monoidal bicategory $\b{\cl Dist}$ of categories and distributors (\ie{} the bicategory underlying the \pdc{} $\Dist$ of \cref{Dist})~\cites[\S1]{day2003lax}[\S B.5]{cruttwell2010unified}. Restricting to the \emph{pseudo} comonoids in $\b{\cl Dist}$ results in the notion of \emph{promonoidal category}~\cite{day1970closed}. In the terminology of \textcite{leinster2004higher}, promonoidal categories qua pseudo comonoids may be presented in two ways: in the \emph{unbiased} presentation, to define a promonoidal category $\b P$, one specifies functors $(\b P^n)\op \times \b P \to \Set$ for all $n \in \N$, which define the sets of $n$-ary multimorphisms; in the \emph{biased} presentation, one only specifies such functors for $n \le 2$ -- the multimorphisms of higher arity being defined freely by composites of lower arity multimorphisms. The unbiased perspective is most naturally viewed as being a property of a multicategory, which intuitively corresponds to the ability to decompose multimorphisms appropriately: this property is called \emph{malleability} by \textcite[Definition~2.3.3 \& Proposition~2.3.10]{roman2023monoidal}. The biased perspective, which is that to which the term \emph{promonoidal category} is typically associated in the literature, is most naturally viewed as a structure of interest in its own right. An advantage of the biased presentation over the unbiased presentation is that it is finitary.

The reason we are interested in malleability for multicategories is that it characterises exponentiability: in other words, a multicategory is exponentiable if and only if it arises from a promonoidal category~\cite[Proposition~2.8]{pisani2014sequential}. Consequently, we expect that exponentiability for \vdcs{} may be characterised similarly, in terms of a decomposition property pertaining to multiary cells. The following definition makes this intuition precise.

\begin{definition}
    Let $T$ be the monad on $\Span(\Grph)$ whose algebras are categories, inducing a monad $\Mod(T)$ on $\Dist(\Grph)$~\cite[Examples~4.7 \& 8.13]{cruttwell2010unified}. A \vdc{} is \emph{malleable} if, viewed as a \emph{$\Mod(T)$-monoid} in the sense of \textcite[Definition~4.3]{cruttwell2010unified}, its multiplication cell is cartesian.\footnote{More generally, we may define malleability analogously for any generalised multicategory in the sense of \textcite{cruttwell2010unified}. We expect that malleability is typically a sufficient, and often necessary, condition for exponentiability.}
\end{definition}

In contrast to malleable multicategories, it is less clear that malleable \vdcs{} admit a finitary presentation (and consequently that there exists an appropriate notion of \emph{pro-double category} generalising the notion of \emph{probicategory}~\cite{day1973embedding}). While we do expect the malleability property to be expressible in terms of decomposition into composites involving solely nullary, unary, and binary cells, it seems unlikely that an analogue of the reassociation operation in the biased presentation of a promonoidal category suffices for \vdcs{}, due to additional constraints imposed by tight morphisms in a \vdc{} that are not present for a multicategory.

These considerations lead to the following conjecture (in which $(1 \iff 4)$ has been verified by \cref{exponentiable-iff-Span-exponential-exists}).

\begin{conjecture}
    \label{exponentiable-vdcs-conjecture}
    The following are equivalent for a \vdc{} $\X$.
    \begin{enumerate}
        \item $\X$ is exponentiable.
        \item $\X$ is malleable.
        \item $\X$ satisfies a restricted form of malleability, expressing every cell as a composite of nullary, unary, and binary cells.\footnote{We shall not give a precise condition, but have in mind an analogue of \cite[Definition~2.3.9]{roman2023monoidal}.}
        \item The exponential $\Span^\X$ exists.
        \qedhere
    \end{enumerate}
\end{conjecture}

\section{Local cocompletion and graded categories}
\label{local-cocompletion}

In \cref{relation-to-multicategories}, we recalled that, given a small monoidal category $\M$, the convolution monoidal structure on the category of presheaves $\Set^{\M\op}$ is given by the exponential multicategory structure~\cite{pisani2014sequential}: this characterises it essentially uniquely via a universal property. However, the convolution monoidal structure also admits a second universal property: it is the free monoidal cocompletion of $\M$~\cite{im1986universal}. Given that, for a \dc{} $\A$, the exponential $\Span^{\A\opt}$ plays an analogous role to $\Set^{\M\op}$, we might wonder whether it also admits a second universal property of a similar flavour to the free monoidal cocompletion.

There is a construction for bicategories that extends the construction of a monoidal category of presheaves on a monoidal category: given a bicategory $\cl A$, there is a bicategory $\widehat{\cl A}$ with the same objects as $\cl A$ and whose hom-categories are the categories of presheaves on the hom-categories of $\cl A$, \ie{} $\widehat{\cl A}(A, B) \defeq \widehat{\cl A(A, B)}$, with composition given by a convolution formula~\cite[\S4]{day1973embedding}. This bicategory $\widehat{\cl A}$ admits an analogous universal property to $\Set^{\M\op}$: it is the \emph{local cocompletion} of $\cl A$, in the sense that it exhibits the free locally cocomplete bicategory extending $\cl A$~\cite[\S5]{kelly2002categories}. (Abstractly, the local cocompletion may be seen to arise from change of base for enriched bicategories~\cite[\S15.8]{garner2016enriched}, though we shall not make use of this perspective here.)

A general advantage of the formalism of double categories over that of bicategories is that double categorical universal properties, such as adjunctions of \dcs, naturally capture local (\ie{} hom-wise) structure in bicategories (in contrast, for instance, to biadjunctions of bicategories, which merely capture global structure). As an illustration of this phenomenon, we shall show that, while the explicit description of the local cocompletion of a bicategory appears to involve `local exponentiation' (whose universal property is not immediately clear), it is an instance of exponentiation for \vdcs. To demonstrate this, we shall need a preliminary concept.

Recall that, for categories, the \emph{full image} of a functor is the category whose objects are those of the domain, and whose morphisms are those of the codomain. For a functor between \vdcs{}, there are two notions of image: the \emph{tight image}, whose tight morphisms are those of the codomain (and whose loose morphisms are those of the domain); and the \emph{loose image}, whose loose morphisms are those in the codomain (and whose tight morphisms are those of the domain). It is the latter that is relevant for what follows.

\begin{definition}
    The \emph{loose image} of a functor $F \colon \A \to \B$ between \vdcs{} is the \vdc{} whose underlying category is $\u\A$, for which a loose morphism $A' \lto A$ is a loose morphism $F(A') \lto F(A)$ in $\B$ and for which a cell with the frame on the left below is a cell in $\B$ with frame on the right below.
    \[
    \begin{tikzcd}
        {A_0} & \cdots & {A_n} \\
        A && {A'}
        \arrow["a"', from=1-1, to=2-1]
        \arrow["{p_1}"'{inner sep=.8ex}, "\shortmid"{marking}, from=1-2, to=1-1]
        \arrow["{p_n}"'{inner sep=.8ex}, "\shortmid"{marking}, from=1-3, to=1-2]
        \arrow["{a'}", from=1-3, to=2-3]
        \arrow["p"{inner sep=.8ex}, "\shortmid"{marking}, from=2-3, to=2-1]
    \end{tikzcd}
    \hspace{6em}
    \begin{tikzcd}
        {F(A_0)} & \cdots & {F(A_n)} \\
        {F(A)} && {F(A')}
        \arrow["{F(a)}"', from=1-1, to=2-1]
        \arrow["{p_1}"'{inner sep=.8ex}, "\shortmid"{marking}, from=1-2, to=1-1]
        \arrow["{p_n}"'{inner sep=.8ex}, "\shortmid"{marking}, from=1-3, to=1-2]
        \arrow["{F(a')}", from=1-3, to=2-3]
        \arrow["p"{inner sep=.8ex}, "\shortmid"{marking}, from=2-3, to=2-1]
    \end{tikzcd}
    \qedshift
    \]
\end{definition}

Every (virtual) bicategory may be viewed as a \emph{tightly discrete} (virtual) \dc, \ie{} a (virtual) \dc{} in which the only tight morphisms are identities. Observe that, viewing a bicategory $\cl A$ as such, we have $\cl A\opt = \cl A\co$. With the notion of loose image at hand, we may exhibit the local cocompletion of a bicategory as arising from exponentiation of \vdcs.

\begin{theorem}
    \label{local-cocompletion-via-exponentiation}
    Let $\cl A$ be a locally small bicategory, viewed as a tightly discrete \pdc. The loose image of the embedding $\cl A({-}_2, {-}_1) \colon \cl A \to \Span^{\cl A\co}$ of \cref{presheaf-embedding} is tightly discrete and representable, and exhibits the local cocompletion $\cl A \to \widehat{\cl A}$.
\end{theorem}

\begin{proof}
    We unwrap the definition of the loose image of the embedding. Trivially, the loose image of any functor from a tightly discrete \vdc{} will itself be tightly discrete, \ie{} a virtual bicategory.
    \begin{itemize}
        \item An object is an object of $\cl A$.
        \item A 1-cell from $X'$ to $X$ comprises the following data, as described in \cref{Span-A-op}.
        \[\begin{tikzcd}
            {\ob{\cl A}} & {\cl A_1\op} & {\ob{\cl A}} \\
            \Set & \Set & \Set
            \arrow[""{name=0, anchor=center, inner sep=0}, "{\cl A_0({-}, X)}"', from=1-1, to=2-1]
            \arrow["{t\op}"', from=1-2, to=1-1]
            \arrow["{s\op}", from=1-2, to=1-3]
            \arrow[""{name=1, anchor=center, inner sep=0}, "\Xi"{description}, from=1-2, to=2-2]
            \arrow[""{name=2, anchor=center, inner sep=0}, "{\cl A_0({-}, X')}", from=1-3, to=2-3]
            \arrow[equals, from=2-2, to=2-1]
            \arrow[equals, from=2-3, to=2-2]
            \arrow[between={0.3}{0.7}, Rightarrow, from=1, to=0]
            \arrow[between={0.3}{0.7}, Rightarrow, from=1, to=2]
        \end{tikzcd}\]
        Here, since $\cl A_0$ is discrete, $\cl A_0({-}, X)$ is the characteristic function, which is equal to a singleton set when applied to $X$, and the empty set otherwise. Consequently, the existence of the two natural transformations above implies that the set $\Xi(p)$ is empty when the domain and codomain of $p$ are not $X'$ and $X$ respectively; subject to this constraint, the two natural transformations are uniquely determined. Consequently, the data of such a 1-cell is equivalently specified by a presheaf on the hom-category $\cl A(X', X)$.
        \item A 2-cell,
        \[\begin{tikzcd}
            {X_0} & \cdots & {X_n} \\
            {X_0} && {X_n}
            \arrow[""{name=0, anchor=center, inner sep=0}, equals, from=1-1, to=2-1]
            \arrow["{\Xi_1}"'{inner sep=.8ex}, "\shortmid"{marking}, from=1-2, to=1-1]
            \arrow["{\Xi_n}"'{inner sep=.8ex}, "\shortmid"{marking}, from=1-3, to=1-2]
            \arrow[""{name=1, anchor=center, inner sep=0}, equals, from=1-3, to=2-3]
            \arrow["\Xi"{inner sep=.8ex}, "\shortmid"{marking}, from=2-3, to=2-1]
            \arrow["\xi"{description}, draw=none, from=0, to=1]
        \end{tikzcd}\]
        where $\Xi_{i + 1} \colon \cl A(X_{i + 1}, X_i)\op \to \Set$ for each $0 \leq i < n$, is given by a natural transformation
        \[\begin{tikzcd}
            {\cl A_n\op} && {\cl A_1\op} \\
            {\Span_n} & {\Span_1} & \Set
            \arrow["{\odot_n\op}", from=1-1, to=1-3]
            \arrow[""{name=0, anchor=center, inner sep=0}, "{\tp{\Xi_1, \ldots, \Xi_n}}"', from=1-1, to=2-1]
            \arrow[""{name=1, anchor=center, inner sep=0}, "\Xi", from=1-3, to=2-3]
            \arrow["{\odot_n}"', from=2-1, to=2-2]
            \arrow["m"', from=2-2, to=2-3]
            \arrow["\xi", between={0.3}{0.7}, Rightarrow, from=0, to=1]
        \end{tikzcd}\]
        satisfying the two equations expressed in \cref{Span-A-op}, hence, for each family $p_i \colon X_{i + 1} \to X_i$ of 1-cells in $\cl A$ (for $0 \leq i < n$), a function
        \[\xi_{p_1, \ldots, p_n} \colon \Xi_1(p_1) \times \cdots \times \Xi_n(p_n) \to \Xi(p_1 \odot \cdots \odot p_n)\]
        natural in $(p_i)_{0 \leq i < n}$. As in the one-object setting, this natural family may equivalently be expressed in terms of a coend as follows, exhibiting the loose image as representable, hence a bicategory.
        \[\xi \colon \int^{p_1, \ldots, p_n} \Xi_1(p_1) \times \cdots \times \Xi_n(p_n) \times \cl A(X_n, X_0)(p_1 \odot \cdots \odot p_n, {-}) \to \Xi\]
    \end{itemize}
    The local cocompletion $\cl A$ is described explicitly, for instance, in \cite[\S5]{kelly2002categories}. It follows by inspection that it agrees with the description of the loose image above.
\end{proof}

Though we shall not do so here, it would be interesting to explore the extent to which \cref{local-cocompletion-via-exponentiation} may be extended from bicategories to \dcs{}, \ie{} whether the loose image of $\A({-}_2, {-}_1) \colon \A \to \Span^{\A\opt}$, for a general \pdc{} $\A$, may also be viewed as some notion of local cocompletion of $\A$.

\subsection{Categories graded by bicategories}

We take the opportunity in this subsection to establish an alternative characterisation of the \pdc{} $\Lax(\cl A\co, \Span)$ of contravariant lax functors from a bicategory $\cl A$ to the bicategory of spans (\ie{} the bicategory underlying the \pdc{} $\Span$ of \cref{Span}), in particular establishing that it arises as the \pdc{} of categories enriched in the bicategory $\widehat{\cl A}$. This serves two purposes: it gives an application of our characterisation of the local cocompletion in \cref{local-cocompletion-via-exponentiation}; and it allows us to draw connections with other constructions in enriched category theory. We shall first establish the characterisation and then discuss the relevant connections.

To begin, we recall the definition of enrichment in a virtual bicategory. Just as internal categories may be defined as loose monads in a \vdc{} of spans (\cref{Mod-Span}), enriched categories may be defined as loose monads in a \vdc{} of matrices. The following definition appears as \cites[Definition~2.62]{kawase2025double}[\S6]{fujii2025familial}; when $\cl V$ is a bicategory with local coproducts, $\cl V\h\Mat$ is representable and agrees with the proarrow equipment of $\cl V$-matrices defined in \cite[\S1]{betti1983variation}.

\begin{definition}
    \label{V-Mat}
    Let $\cl V$ be a virtual bicategory. Define a \vdc{} $\cl V\h\Mat$ as follows.
    \begin{enumerate}
        \item An object is a \emph{$\cl V$-set}, \ie{} a set $\ob A$ together with a function $\ex{\ph} \colon \ob A \to \ob{\cl V}$ defining the \emph{extent} of each element.
        \item A tight morphism is a \emph{$\cl V$-function}, \ie{} a function $\ob f \colon \ob A \to \ob B$ such that, for each $a \in \ob A$, we have $\ex{\ob f a} = \ex a$.
        \item A loose morphism  $m \colon A' \lto A$ is a \emph{$\cl V$-matrix}, \ie{} an assignment taking each $a \in \ob A$ and $a' \in \ob{A'}$ to a loose morphism $m(a, a') \colon \ex{a'} \lto \ex a$ in $\cl V$.
        \item A cell with frame on the left below is a \emph{$\cl V$-transformation}, \ie{} an assignment taking each $a_0 \in \ob{A_0}, \ldots, a_n \in \ob{A_n}$ to a cell in $\cl V$ as on the right below.
        \[
        \begin{tikzcd}
            {A_0} & \cdots & {A_n} \\
            A && {A'}
            \arrow[""{name=0, anchor=center, inner sep=0}, "f"', from=1-1, to=2-1]
            \arrow["{m_1}"'{inner sep=.8ex}, "\shortmid"{marking}, from=1-2, to=1-1]
            \arrow["{m_n}"'{inner sep=.8ex}, "\shortmid"{marking}, from=1-3, to=1-2]
            \arrow[""{name=1, anchor=center, inner sep=0}, "{f'}", from=1-3, to=2-3]
            \arrow["m"{inner sep=.8ex}, "\shortmid"{marking}, from=2-3, to=2-1]
            \arrow["\alpha"{description}, draw=none, from=0, to=1]
        \end{tikzcd}
        \hspace{4em}
        \begin{tikzcd}[column sep=huge]
            {\ex{a_0}} & \cdots & {\ex{a_n}} \\
            {\ex{\ob f(a_0)}} && {\ex{\ob{f'}(a_n)}}
            \arrow[""{name=0, anchor=center, inner sep=0}, equals, from=1-1, to=2-1]
            \arrow["{m_1(a_0, a_1)}"'{inner sep=.8ex}, "\shortmid"{marking}, from=1-2, to=1-1]
            \arrow["{m_n(a_{n - 1}, a_n)}"'{inner sep=.8ex}, "\shortmid"{marking}, from=1-3, to=1-2]
            \arrow[""{name=1, anchor=center, inner sep=0}, equals, from=1-3, to=2-3]
            \arrow["{m(f(a_0), f'(a_n))}"{inner sep=.8ex}, "\shortmid"{marking}, from=2-3, to=2-1]
            \arrow["{\alpha_{a_0, \ldots, a_n}}"{description}, draw=none, from=0, to=1]
        \end{tikzcd}
        \]
    \end{enumerate}
    Composition and identities of tight morphisms are simply those of functions; composition and identities of cells are as in $\cl V$. There is an inclusion functor $\cl V \to \cl V\h\Mat$ sending each object $V \in \cl V$ to the singleton set, whose unique element has extent $V$. $\ph\h\Mat$ extends to a pointed 2-functor $\b{VBicat}_l \to \VDbl$ via postcomposition, where $\b{VBicat}_l$ is the full sub-2-category of $\VDbl$ spanned by the virtual bicategories.
\end{definition}

The following definition appears as \cites[Definition~2.67]{kawase2025double}[\S7]{fujii2025familial}; when $\cl V$ is a locally cocomplete bicategory, $\cl V\h\Dist$ is representable and agrees with the proarrow equipment of $\cl V$-modules defined in \cite[\S3]{betti1983variation}.

\begin{definition}
    \label{V-Dist}
    For a virtual bicategory $\cl V$, define the normal \vdc{} $\cl V\h\Dist \defeq \Mod(\cl V\h\Mat)$ of $\cl V$-enriched categories, $\cl V$-enriched functors, $\cl V$-enriched distributors, and $\cl V$-enriched natural transformations.
\end{definition}

The reader who has not previously encountered the notion of enrichment in a bicategory is invited to expand \cref{V-Dist}; the explicit definition will not be necessary for what follows. To establish the promised characterisation of $\Lax(\cl A\co, \Span)$, observe that both it and $\cl V\h\Dist$ arise by applying the $\Mod$ construction to simpler \vdcs{} (\cref{Lax-is-Mod,V-Dist}). Thus, to establish that $\Lax(\cl A\co, \Span)$ arises as a \vdc{} of enriched categories, it suffices to establish that the exponential \vdc{} $\Span^{\cl A\co}$ is equivalent to a \vdc{} of matrices.

\begin{proposition}
    \label{local-cocompletion-matrices-are-indexed-spans}
    For each small bicategory $\cl A$, there is an equivalence of \pdcs, 2-natural in $\cl A$:
    \[\widehat{\cl A}\h\Mat \equiv \Span^{\cl A\co}\]
\end{proposition}

\begin{proof}
    First, observe that $\widehat{\cl A}\h\Mat$ is representable, since $\widehat{\cl A}$ is, by definition, a locally cocomplete bicategory. The functor $\cl A \to \widehat{\cl A} \to \widehat{\cl A}\h\Mat$ induces a nerve $\widehat{\cl A}\h\Mat \to \Span^{\cl A\co}$ by \cref{nerve}: this is an equivalence. The underlying functor of categories is simply the equivalence
    \[\Set/\ob{\widehat{\cl A}} = \Set/\ob{\cl A} \equiv \Set^{\ob{\cl A}}\]
    between sets over $\ob{\cl A}$ and sets indexed by $\ob{\cl A}$. Given $\widehat{\cl A}$-sets $\epsilon \colon \ob A \to \ob{\widehat{\cl A}}$ and $\epsilon' \colon \ob{A'} \to \ob{\widehat{\cl A}}$ (where $\ob{\widehat{\cl A}} = \ob{\cl A}$), a loose morphism $\Xi$ in $\Span^{\cl A\op}$ (as in \cref{Span-A-op}) between the corresponding indexed sets
    \[\begin{tikzcd}
        {\ob{\cl A}} & {\cl A_1\op} & {\ob{\cl A}} \\
        \Set & \Set & \Set
        \arrow[""{name=0, anchor=center, inner sep=0}, "{\epsilon\inv}"', from=1-1, to=2-1]
        \arrow[""{name=0p, anchor=center, inner sep=0}, phantom, from=1-1, to=2-1, start anchor=center, end anchor=center]
        \arrow["{t\op}"', from=1-2, to=1-1]
        \arrow["{s\op}", from=1-2, to=1-3]
        \arrow[""{name=1, anchor=center, inner sep=0}, "\Xi"{description}, from=1-2, to=2-2]
        \arrow[""{name=1p, anchor=center, inner sep=0}, phantom, from=1-2, to=2-2, start anchor=center, end anchor=center]
        \arrow[""{name=1p, anchor=center, inner sep=0}, phantom, from=1-2, to=2-2, start anchor=center, end anchor=center]
        \arrow[""{name=2, anchor=center, inner sep=0}, "{(\epsilon')\inv}", from=1-3, to=2-3]
        \arrow[""{name=2p, anchor=center, inner sep=0}, phantom, from=1-3, to=2-3, start anchor=center, end anchor=center]
        \arrow[equals, from=2-2, to=2-1]
        \arrow[equals, from=2-3, to=2-2]
        \arrow["{\Xi_t}"', between={0.3}{0.7}, Rightarrow, from=1p, to=0p]
        \arrow["{\Xi_s}", between={0.3}{0.7}, Rightarrow, from=1p, to=2p]
    \end{tikzcd}\]
    is isomorphic to the image, under the nerve, of the $\widehat{\cl A}$-matrix that sends $a \in \ob A$ and $a' \in \ob{A'}$ to the presheaf $\Xi_{\ex{a'}, \ex a} \colon \cl A(\ex{a'}, \ex a)\op \to \Set$. Thus, the nerve is essentially surjective on loose morphisms. By \cref{local-cocompletion-via-exponentiation}, a cell in $\widehat{\cl A}\h\Mat$ comprises, for each $a_0 \in \ob{A_0}, \ldots, a_n \in \ob{A_n}$, a cell with the following frame in $\Span^{\cl A\co}$. Using once more the correspondence between sets fibred over $\ob{\cl A}$ and $\ob{\cl A}$-indexed sets, this data is in bijection with cells in $\Span^{\cl A\co}$ between the corresponding indexed sets. Thus, the nerve is \ff.
    \[\begin{tikzcd}[column sep=huge]
        {\cl A({-}, \ex{a_0})} & \cdots & {\cl A({-}, \ex{a_n})} \\
        {\cl A({-}, \ex{\ob f(a_0)})} && {\cl A({-}, \ex{\ob{f'}(a_n)})}
        \arrow[equals, from=1-1, to=2-1]
        \arrow["{m_1(a_0, a_1)}"'{inner sep=.8ex}, "\shortmid"{marking}, from=1-2, to=1-1]
        \arrow["{m_n(a_{n - 1}, a_n)}"'{inner sep=.8ex}, "\shortmid"{marking}, from=1-3, to=1-2]
        \arrow[equals, from=1-3, to=2-3]
        \arrow["{m(f(a_0), f'(a_n))}"{inner sep=.8ex}, "\shortmid"{marking}, from=2-3, to=2-1]
    \end{tikzcd}\qedshift\]
\end{proof}

We consequently obtain our promised characterisation.

\begin{theorem}[Cockett--Niefield--Wood]
    \label{lax-functors-are-graded-categories}
    For each small bicategory $\cl A$, there is an equivalence of \pdcs, 2-natural in $\cl A$:
    \[\widehat{\cl A}\h\Dist \equiv \Lax(\cl A\co, \Span)\]
\end{theorem}

\begin{proof}
    First, observe that $\widehat{\cl A}\h\Dist$ is representable, since $\widehat{\cl A}$ is, by definition, a locally cocomplete bicategory. The equivalence then follows immediately from applying $\Mod$ to the equivalence of \cref{local-cocompletion-matrices-are-indexed-spans}, using \cref{Lax-is-Mod}.
\end{proof}

\Cref{lax-functors-are-graded-categories} recovers an unpublished correspondence due to Cockett, Niefield and Wood~\cite{niefield2012double}. When $\cl A$ is a 2-category, the equivalence of 2-categories underlying \cref{lax-functors-are-graded-categories} appears as \cite[Proposition~6.1]{cockett2014restriction}.

When $\cl A = \Sigma\b V$ is the delooping of a monoidal category, a $\widehat{\cl A}$-enriched category is what \textcite{wood1976indicial} calls a \emph{large $\b V$-category} and what is now generally referred to as a \emph{$\b V$-graded category}~\cite{lucyshyn2025graded}. $\b V$-graded categories generalise $\b V$-enriched categories and provide a useful setting in which to consider notions of strength~\cite{wood1976indicial}. In the following section, we will show that contravariant lax functors into $\Span$ correspond to discrete fibrations of \dcs{} (\cref{elements-correspondence}). In particular, considering discrete fibrations over $\Sigma\b V$ recovers the notion of \emph{$\b V$-graduated bicategory} of \textcite[Definition~8]{guitart1980tenseurs}, which was introduced as a reformulation of the notion of $\b V$-graded category.

As with \cref{local-cocompletion-via-exponentiation}, it would be interesting to explore the extent to which \cref{lax-functors-are-graded-categories} extends from bicategories to \pdcs.

\section{The elements construction}
\label{elements-construction}

The data of the exponential $\Span^{\A\opt}$ in \cref{Span-A-op} is expressed in terms of presheaves (of categories) and natural transformations therebetween, and we are inevitably led to try transferring this data across the correspondence between presheaves and discrete fibrations. This motivates the following definition.

\begin{definition}
    \label{LvAmnIso}
    Let $\A$ be a \pdc. We define a \vdc{} $\LvAmnIso(\A)$ (for \emph{levelwise amnestic isofibrations}) as follows.
    \begin{itemize}
        \item The underlying category is the full subcategory of the slice category $\Cat/\A_0$ spanned by the amnestic isofibrations\footnotemark{}.
        \footnotetext{An amnestic isofibration is analogous to a discrete (op)fibration, but in which we only ask for unique lifts of \emph{iso}morphisms.}
        \item A loose morphism from $D' \colon \D' \to \A_0$ to $D \colon \D \to \A_0$ is given by an amnestic isofibration $E \colon \E \to \A_1$ and a span rendering the following diagram commutative.
        \[\begin{tikzcd}
            \D & \E & {\D'} \\
            {\A_0} & {\A_1} & {\A_0}
            \arrow["D"', from=1-1, to=2-1]
            \arrow["{E_t}"', from=1-2, to=1-1]
            \arrow["{E_s}", from=1-2, to=1-3]
            \arrow["E"{description}, from=1-2, to=2-2]
            \arrow["{D'}", from=1-3, to=2-3]
            \arrow["t", from=2-2, to=2-1]
            \arrow["s"', from=2-2, to=2-3]
        \end{tikzcd}\]
        \item A cell $\zeta$ as on the left below (where each loose morphism $E_i$ on the left is given explicitly by the data of functors $E_i$, $(E_i)_s$ and $(E_i)_t$, as on the right below)
        \[
        \begin{tikzcd}
            {D_0} & \cdots & {D_n} \\
            D && {D'}
            \arrow[""{name=0, anchor=center, inner sep=0}, "H"', from=1-1, to=2-1]
            \arrow["{E_1}"'{inner sep=.8ex}, "\shortmid"{marking}, from=1-2, to=1-1]
            \arrow["{E_n}"'{inner sep=.8ex}, "\shortmid"{marking}, from=1-3, to=1-2]
            \arrow[""{name=1, anchor=center, inner sep=0}, "{H'}", from=1-3, to=2-3]
            \arrow["E"{inner sep=.8ex}, "\shortmid"{marking}, from=2-3, to=2-1]
            \arrow["\zeta"{description}, draw=none, from=1, to=0]
        \end{tikzcd}
        \hspace{3em}
        \begin{tikzcd}
            {\D_{i - 1}} & {\E_i} & {\D_i} \\
            {\A_0} & {\A_1} & {\A_0}
            \arrow["{D_{i - 1}}"', from=1-1, to=2-1]
            \arrow["{(E_i)_t}"', from=1-2, to=1-1]
            \arrow["{(E_i)_s}", from=1-2, to=1-3]
            \arrow["{E_i}"{description}, from=1-2, to=2-2]
            \arrow["{D_i}", from=1-3, to=2-3]
            \arrow["t", from=2-2, to=2-1]
            \arrow["s"', from=2-2, to=2-3]
        \end{tikzcd}
        \]
        is given by a functor $\zeta$ making the following two diagrams commute.
        \[
        \begin{tikzcd}
            {\E_{1, \ldots, n}} & \E \\
            {\A_n} & {\A_1}
            \arrow["\zeta", from=1-1, to=1-2]
            \arrow["{E_{1, \ldots, n}}"', from=1-1, to=2-1]
            \arrow["E", from=1-2, to=2-2]
            \arrow["{\odot_n}"', from=2-1, to=2-2]
        \end{tikzcd}
        \hspace{3em}
        \begin{tikzcd}[column sep=large]
            {\D_0} & {\E_1} & {\E_{1, \ldots, n}} & {\E_n} & {\D_n} \\
            \D && \E && {\D'}
            \arrow["H"', from=1-1, to=2-1]
            \arrow["{(E_1)_t}"', from=1-2, to=1-1]
            \arrow["{\pi_1}"', from=1-3, to=1-2]
            \arrow["{\pi_n}", from=1-3, to=1-4]
            \arrow["\zeta"{description}, from=1-3, to=2-3]
            \arrow["{(E_n)_s}", from=1-4, to=1-5]
            \arrow["{H'}", from=1-5, to=2-5]
            \arrow["{E_t}", from=2-3, to=2-1]
            \arrow["{E_s}"', from=2-3, to=2-5]
        \end{tikzcd}
        \]
        Above, we have denoted by $\E_{1, \ldots, n}$ the limit (\ie{} wide pullback) of the following chain of spans, and by $E_{1, \ldots, n} \colon \E_{1, \ldots, n} \to \A_n$ the canonical functor induced by the functors $E_1, \ldots, E_n$, using that $\A_n \iso \A_1 \times_{\A_0} \cdots \times_{\A_0} \A_1$.
        \[\begin{tikzcd}
            & {\E_1} && {\E_2} && \cdots \\
            {\D_0} && {\D_1} && {\D_2} && {\D_n}
            \arrow["{(E_1)_t}"', from=1-2, to=2-1]
            \arrow["{(E_1)_s}"{description}, from=1-2, to=2-3]
            \arrow["{(E_2)_t}"{description}, from=1-4, to=2-3]
            \arrow["{(E_2)_s}"{description}, from=1-4, to=2-5]
            \arrow["{(E_3)_t}"{description}, from=1-6, to=2-5]
            \arrow["{(E_n)_s}", from=1-6, to=2-7]
        \end{tikzcd}\]
        \item The identity cell on $E \colon \E \to \A_1$ is given by the identity functor on $\E$.
        \item For composition, observe that the diagram on the left below commutes, but the diagram on the right below commutes only up to isomorphism, using the structural isomorphisms in $\A$.
        \[
        \begin{tikzcd}[column sep=large]
            {\E_{1, \ldots, \sum_{1 \leq i \leq n} m_i}} & {\E_{1, \ldots, n}} & \E \\
            {\A_{\sum_{1 \leq i \leq n} m_i}} & {\A_n} & {\A_1}
            \arrow["{\zeta_1, \ldots, \zeta_n}", from=1-1, to=1-2]
            \arrow["{E_{1, \ldots, \sum_{1 \leq i \leq n} m_i}}"{description}, from=1-1, to=2-1]
            \arrow["\zeta", from=1-2, to=1-3]
            \arrow["{E_{1, \ldots, n}}"{description}, from=1-2, to=2-2]
            \arrow["E", from=1-3, to=2-3]
            \arrow["{\tp{\odot_{m_1}, \ldots, \odot_{m_n}}}"', from=2-1, to=2-2]
            \arrow["{\odot_n}"', from=2-2, to=2-3]
        \end{tikzcd}
        \hspace{2em}
        \begin{tikzcd}[column sep=large]
            {\E_{1, \ldots, \sum_{1 \leq i \leq n} m_i}} & {\E_{1, \ldots, n}} & \E \\
            {\A_{\sum_{1 \leq i \leq n} m_i}} && {\A_1}
            \arrow["{\zeta_1, \ldots, \zeta_n}", from=1-1, to=1-2]
            \arrow[""{name=0, anchor=center, inner sep=0}, "{E_{1, \ldots, \sum_{1 \leq i \leq n} m_i}}"{description}, from=1-1, to=2-1]
            \arrow["\zeta", from=1-2, to=1-3]
            \arrow[""{name=1, anchor=center, inner sep=0}, "E", from=1-3, to=2-3]
            \arrow["{\odot_{\sum_{1 \leq i \leq n} m_i}}"', from=2-1, to=2-3]
            \arrow["\iso"{description}, draw=none, from=0, to=1]
        \end{tikzcd}
        \]
        Since $E \colon \E \to \A_1$ is an amnestic isofibration, the invertible natural transformation on the right above lifts, defining a functor $\E_{1, \ldots, \sum_{1 \leq i \leq n} m_i} \to \E$ as below, which we take to be the composite.
        \[\begin{tikzcd}[column sep=large]
            {\E_{1, \ldots, \sum_{1 \leq i \leq n} m_i}} & {\E_{1, \ldots, n}} & \E \\
            {\A_{\sum_{1 \leq i \leq n} m_i}} && {\A_1}
            \arrow["{\zeta_1, \ldots, \zeta_n}", from=1-1, to=1-2]
            \arrow[""{name=0, anchor=center, inner sep=0}, curve={height=30pt}, dashed, from=1-1, to=1-3]
            \arrow["{E_{1, \ldots, \sum_{1 \leq i \leq n} m_i}}"', from=1-1, to=2-1]
            \arrow["\zeta", from=1-2, to=1-3]
            \arrow["E", from=1-3, to=2-3]
            \arrow["{\odot_{\sum_{1 \leq i \leq n} m_i}}"', from=2-1, to=2-3]
            \arrow["\iso"{description}, draw=none, from=1-2, to=0]
        \end{tikzcd}\]
    \end{itemize}
    That composition of cells is indeed associative and unital follows from uniqueness of liftings for the amnestic isofibrations. $\LvAmnIso$ extends to a 2-functor $\Dbl_l\op \to \VDbl$ via pullback of categories (under which amnestic isofibrations are stable).

    Define $\LvDFib(\A)$ (for \emph{levelwise discrete fibrations}) to be the full sub-\vdc{} of $\LvAmnIso(\A)$ whose objects and loose morphisms are the discrete fibrations over $\A_0$ and $\A_1$ respectively, which likewise extends to a 2-functor $\Dbl_l\op \to \VDbl$.
\end{definition}

\begin{remark}
    We introduced $\LvDFib(\A)$ by way of $\LvAmnIso(\A)$ above to emphasise which properties of discrete fibrations are important in defining the \vdc{} structure (namely, the isomorphism lifting property). However, in what follows, we will focus on $\LvDFib(\A)$, since it is this construction that is relevant to the theory of \cref{Yoneda-theory}. Let us simply remark that several of the properties below (\eg{} \cref{LvDFib-as-slice,DFib-fibre-as-slice}) also hold at the greater level of generality of amnestic isofibrations.
\end{remark}

This definition is justified by the following, which establishes that $\LvDFib(\A)$ is indeed the fibrational analogue of $\Span^{\A\opt}$.

\begin{proposition}
    \label{Span-A-op-equiv-DFib-A}
    For each \pdc{} $\A$, there is an equivalence of \vdcs{} $\Span^{\A\opt} \equiv \LvDFib(\A)$, 2-natural in $\A$.
\end{proposition}

\begin{proof}
    The explicit description of $\Span^{\A\opt}$ in \cref{Span-A-op} is formulated in terms of presheaves on $\A_n$ (for each $n \in \N$) and natural transformations therebetween. Thus, we may transfer the descriptions of the data of the objects, tight morphisms, loose morphisms, cells, identities and composites in $\Span^{\A\opt}$ across the correspondence between presheaves (and natural transformations therebetween), and discrete fibrations (and commutative squares therebetween): this produces precisely the \vdc{} $\LvDFib(\A)$. To illustrate the manipulations involved in transferring data across the equivalence, consider the data of an $n$-ary cell in $\Span^{\A\opt}$ described in \cref{Span-A-op}.
    \[\begin{tikzcd}
        {\A_n\op} && {\A_1\op} \\
        {\Span_n} & {\Span_1} & \Set
        \arrow["{\odot_n\op}", from=1-1, to=1-3]
        \arrow[""{name=0, anchor=center, inner sep=0}, "{\tp{\Xi_1, \ldots, \Xi_n}}"', from=1-1, to=2-1]
        \arrow[""{name=1, anchor=center, inner sep=0}, "\Xi", from=1-3, to=2-3]
        \arrow["{\odot_n}"', from=2-1, to=2-2]
        \arrow["m"', from=2-2, to=2-3]
        \arrow["\xi", between={0.3}{0.7}, Rightarrow, from=0, to=1]
    \end{tikzcd}\]
    Denote by $E_1, \ldots, E_n, E$ the discrete fibrations corresponding, via the category of elements construction, to the presheaves $\Xi_1, \ldots, \Xi_n, \Xi$. Recalling that precomposition for presheaves corresponds to pullback for discrete fibrations, the discrete fibration corresponding to the codomain of $\xi$ is precisely the pullback of $E$ along $\odot_n \colon \A_n \to \A_1$. We may directly compute the category of elements of the domain of $\xi$: for instance, the objects comprise pairs of a chain $p_1, \ldots, p_n$ of loose morphisms in $\A$ and an element of the limit (\ie{} wide pullback) of the diagram $\tp{\Xi_1, \ldots, \Xi_n}$. Denoting this discrete fibration suggestively by $E_{1, \ldots, n} \colon \E_{1, \ldots, n} \to \A_n$, we see that $\xi$ corresponds to a functor from $\E_{1, \ldots, n}$ into the aforementioned pullback, whose composition with the projection into $\A_n$ must be $E_{1, \ldots, n}$. This is precisely the data of an $n$-ary cell in $\LvDFib(\A)$. Verification that the composition of cells in the two \vdcs{} agree is a similarly straightforward exercise. 2-naturality follows likewise from 2-naturality of the correspondence between presheaves and discrete fibrations on categories.
\end{proof}

As is suggested by \cref{Lax-is-Mod}, it is instructive to consider loose monads in $\LvDFib(\A)$. These admit a particularly simple description in the special case in which $\A$ is strict, by virtue of the fact that strict double categories are precisely loose monads in $\Span(\Cat)$ (in contrast, \pdcs{} are merely pseudomonads, the consideration of which requires three-dimensional structure as in \cite{cruttwell2022double}). To explain this, we will need a couple of preparatory lemmas.

\begin{lemma}
    \label{LvDFib-as-slice}
    For each strict \dc{} $\A$, the \vdc{} $\LvDFib(\A)$ is isomorphic to the full sub-\vdc{} $\Span(\Cat)\dfibsl\A$ of the slice $\Span(\Cat)/\A$ (\cref{slice-vdc}) spanned by the functors and span morphisms whose components are discrete fibrations.
\end{lemma}

\begin{proof}
    The objects, tight morphisms, loose morphisms, cells, and identities of the two \vdcs{} are trivially in bijection with one another; the only subtlety is the relationship between composition of cells in the two \vdcs. Examining the definition of composition of cells in $\LvDFib(\A)$ (\cref{LvAmnIso}), under the assumption that $\A$ is strict, we observe that following diagram commutes strictly (rather than up to isomorphism as in a \pdc).
    \[\begin{tikzcd}
        {\A_{\sum_{1 \leq i \leq n} m_i}} \\
        {\A_n} & {\A_1}
        \arrow["{\tp{\odot_{m_1}, \ldots, \odot_{m_n}}}"', from=1-1, to=2-1]
        \arrow["{\odot_{\sum_{1 \leq i \leq n} m_i}}", from=1-1, to=2-2]
        \arrow["{\odot_n}"', from=2-1, to=2-2]
    \end{tikzcd}\]
    Consequently, the lift appearing in the definition of composition of cells is trivial, so that composition is given by the following composite functor, in which the left-hand functor is induced by the universal property of the pullback $\E_{1, \ldots, n}$.
    \[\E_{1, \ldots, \sum_{1 \leq i \leq n} m_i} \xto{\zeta_1, \ldots, \zeta_n} \E_{1, \ldots, n} \xto\zeta \E\]
    This corresponds exactly to the composition of cells in $\Span(\Cat)/\A$.
\end{proof}

\begin{remark}
    Let $(\b A, \otimes, I)$ be a strict monoidal category. By \cref{Span-A-op-equiv-DFib-A,LvDFib-as-slice}, we have the following equivalence of \vdcs.
    \[\Span(\Cat)\dfibsl(\Sigma\b A) \iso \LvDFib(\Sigma\b A) \equiv \Span^{\Sigma(\b A\op)}\]
    By fixing an object of the foremost \vdc{} (\viz{} a discrete fibration over $(\Sigma\b A)_0 = \b1$), the equivalence of \vdcs{} restricts to an equivalence of multicategories (viewed as one-object \vdcs).

    In particular, the identity functor on the terminal category is sent under the equivalence to the terminal presheaf $\b1\op \to \Set$. As a result, by inspection of \cref{Span-A-op}, we obtain an equivalence between the following (representable) multicategories.
    \begin{enumerate}
        \item The multicategory whose objects are discrete fibrations over the strict monoidal category $\b A$, for which a multimorphism $E_1, \ldots, E_n \to E$ is given by a functor $\E_1 \times \cdots \times \E_n \to \E$ rendering the following diagram commutative, and for which composition is induced by the universal property of the products.
        \[\begin{tikzcd}
            {\E_1 \times \cdots \times \E_n} & \E \\
            {\b A \times \cdots \times \b A} & {\b A}
            \arrow[dashed, from=1-1, to=1-2]
            \arrow["{E_1 \times \cdots \times E_n}"', from=1-1, to=2-1]
            \arrow["E", from=1-2, to=2-2]
            \arrow["{\otimes_n}"', from=2-1, to=2-2]
        \end{tikzcd}\]
        \item The exponential multicategory $\Set^{\b A\op}$.
    \end{enumerate}
    As discussed in \cref{relation-to-multicategories}, (2) is precisely the convolution monoidal structure induced by $\b A$. We thereby recover \cite[Theorem~73]{eberhart2019template}. The authors of \cite{eberhart2019template} indicated that it would be desirable to extend their work to the setting in which $\b A$ is not assumed strict~\cite[\S6]{eberhart2019template}, for which they anticipated working in a three-dimensional setting. Our \cref{LvDFib-as-slice} suggests that an alternative, potentially simpler approach would be to work with $\LvDFib(\Sigma\b A)$ rather than $\Span(\Cat)\dfibsl(\Sigma\b A)$: in this case, the only difference in the description of (1) above is that composition is no longer given solely by the universal property of the products, but involves lifting a coherence isomorphism along $E$ as in \cref{LvAmnIso}.
\end{remark}

\begin{definition}
    $\DFib$ is the full subcategory of the arrow category $\Cat^\to$ spanned by the discrete fibrations.
\end{definition}

\begin{lemma}
    \label{DFib-fibre-as-slice}
    For each strict \dc{} $\A$, the apex of the following pullback in $\VDbl$ is isomorphic to $\Span(\Cat)\dfibsl\A$ (defined in \cref{LvDFib-as-slice}).
    \[\begin{tikzcd}
        {\Span(\DFib)_\A} & {\Span(\DFib)} \\
        {\bbn 1} & {\Span(\Cat)}
        \arrow[from=1-1, to=1-2]
        \arrow[""{name=0, anchor=center, inner sep=0}, "\unit"', from=1-1, to=2-1]
        \arrow[""{name=1, anchor=center, inner sep=0}, "{\Span(\cod)}", from=1-2, to=2-2]
        \arrow["\A"', from=2-1, to=2-2]
        \arrow["\pb"{description}, draw=none, from=0, to=1]
    \end{tikzcd}\]
\end{lemma}

\begin{proof}
    Consider the following diagram. Since $\Cat^\to$ is a power by the arrow category, hence a comma object, and $\Span$ preserves 2-limits, the right-hand square is also a comma diagram. By the pasting law for comma objects (\cref{pasting-law}), the outer rectangle is a comma object, which is precisely the universal property of the slice $\Span(\Cat)/\A$.
    \[\begin{tikzcd}[column sep=huge]
        {\Span(\Cat^\to)_\A} & {\Span(\Cat^\to)} & {\Span(\Cat)} \\
        {\bbn 1} & {\Span(\Cat)} & {\Span(\Cat)}
        \arrow[from=1-1, to=1-2]
        \arrow[""{name=0, anchor=center, inner sep=0}, "\unit"', from=1-1, to=2-1]
        \arrow["{\Span(\dom)}", from=1-2, to=1-3]
        \arrow[""{name=1, anchor=center, inner sep=0}, "{\Span(\cod)}"{description}, from=1-2, to=2-2]
        \arrow[between={0.4}{0.6}, Rightarrow, from=1-3, to=2-2]
        \arrow[equals, from=1-3, to=2-3]
        \arrow["\A"', from=2-1, to=2-2]
        \arrow[equals, from=2-2, to=2-3]
        \arrow["\pb"{description}, draw=none, from=0, to=1]
    \end{tikzcd}\]
    That this restricts to discrete fibrations, as stated, follows by pasting the following pullback on top of the pullback above.
    \[\begin{tikzcd}
        {\Span(\Cat)\dfibsl\A} & {\Span(\DFib)} \\
        {\Span(\Cat)/\A} & {\Span(\Cat^\to)}
        \arrow[from=1-1, to=1-2]
        \arrow[""{name=0, anchor=center, inner sep=0}, hook', from=1-1, to=2-1]
        \arrow[""{name=1, anchor=center, inner sep=0}, hook', from=1-2, to=2-2]
        \arrow[from=2-1, to=2-2]
        \arrow["\pb"{description}, draw=none, from=0, to=1]
    \end{tikzcd}\qedshift\]
\end{proof}

Discrete fibrations of strict and \pdcs{} have appeared before in the literature in several guises\footnotemark{}. The following relates our \vdc{} $\LvDFib(\A)$ to each of the notions appearing in the literature.
\footnotetext{Perhaps the earliest definition of a (not necessarily discrete) fibration of \dcs{} is \cite[Definition~1.7]{dawson2006paths}.}%

\begin{theorem}
    \label{DFib-is-Mnd}
    Let $\A$ be a \pdc{}.
    \begin{enumerate}
        \item For $\A$ strict, the \vdc{} $\DFibDbl(\A)$ of \cite[Definition~3.24]{lambert2021discrete} is equivalent to $\Mod(\LvDFib(\A))$.
        \item For $\A$ strictly normal, the category $\b{DDblFib}(\A)$ of \cite[Definition~3.43 \& (3.20)]{cruttwell2022double} is equivalent to the underlying category of $\Mod(\LvDFib(\A))$.
        \item For $\A$ strict, the 2-category $\tx D\cl{F}\tx{ib}(\A)$ of \cite[Notation~5.2]{frohlich2024yoneda} is equivalent to the underlying 2-category of $\Mod(\LvDFib(\A))$.
    \end{enumerate}
\end{theorem}

\begin{proof}
    \begin{enumerate}
        \item The \vdc{} of discrete fibrations over a \sdc{} $\A$ is defined by \citeauthor{lambert2021discrete} to be the following pullback in $\VDbl$.\footnote{The reader should be warned that \textcite[Example~3.23]{lambert2021discrete} uses the notation $\Dist(\DFib)/\A$ for what we would call $\Dist(\DFib)_\A$ (though it is not a slice construction).}
        \[\begin{tikzcd}
            {\DFibDbl(\A)} & {\Dist(\DFib)} \\
            {\bbn 1} & {\Dist(\Cat)}
            \arrow[from=1-1, to=1-2]
            \arrow[""{name=0, anchor=center, inner sep=0}, "\unit"', from=1-1, to=2-1]
            \arrow[""{name=1, anchor=center, inner sep=0}, "{\Dist(\cod)}", from=1-2, to=2-2]
            \arrow["\A"', from=2-1, to=2-2]
            \arrow["\pb"{description}, draw=none, from=0, to=1]
        \end{tikzcd}\]
        Since $\Mod$ preserves 2-limits (\cref{Mod-is-right-adjoint}), this pullback is isomorphic by \cref{DFib-fibre-as-slice} to the \vdc{} $\Dist(\Cat)\dfibsl\A \defeq \Mod(\Span(\Cat)\dfibsl\A)$. Using \cref{Mod-slice}, we have that $\Mod(\Span(\Cat)/\A) \iso \Mod(\Span(\Cat))/\A \iso \Dist(\Cat)/\A$, from which it follows that $\Dist(\Cat)\dfibsl\A$ is precisely the full sub-\vdc{} of $\Dist(\Cat)/\A$ whose objects' and loose morphisms' component functors into $\A_0$ and $\A_1$ are discrete fibrations. Finally, by \cref{LvDFib-as-slice}, we have that $\Span(\Cat)\dfibsl\A \iso \LvDFib(\A)$, from which the equivalence follows by applying $\Mod$.
        \item Explicitly, an object of $\b{DDblFib}(\A)$ is a strict functor of \pdcs{} ${E \colon \dc E \to \A}$ (\ie{} a functor preserving the chosen opcartesian cells) such that $E_0 \colon \dc E_0 \to \A_0$ and $E_1 \colon \dc E_1 \to \A_1$ are discrete fibrations; and a morphism is a commutative triangle of strict functors of \pdcs{}. By inspection of \cref{LvAmnIso}, a loose endomorphism in $\LvDFib(\A)$ comprises a pair of discrete fibrations $E_0 \colon \E_0 \to \A_0$ and $E_1 \colon \E_1 \to \A_1$ and a morphism of spans as follows.
        \[\begin{tikzcd}
            {\E_0} & {\E_1} & {\E_0} \\
            {\A_0} & {\A_1} & {\A_0}
            \arrow["{E_0}"', from=1-1, to=2-1]
            \arrow["{E_t}"', from=1-2, to=1-1]
            \arrow["{E_s}", from=1-2, to=1-3]
            \arrow["{E_1}"{description}, from=1-2, to=2-2]
            \arrow["{E_0}", from=1-3, to=2-3]
            \arrow["t", from=2-2, to=2-1]
            \arrow["s"', from=2-2, to=2-3]
        \end{tikzcd}\]
        By unwrapping the definitions, we see that equipping this loose endomorphism with the structure of a loose monad amounts to specifying composition and identity structure on the graph of categories $\E_1 \rightrightarrows \E_0$ that is strictly preserved by $E_0$ and $E_1$. This is almost the structure of a normal \pdc{}: it remains to specify associativity isomorphisms in $\E_1$, which are uniquely determined by lifting the associativity isomorphisms in $\A$ along $E_1 \colon \E_1 \to \A_1$. Consequently, $\E_1 \rightrightarrows \E_0$ is equipped with the structure of a normal \pdc{} $\dc E$; note that, in contrast to loose monads in $\Span(\Cat)$ (which are \emph{strict} \dcs), the reason a loose monad in $\LvDFib(\A)$ need not be strict is that the composition of cells in $\LvDFib(\A)$ involves the coherence isomorphisms in $\A$. Since the structural isomorphisms in $\dc E$ are given by lifting those in $\A$, the functors $E_0$ and $E_1$ form a strict functor of \pdcs, whose components are discrete fibrations by definition. A similar analysis shows that a loose monad morphism in $\LvDFib(\A)$ is precisely a strict functor, the data of the tight morphism determining the action on $\E_0$ and the data of the cell determining the action on~$\E_1$.
        \item The 2-category of discrete fibrations over a \sdc{} $\A$ is defined by \citeauthor{frohlich2024yoneda} as the underlying 2-category of the full sub-\vdc{} of $\Dist(\Cat)/\A$ spanned by the strict functors of \dcs{} that satisfy a certain pullback condition~\cite[Definition~5.1]{frohlich2024yoneda}. By \cite[Proposition~2.17]{lambert2021discrete}, this condition is equivalent to asking for the components of the span morphism underlying a functor of \dcs{} to be discrete fibrations. Consequently, the equivalence follows as in~(1).
        \qedhere
    \end{enumerate}
\end{proof}

A \vdc{} of discrete fibrations over a \emph{pseudo} \dc{} has not previously been defined in the literature. \Cref{DFib-is-Mnd} justifies us in defining $\DFibDbl(\A)$ to be the \vdc{} $\Mod(\LvDFib(\A))$. Having set up our definitions, we immediately obtain a correspondence between presheaves on, and discrete fibrations over, a \pdc{} $\A$.

\begin{corollary}
    \label{elements-correspondence}
    For each \pdc{} $\A$, there are equivalences of normal \vdcs{}, 2-natural in $\A$:
    \[\Lax(\A\opt, \Span) \iso \LaxN(\A\opt, \Dist) \equiv \DFibDbl(\A)\]
\end{corollary}

\begin{proof}
    The isomorphism follows from \cref{Mod-is-internal-right-adjoint}. For the equivalence, from \cref{Span-A-op-equiv-DFib-A}, we have $\Span^{\A\opt} \equiv \LvDFib(\A)$, whereupon applying $\Mod$ gives \[\Lax(\A\opt, \Span) \iso \Mod(\Span^{\A\opt}) \equiv \Mod(\LvDFib(\A)) \iso \DFibDbl(\A)\] by \cref{Lax-is-Mod} and by definition.
\end{proof}

\begin{example}
    Let $\b A$ be a category and denote by $\cl A$ the corresponding locally discrete 2-category. Then \cref{elements-correspondence} recovers the classical correspondence between the following~\cite[Theorem~3.16]{pare2011yoneda}.
    \begin{enumerate}
        \item Lax functors $\cl A\co \to \b{\cl Span}$ into the bicategory of spans.
        \item Normal lax functors $\cl A\co \to \b{\cl Dist}$ into the bicategory of distributors.
        \item Functors into $\b A$.
    \end{enumerate}
    \Cref{lax-functors-are-graded-categories} further extends this correspondence to the following.
    \begin{enumerate}[resume]
        \item Categories enriched in the local cocompletion of $\cl A$.
    \end{enumerate}
    The local cocompletion of a locally discrete 2-category admits an alternative description: it is the slice of $\Set$ over the category $\b A$, in the sense of \cite{fujii2024oplax}, and as such we recover the special case of \cite[Theorem~4.5]{fujii2024oplax} in which the base of enrichment is $\Set$ (\cf~\cite[Example~4.6]{fujii2024oplax}).
\end{example}

When $\A$ is strict, \cref{elements-correspondence} recovers \cite[Theorem~4.14]{lambert2021discrete}; while the underlying equivalence of 2-categories recovers \cite[Theorem~6.13]{frohlich2024yoneda}. When $\A$ is a strictly normal \pdc, the underlying equivalence of categories recovers \cite[Theorem~3.45]{cruttwell2022double}. When $\A$ is a \pdc{}, the \ffness{} of the functor $\Lax(\A\op, \Span) \to \DFibDbl(\A)$ recovers \cite[Theorems~3.8 \& 3.11 \& Proposition~3.9]{pare2011yoneda}. Finally, when $\A$ is a bicategory, \cref{elements-correspondence} recovers an unpublished correspondence due to Cockett, Niefield and Wood~\cite{niefield2012double}, making use of the following observation.

\begin{lemma}
    \label{local-discrete-fibration}
    Let $\cl A$ be a bicategory, viewed as a tightly discrete \dc.
    A strict functor $E \colon \dc E \to \cl A$ of \pdcs{} is a discrete fibration if and only if the following conditions hold.
    \begin{enumerate}
        \item $\dc E$ is tightly discrete (\ie{} is a bicategory $\cl E$).
        \item $E$ is a local discrete fibration, \ie{} each functor $E_{X, Y} \colon \cl E(X, Y) \to \cl A(EX, EY)$ is a discrete fibration.
    \end{enumerate}
\end{lemma}

\begin{proof}
    For $E$ to be a discrete fibration amounts to the functors $E_0$ and $E_1$ being discrete fibrations (\cref{DFib-is-Mnd}). Since $\A_0$ is discrete, $E_0$ is a discrete fibration if and only if $\dc E_0$ is discrete, which is condition (1). Then, assuming (1), the property that $E_1$ is a discrete fibration is equivalent to condition (2).
\end{proof}

\begin{remark}
    We see from \cref{local-discrete-fibration} that a discrete fibration of strict \dcs{} is a generalisation of the notion of discrete 2-fibration~\cite[Definition~2.4]{lambert2024discrete}, which furthermore requires that, viewing $E$ as a 2-functor, the underlying functor of categories is a fibration. This refutes \cite[Remark~2.6 \& \S5.2]{lambert2024discrete}, where it is erroneously stated that `a discrete 2-fibration is not a kind of discrete double fibration'. That discrete fibrations of \dcs{} generalise discrete 2-fibrations is perhaps more easily seen from the indexed perspective: a discrete fibration of double categories over a 2-category corresponds to a contravariant normal lax functor into $\b{\cl Dist}$, whereas a discrete 2-fibration corresponds to a contravariant pseudo functor into $\Cat$~\cite[Theorem~3.7]{lambert2024discrete}. We note in passing that \citeauthor{lambert2024discrete}'s elements correspondence for split discrete 2-fibrations may consequently be recovered by restricting the left-hand side of the equivalence
    \[\LaxN(\A\opt, \Dist) \equiv \DFibDbl(\A) \tag{\cref{elements-correspondence}}\]
    to the strict functors $\A\opt \to \dc{Sq}(\Cat)$ into the strict \dc{} of natural transformations in the 2-category $\Cat$.
\end{remark}

\begin{remark}
    We note that the correspondence proven by \citeauthor{cruttwell2022double} permits the base $\A$ to vary. Our view is that it is the correspondence with the fixed base that is primary; we explain briefly how one might go about obtaining the correspondence for a varying base from the correspondence for a fixed base. \Cref{elements-correspondence} states that the 2-functors $\Lax({-}\op, \Span) \colon \Dbl_l \to \VDbln$ and $\DFibDbl({-}) \colon \Dbl_l \to \VDbln$ are 2-naturally equivalent. Our expectation is that such functors should admit an elements construction (necessarily of a three-dimensional nature). Thereafter, applying the elements construction to the equivalence $\Lax({-}\op, \Span) \equiv \DFibDbl({-})$ would produce the desired correspondence for a varying~base.
\end{remark}

We conclude this analysis with some simple but useful consequences of the correspondence between presheaves and discrete fibrations.

\begin{proposition}
    \label{Yoneda-embedding-is-slice-embedding}
    The embedding $\A({-}_2, {-}_1) \colon \A \to \Span^{\A\opt}$ corresponds under \cref{Span-A-op-equiv-DFib-A} to the functor $\A \to \LvDFib(\A)$ sending a loose morphism $p \colon A \lto B$ to the following span morphism.
    \[\begin{tikzcd}
        {\A_0/A} & {\A_1/p} & {\A_0/B} \\
        {\A_0} & {\A_1} & {\A_0}
        \arrow["{\pi_A}"', from=1-1, to=2-1]
        \arrow["{s/p}"', from=1-2, to=1-1]
        \arrow["{t/p}", from=1-2, to=1-3]
        \arrow["{\pi_p}"{description}, from=1-2, to=2-2]
        \arrow["{\pi_B}", from=1-3, to=2-3]
        \arrow["s", from=2-2, to=2-1]
        \arrow["t"', from=2-2, to=2-3]
    \end{tikzcd}\]
    Consequently, the embedding $Y_\A \colon \A \to \Lax(\A\op, \Span)$ corresponds under \cref{elements-correspondence} to the functor $\A \to \DFibDbl(\A)$ that sends each object $A \in \A$ to the projection functor from the slice \dc{} $\pi_A \colon \A/A \to \A$.
\end{proposition}

\begin{proof}
    The equivalence of \cref{Span-A-op-equiv-DFib-A} is levelwise, and so the results follow from the corresponding result for categories by unwrapping the definition of the embeddings in \cref{presheaf-embedding}.
\end{proof}

A discrete fibration $P \colon \b E \to \b A$ of categories corresponds to a representable presheaf on $\b A$ if and only if its domain admits a terminal object: in this case, $P$ is isomorphic to a projection functor $\pi_A \colon \b A/A \to \b A$, the identity morphism on the object $A$ being terminal in $\b A/A$. Applying this principle componentwise, we obtain an analogous representability condition for discrete fibrations of \pdcs.

\begin{corollary}
    \label{representability-of-discrete-fibration}
    A discrete fibration $\dc E \to \A$ of \pdcs{} corresponds under the equivalence of \cref{elements-correspondence} to a representable presheaf if and only if $\dc E_0$ admits a terminal object whose identity loose morphism is terminal in $\dc E_1$.
\end{corollary}

\begin{proof}
    Following \cref{Yoneda-embedding-is-slice-embedding}, the discrete fibration of \pdcs{} corresponding to the representable presheaf on an object $A \in \A$ is presented by the following span morphism. Trivially, $A$ is terminal in $\A_0/A$ and $A(1, 1)$ is terminal in $\A_1/A(1, 1)$.
    \begin{equation}
    \label{representable-discrete-fibration}
    \begin{tikzcd}[column sep=large]
        {\A_0/A} & {\A_1/A(1, 1)} & {\A_0/A} \\
        {\A_0} & {\A_1} & {\A_0}
        \arrow["{\pi_A}"', from=1-1, to=2-1]
        \arrow["{s/A(1, 1)}"', from=1-2, to=1-1]
        \arrow["{t/A(1, 1)}", from=1-2, to=1-3]
        \arrow["{\pi_{A(1, 1)}}"{description}, from=1-2, to=2-2]
        \arrow["{\pi_A}", from=1-3, to=2-3]
        \arrow["s", from=2-2, to=2-1]
        \arrow["t"', from=2-2, to=2-3]
    \end{tikzcd}
    \end{equation}

    Conversely, suppose $P \colon \dc E \to \A$ is a discrete fibration of \pdcs{} for which $\dc E_0$ admits a terminal object whose identity loose morphism is terminal in $\dc E_1$. Then the span morphism presenting $P$, whose components are discrete fibrations, is necessarily isomorphic to one of the form \eqref{representable-discrete-fibration}, by applying the characterisation of discrete fibrations of categories corresponding to representable presheaves componentwise.
\end{proof}

For instance, in combination with \cref{Yoneda-lemma}, \cref{Yoneda-embedding-is-slice-embedding} extends \cite[Theorem~7.4]{frohlich2024yoneda} from strict \dcs{} to \pdcs, while \cref{representability-of-discrete-fibration} extends \cite[Theorems~7.7 \& 7.9]{frohlich2024yoneda} likewise.

\begin{remark}
    In \cite[\S8]{fujii2025familial}, \citeauthor{fujii2025familial} study an elements construction for functors $\X \to \Span$ in which $\X$ is not assumed to be representable, and observe that this recovers \citeauthor{pare2011yoneda}'s elements construction when $\X$ is representable~\cite[Remark~8.8]{fujii2025familial}. It seems reasonable to expect that
    \begin{enumerate}
        \item this extends to an equivalence between the category of functors $\X \to \Span$ and the category of discrete opfibrations over $\X$ in the sense of \cite[Definition~5.1]{fujii2025familial};
        \item for representable $\X$, this recovers the underlying equivalence of categories between $\Lax(\X, \Span)_0$ and $\DFibDbl(\X\opt)_0$ established in \cref{elements-correspondence}.
    \end{enumerate}
    However, it seems unclear how one might recover the full equivalence of \emph{\vdcs} established in \cref{elements-correspondence} from a more general correspondence for arbitrary \vdcs{} $\X$, since the definition of $\Lax(\X, \Span)$ relies fundamentally on exponentiability.
\end{remark}

\subsection{Tight distributors and two-sided discrete fibrations}

For simplicity, we focused above on the elements correspondence between tight presheaves and discrete fibrations of \pdcs. However, the correspondence extends in exactly the same way to tight distributors and two-sided discrete fibrations of \pdcs{}. For completeness, we make explicit this more general correspondence. For an exposition of the correspondence relating distributors between categories and two-sided discrete fibrations, see \cite[Theorem~2.3.2]{loregian2020categorical}. We shall need a small lemma regarding stability of two-sided fibrations under pullback.

\begin{lemma}[{\cite[Proposition~7.4.5]{riehl2022elements}}]
    \label{pullback-stability-of-two-sided-fibrations}
    Let $D \colon \D \to \b A \times \b B$ be a two-sided discrete fibration and let $F \colon \b A' \to \b A$ and $G \colon \b B' \to \b B$ be functors. The pullback of $D$ along $F \times G \colon \b A' \times \b B' \to \b A \times \b B$ is a two-sided discrete fibration.
\end{lemma}

\begin{proof}
    $D$ is the projection from the category of elements for a distributor, say $P \colon \b B\op \times \b A \to \Set$. A short calculation shows that pulling back along $F \times G$ produces the projection from the category of elements for the distributor $P(F, G) \colon (\b B')\op \times \b A' \to \Set$.
\end{proof}

\begin{definition}
    Let $\A$ and $\B$ be \pdcs. We define a \vdc{} $\LvDFib(\A, \B)$ (for \emph{levelwise two-sided discrete fibrations}) as follows.
    \begin{itemize}
        \item The underlying category is the full subcategory of the slice category $\Cat/(\A_0 \times \B_0)$ spanned by the two-sided discrete fibrations from $\A_1$ to $\B_1$.
        \item A loose morphism from $D' \colon \D' \to \A_0 \times \B_0$ to $D \colon \D \to \A_0 \times \B_0$ is given by a two-sided discrete fibration $E \colon \E \to \A_1 \times \B_1$ from $\A_1$ to $\B_1$ and a span rendering the following diagram commutative.
        \[\begin{tikzcd}
            \D & \E & {\D'} \\
            {\A_0 \times \B_0} & {\A_1 \times \B_1} & {\A_0 \times \B_0}
            \arrow["D"', from=1-1, to=2-1]
            \arrow["{E_t}"', from=1-2, to=1-1]
            \arrow["{E_s}", from=1-2, to=1-3]
            \arrow["E"{description}, from=1-2, to=2-2]
            \arrow["{D'}", from=1-3, to=2-3]
            \arrow["{t \times t}", from=2-2, to=2-1]
            \arrow["{s \times s}"', from=2-2, to=2-3]
        \end{tikzcd}\]
        \item A cell $\zeta$ as on the left below (where each loose morphism $E_i$ on the left is given explicitly by the data of functors $E_i$, $(E_i)_s$ and $(E_i)_t$, as on the right below)
        \[
        \begin{tikzcd}
            {D_0} & \cdots & {D_n} \\
            D && {D'}
            \arrow[""{name=0, anchor=center, inner sep=0}, "H"', from=1-1, to=2-1]
            \arrow["{E_1}"'{inner sep=.8ex}, "\shortmid"{marking}, from=1-2, to=1-1]
            \arrow["{E_n}"'{inner sep=.8ex}, "\shortmid"{marking}, from=1-3, to=1-2]
            \arrow[""{name=1, anchor=center, inner sep=0}, "{H'}", from=1-3, to=2-3]
            \arrow["E"{inner sep=.8ex}, "\shortmid"{marking}, from=2-3, to=2-1]
            \arrow["\zeta"{description}, draw=none, from=1, to=0]
        \end{tikzcd}
        \hspace{3em}
        \begin{tikzcd}
            {\D_{i - 1}} & {\E_i} & {\D_i} \\
            {\A_0 \times \B_0} & {\A_1 \times \B_1} & {\A_0 \times \B_0}
            \arrow["{D_{i - 1}}"', from=1-1, to=2-1]
            \arrow["{(E_i)_t}"', from=1-2, to=1-1]
            \arrow["{(E_i)_s}", from=1-2, to=1-3]
            \arrow["{E_i}"{description}, from=1-2, to=2-2]
            \arrow["{D_i}", from=1-3, to=2-3]
            \arrow["{t \times t}", from=2-2, to=2-1]
            \arrow["{s \times s}"', from=2-2, to=2-3]
        \end{tikzcd}
        \]
        is given by a functor $\zeta$ making the following two diagrams commute.
        \[
        \begin{tikzcd}[column sep=large]
            {\E_{1, \ldots, n}} & \E \\
            {\A_n \times \B_n} & {\A_1 \times \B_1}
            \arrow["\zeta", from=1-1, to=1-2]
            \arrow["{E_{1, \ldots, n}}"', from=1-1, to=2-1]
            \arrow["E", from=1-2, to=2-2]
            \arrow["{\odot_n \times \odot_n}"', from=2-1, to=2-2]
        \end{tikzcd}
        \hspace{2em}
        \begin{tikzcd}[column sep=large]
            {\D_0} & {\E_1} & {\E_{1, \ldots, n}} & {\E_n} & {\D_n} \\
            \D && \E && {\D'}
            \arrow["H"', from=1-1, to=2-1]
            \arrow["{(E_1)_t}"', from=1-2, to=1-1]
            \arrow["{\pi_1}"', from=1-3, to=1-2]
            \arrow["{\pi_n}", from=1-3, to=1-4]
            \arrow["\zeta"{description}, from=1-3, to=2-3]
            \arrow["{(E_n)_s}", from=1-4, to=1-5]
            \arrow["{H'}", from=1-5, to=2-5]
            \arrow["{E_t}", from=2-3, to=2-1]
            \arrow["{E_s}"', from=2-3, to=2-5]
        \end{tikzcd}
        \]
        Above, we have denoted by $\E_{1, \ldots, n}$ the limit (\ie{} wide pullback) of the following chain of spans, and by $E_{1, \ldots, n} \colon \E_{1, \ldots, n} \to \A_n \times \B_n$ the canonical functor induced by the functors $E_1, \ldots, E_n$.
        \[\begin{tikzcd}
            & {\E_1} && {\E_2} && \cdots \\
            {\D_0} && {\D_1} && {\D_2} && {\D_n}
            \arrow["{(E_1)_t}"', from=1-2, to=2-1]
            \arrow["{(E_1)_s}"{description}, from=1-2, to=2-3]
            \arrow["{(E_2)_t}"{description}, from=1-4, to=2-3]
            \arrow["{(E_2)_s}"{description}, from=1-4, to=2-5]
            \arrow["{(E_3)_t}"{description}, from=1-6, to=2-5]
            \arrow["{(E_n)_s}", from=1-6, to=2-7]
        \end{tikzcd}\]
        \item The identity cell on $E \colon \E \to \A_1 \times \B_1$ is given by the identity functor on $\E$.
        \item Composition of cells is given, as in \cref{LvAmnIso}, by lifting the structural isomorphisms in $\A$ and $\B$.
        \[\begin{tikzcd}[column sep=huge]
            {\E_{1, \ldots, \sum_{1 \leq i \leq n} m_i}} & {\E_{1, \ldots, n}} & \E \\
            {\A_{\sum_{1 \leq i \leq n} m_i} \times \B_{\sum_{1 \leq i \leq n} m_i}} && {\A_1 \times \B_1}
            \arrow["{\zeta_1, \ldots, \zeta_n}", from=1-1, to=1-2]
            \arrow[""{name=0, anchor=center, inner sep=0}, curve={height=30pt}, dashed, from=1-1, to=1-3]
            \arrow["{E_{1, \ldots, \sum_{1 \leq i \leq n} m_i}}"', from=1-1, to=2-1]
            \arrow["\zeta", from=1-2, to=1-3]
            \arrow["E", from=1-3, to=2-3]
            \arrow["{\odot_{\sum_{1 \leq i \leq n} m_i} \times \odot_{\sum_{1 \leq i \leq n} m_i}}"', from=2-1, to=2-3]
            \arrow["\iso"{description}, draw=none, from=1-2, to=0]
        \end{tikzcd}\]
    \end{itemize}
    That composition of cells is indeed associative and unital follows from uniqueness of liftings for the two-sided discrete fibrations. $\LvDFib$ extends to a 2-functor $\Dbl_l\op \times \Dbl_l\coop \to \VDbl$ via pullback of categories (\cref{pullback-stability-of-two-sided-fibrations}).
\end{definition}

To our knowledge, two-sided discrete fibrations of \pdcs{} have not previously been studied. The one-sided analogue motivates the following definition (\cf~\cref{DFib-is-Mnd}).

\begin{definition}
    \label{double-two-sided-discrete-fibration}
    Let $\A$ and $\B$ be \pdcs. The \vdc{} of \emph{two-sided discrete fibrations} from $\A$ to $\B$ is $\DFibDbl(\A, \B) \defeq \Mod(\LvDFib(\A, \B))$.
\end{definition}

This definition is justified by an elements correspondence extending that for tight presheaves and discrete fibrations of \pdcs.

\begin{theorem}
    \label{tight-distributors-and-two-sided-discrete-fibrations}
    For each pair of \pdcs{} $\A$ and $\B$, there is an equivalence of \vdcs{}, 2-natural in $\A$ and $\B$:
    \[\Span^{\B\opt \times \A} \equiv \LvDFib(\A, \B)\]
    and consequently equivalences of normal \vdcs{}, 2-natural in $\A$ and $\B$:
    \[\Lax(\B\opt \times \A, \Span) \iso \LaxN(\B\opt \times \A, \Dist) \equiv \DFibDbl(\A, \B)\]
\end{theorem}

\begin{proof}
    The proofs of two statements follow exactly as in \cref{Span-A-op-equiv-DFib-A,elements-correspondence}, using the correspondence relating distributors between categories and two-sided discrete fibrations.
\end{proof}

For concreteness, we spell out the definition of the \pdc{} obtained from a lax functor $\B\opt \times \A \to \Span$ via \cref{tight-distributors-and-two-sided-discrete-fibrations}, generalising \cite[Theorem~3.8]{pare2011yoneda}.

\begin{definition}
    \label{double-category-of-elements}
    Let $P \colon \B\opt \times \A \to \Span$ be a lax functor of \pdcs. The \emph{\pdc{} of elements} $\El(P)$ associated to $P$ is defined as follows.
    \begin{enumerate}
        \item An object is a triple $(B \in \B, A \in \A, f \in P(B, A))$. We depict $f$ as a `tight heteromorphism' $f \colon B \squigto A$.
        \item A tight morphism from $(B, A, f)$ to $(B', A', f')$ is a pair $(b \colon B \to B', a \colon A \to A')$ such that the following square commutes (\ie{} the induced elements of $P(B, A')$ are equal).
        \[\begin{tikzcd}
            B & A \\
            {B'} & {A'}
            \arrow["f", squiggly, from=1-1, to=1-2]
            \arrow["b"', from=1-1, to=2-1]
            \arrow["a", from=1-2, to=2-2]
            \arrow["{f'}"', squiggly, from=2-1, to=2-2]
        \end{tikzcd}\]
        \item A loose morphism from $(B', A', f')$ to $(B, A, f)$ is a triple $(q \colon B' \lto B, p \colon A' \lto A, \phi \in P(q, p)_{f', f})$. We depict $\phi$ as a `hetero-cell'.
        \[\begin{tikzcd}
            B & {B'} \\
            A & {A'}
            \arrow[""{name=0, anchor=center, inner sep=0}, "f"', squiggly, from=1-1, to=2-1]
            \arrow["q"'{inner sep=.8ex}, "\shortmid"{marking}, from=1-2, to=1-1]
            \arrow[""{name=1, anchor=center, inner sep=0}, "{f'}", squiggly, from=1-2, to=2-2]
            \arrow["p"{inner sep=.8ex}, "\shortmid"{marking}, from=2-2, to=2-1]
            \arrow["\phi"{description}, draw=none, from=0, to=1]
        \end{tikzcd}\]
        \item A cell with the following frame,
        \[\begin{tikzcd}[column sep=huge]
            {(B_0, A_0, f_0)} & \cdots & {(B_n, A_n, f_n)} \\
            {(B, A, f)} && {(B', A', f')}
            \arrow["{(b, a)}"', from=1-1, to=2-1]
            \arrow["{(q_1, p_1, \phi_1)}"'{inner sep=.8ex}, "\shortmid"{marking}, from=1-2, to=1-1]
            \arrow["{(q_n, p_n, \phi_n)}"'{inner sep=.8ex}, "\shortmid"{marking}, from=1-3, to=1-2]
            \arrow["{(b', a')}", from=1-3, to=2-3]
            \arrow["{(q, p, \phi)}"{inner sep=.8ex}, "\shortmid"{marking}, from=2-3, to=2-1]
        \end{tikzcd}\]
        comprises cells $\alpha$ in $\A$ and $\beta$ in $\B$,
        \[
        \begin{tikzcd}[column sep=large]
            {B_0} & \cdots & {B_n} \\
            B && {B'}
            \arrow[""{name=0, anchor=center, inner sep=0}, "b"', from=1-1, to=2-1]
            \arrow["{q_1}"'{inner sep=.8ex}, "\shortmid"{marking}, from=1-2, to=1-1]
            \arrow["{q_n}"'{inner sep=.8ex}, "\shortmid"{marking}, from=1-3, to=1-2]
            \arrow[""{name=1, anchor=center, inner sep=0}, "{b'}", from=1-3, to=2-3]
            \arrow["q"{inner sep=.8ex}, "\shortmid"{marking}, from=2-3, to=2-1]
            \arrow["\beta"{description}, draw=none, from=0, to=1]
        \end{tikzcd}
        \hspace{3em}
        \begin{tikzcd}[column sep=large]
            {A_0} & \cdots & {A_n} \\
            A && {A'}
            \arrow[""{name=0, anchor=center, inner sep=0}, "a"', from=1-1, to=2-1]
            \arrow["{p_1}"'{inner sep=.8ex}, "\shortmid"{marking}, from=1-2, to=1-1]
            \arrow["{p_n}"'{inner sep=.8ex}, "\shortmid"{marking}, from=1-3, to=1-2]
            \arrow[""{name=1, anchor=center, inner sep=0}, "{a'}", from=1-3, to=2-3]
            \arrow["p"{inner sep=.8ex}, "\shortmid"{marking}, from=2-3, to=2-1]
            \arrow["\alpha"{description}, draw=none, from=0, to=1]
        \end{tikzcd}
        \]
        such that the following equation holds (\ie{} the induced elements of $P(q_1 \odot \cdots \odot q_n, p)$ are equal).
        \[
        \begin{tikzcd}
            {B_0} & \cdots & {B_n} \\
            B && {B'} \\
            A && {A'}
            \arrow[""{name=0, anchor=center, inner sep=0}, "b"', from=1-1, to=2-1]
            \arrow["{q_1}"'{inner sep=.8ex}, "\shortmid"{marking}, from=1-2, to=1-1]
            \arrow["{q_n}"'{inner sep=.8ex}, "\shortmid"{marking}, from=1-3, to=1-2]
            \arrow[""{name=1, anchor=center, inner sep=0}, "{b'}", from=1-3, to=2-3]
            \arrow[""{name=2, anchor=center, inner sep=0}, "f"', squiggly, from=2-1, to=3-1]
            \arrow["q"{description}, from=2-3, to=2-1]
            \arrow[""{name=3, anchor=center, inner sep=0}, "{f'}", squiggly, from=2-3, to=3-3]
            \arrow["p"{inner sep=.8ex}, "\shortmid"{marking}, from=3-3, to=3-1]
            \arrow["\beta"{description}, draw=none, from=0, to=1]
            \arrow["\phi"{description}, draw=none, from=2, to=3]
        \end{tikzcd}
        \quad = \quad
        \begin{tikzcd}
            {B_0} & \cdots & {B_n} \\
            {A_0} & \cdots & {A_n} \\
            A && {A'}
            \arrow[""{name=0, anchor=center, inner sep=0}, "{f_0}"', squiggly, from=1-1, to=2-1]
            \arrow["{q_1}"'{inner sep=.8ex}, "\shortmid"{marking}, from=1-2, to=1-1]
            \arrow[""{name=1, anchor=center, inner sep=0}, "\cdots"{description}, draw=none, from=1-2, to=2-2]
            \arrow["{q_n}"'{inner sep=.8ex}, "\shortmid"{marking}, from=1-3, to=1-2]
            \arrow[""{name=2, anchor=center, inner sep=0}, "{f_n}", squiggly, from=1-3, to=2-3]
            \arrow[""{name=3, anchor=center, inner sep=0}, "a"', from=2-1, to=3-1]
            \arrow["{p_1}"{description}, from=2-2, to=2-1]
            \arrow["{p_n}"{description}, from=2-3, to=2-2]
            \arrow[""{name=4, anchor=center, inner sep=0}, "{a'}", from=2-3, to=3-3]
            \arrow["p"{inner sep=.8ex}, "\shortmid"{marking}, from=3-3, to=3-1]
            \arrow["{\phi_1}"{description}, draw=none, from=0, to=1]
            \arrow["{\phi_n}"{description}, draw=none, from=1, to=2]
            \arrow["\alpha"{description}, draw=none, from=3, to=4]
        \end{tikzcd}
        \]
    \end{enumerate}
    Composition of tight morphisms and of cells is given by that in $\A$ and $\B$. Loose composites are given by those in $\A$, $\B$, and $\Span$. The \dc{} of elements is equipped with a span of strict functors of \dcs{} $\A \xfrom{\pi_1} \El(P) \xto{\pi_2} \B$ forming a two-sided discrete fibration in the sense of \cref{double-two-sided-discrete-fibration}.
\end{definition}

\begin{example}
    Every pseudo functor $\cl B\co \times \cl A \to \Cat$ induces a pseudo functor $\cl B\co \times \cl A \to \b{\cl Dist}$ by postcomposing the embedding $\Cat \to \b{\cl Dist}$ of functors into representable distributors, and hence induces a normal lax functor $\cl B\co \times \cl A \to \Span$ by \cref{Mod-is-right-adjoint}. In this case, the associated \pdc{} of elements (\cref{double-category-of-elements}) is tightly discrete (\ie{} a bicategory) and coincides with \citeauthor{street1980fibrations}'s two-sided Grothendieck construction for bicategories~\cite[(1.10)]{street1980fibrations}.
\end{example}

Note that the forgetful functor $\u\ph \colon \VDbl \to \Cat$, being a right adjoint, preserves discrete fibrations and so the underlying category of $\El(P)$ is the category of elements of the underlying distributor $\u P \colon \u\B\op \times \u\A \to \Set$.

We conclude by observing that there is a natural assignment sending lax functors between \pdcs{} to tight distributors, analogous to the embedding $\Cat \ffto \b{\cl Dist}$ of functors into representable distributors.

\begin{proposition}
    For each pair of \pdcs{} $\A$ and $\B$, there are \ff{} functors
    \begin{align*}
        \B^\A & \ffto \Span^{\B\opt \times \A} &
        \Lax(\A, \B) & \ffto \Lax(\B\opt \times \A, \Span)
    \end{align*}
    of \vdcs{} and normal \vdcs{} respectively.
\end{proposition}

\begin{proof}
    The two embeddings are respectively given by the following (using \cref{presheaf-embedding}).
    \[\B^\A \xffto{\B({-}_2, {-}_1)^\A} (\Span^{\B\opt})^\A \iso \Span^{\B\opt \times \A}\]
    \[\Lax(\A, \B) \xffto{\Lax(\A, Y_\B)} \Lax(\A, \Lax(\B\opt, \Span)) \iso \Lax(\B\opt \times \A, \Span) \qedhere\]
\end{proof}

Consequently, by taking tight opposites, we also obtain an analogue of the embedding ${\Cat\coop \ffto \b{\cl Dist}}$ of functors into corepresentable distributors.
\[(\B\opt)^{(\A\opt)} \ffto \Span^{\B \times \A\opt} \iso \Span^{\A\opt \times \B}\]
\[\Lax(\A\opt, \B\opt) \ffto \Lax(\B \times \A\opt, \Span) \iso \Lax(\A\opt \times \B, \Span)\]

(Given that, for ordinary categories $\b A$ and $\b B$, we have $(\b B\op)^{(\b A\op)} \iso (\b B^{\b A})\op$, we might be tempted to suggest that $(\B^\A)\opt \iso (\B\opt)^{(\A\opt)}$ and $\Lax(\A, \B)\opt \iso \Lax(\A\opt, \B\opt)$, but it is not clear to what extent the tight opposite makes sense for non-representable \vdcs.)

\printbibliography

\end{document}